\documentclass[a4paper]{gtart}

\usepackage[inner=20mm, outer=20mm, textheight=245mm]{geometry}

\usepackage[toc,page]{appendix}

\usepackage{psfrag, graphicx, subfigure, epsfig,color}

\usepackage{amsmath, amssymb, latexsym, euscript}

\usepackage[colorlinks,citecolor=blue,linkcolor=blue, backref=page]{hyperref}
\usepackage[nameinlink]{cleveref}

\usepackage[margin=1cm]{caption}

\usepackage{mathptmx}
\usepackage{wrapfig}
\usepackage[all,cmtip]{xy}
\usepackage{mathtools}
\usepackage{faktor}
\usepackage{tikz}
\usetikzlibrary{cd}
\usepackage{multicol}
\usepackage{listings}
\usepackage{stmaryrd} 

\usepackage[colorinlistoftodos]{todonotes}
\newcommand{\Mod}[1]{\ (\mathrm{mod}\ #1)}

\newcommand{\Z}{\mathbb{Z}}
\newcommand{\C}{\mathbb{C}}
\newcommand{\F}{\mathbb{F}}
\newcommand{\K}{\mathbb{K}}

\newcommand{\N}{\mathbb{N}}

\DeclareMathOperator{\im}{im}

\DeclareMathOperator{\GL}{GL}
\DeclareMathOperator{\SL}{SL}
\DeclareMathOperator{\PSL}{PSL}
\DeclareMathOperator{\ltrn}{Tr_n}
\DeclareMathOperator{\strn}{tr_n}
\DeclareMathOperator{\qn}{q_n}

\DeclareMathOperator{\tr}{tr}

\DeclareMathOperator{\Hom}{Hom}

\DeclareMathOperator{\SLC}{\SL_2(\C)}
\DeclareMathOperator{\SLF}{\SL_2(\F)}
\DeclareMathOperator{\SLFt}{\SL_2(\F_2)}
\DeclareMathOperator{\SLnC}{\SL_n(\C)}

\DeclareMathOperator{\PSLF}{\PSL_2(\F)}
\DeclareMathOperator{\Xred}{X^{red}}
\DeclareMathOperator{\Xirr}{X^{irr}}
\DeclareMathOperator{\R}{R}
\DeclareMathOperator{\X}{X}
\DeclareMathOperator{\D}{D_{\infty}}
\def\Dn{\text{D}_{\infty}^n}

\DeclareMathOperator{\newt}{Newt}
\DeclareMathOperator{\sphD}{Sph}
\DeclareMathOperator{\multi}{multideg}
\DeclareMathOperator{\LM}{LM}
\DeclareMathOperator{\LT}{LT}
\DeclareMathOperator{\LC}{LC}

\def\ii{i_0}

\def\M{\mathcal{M}}
\def\L{\mathcal{L}}

\def\P{\mathcal{P}}
\def\O{\mathcal{O}}
\def\V{\mathcal{V}}
\def\E{\mathcal{E}}

\def\idMat{E}

\def\tree{\mathbf{T}}
\def\trFunc{I}

\def\Dih{\text{Dih}(\Z^n)}

\def\address#1{{\let\newline\par\xdef\previousaddresses{\theaddress}}
 \ifx\theaddress\relax\def\theaddress{#1}\else
 \def\theaddress{\previousaddresses\par\vskip 2pt\par#1}\fi}
\def\secondaddress#1{{\let\newline\par\xdef\previousaddresses{\theaddress}}
 \ifx\theaddress\relax\def\theaddress{#1}\else
 \def\theaddress{\previousaddresses\par{\rm and}\par#1}\fi}   
 \def\Addresses{\bigskip
{\small \parskip 0pt \leftskip 0pt \rightskip 0pt plus 1fil  \def\\{\par}
\sl\theaddress\par\medskip \rm} }
\let\theaddress\relax\let\theemail\relax\let\theurl\relax

\newcommand{\bound}{\partial}
\newcommand{\pie}{\pi_1}

\theoremstyle{plain}
\newtheorem{theorem}{Theorem}
\newtheorem*{theorem*}{Theorem}
\newtheorem{lemma}[theorem]{Lemma}
\newtheorem{proposition}[theorem]{Proposition}
\newtheorem{corollary}[theorem]{Corollary}

\theoremstyle{definition}
\newtheorem{definition}[theorem]{Definition}
\newtheorem*{definition*}{Definition}
\newtheorem{example}[theorem]{Example}

\newtheorem{question}[theorem]{Question}

\theoremstyle{remark}
\newtheorem{remark}[theorem]{Remark}

\numberwithin{equation}{section}

\usepackage{enumitem}
\newlist{enumRoman}{enumerate}{1}
\setlist[enumRoman]{label=(\Roman*)}

\begin{document}

\title{An invitation to Culler-Shalen theory in arbitrary characteristic}
\author{Grace S. Garden and Stephan Tillmann}

\begin{abstract} 
In seminal work of Culler and Shalen from 1983, essential surfaces in 3--manifolds are associated to ideal points of their $\SLC$--character varieties, and connections between the algebraic geometry of the character variety and the topology of the 3--manifold are established via group actions on trees. Here, we lay a general foundation for this theory in arbitrary characteristic by using the same approach instead over an arbitrary algebraically closed field. Examples include a change in the $A$--polynomial in characteristic 2, a closed Haken hyperbolic 3--manifold with no detected essential surface in any characteristic, and a closed Haken hyperbolic 3--manifold with an essential surface only detected in characteristic 2.
\end{abstract}



\makeshorttitle


\section{Introduction}
\label{sec:intro}

The ground-breaking work of Culler and Shalen \cite{Culler-Shalen-varieties-1983} associates essential surfaces in a 3--manifold to ideal points of curves in its $\SLC$--character variety.
The method underscores connections between the theory of incompressible surfaces in 3--manifolds, the geometry of character varieties, and group actions on trees. 
Character variety techniques have been incredibly influential in knot theory and the study of 3--manifolds. They led to the definitions of the $A$--polynomial \cite{Cooper-Plane-1994} and the Culler-Shalen norm \cite{Culler-Dehn-1987}, and were used in the proofs of numerous well-known theorems, including the weak Neuwirth conjecture \cite{Culler-bounded-1984}, the Smith conjecture \cite{Shalen-ThreeManifold-1999}, the cyclic surgery theorem \cite{Culler-Dehn-1987} and the finite surgery theorem~\cite{Boyer-finite-2001}.

In this paper, we lay the foundation for the theory over any algebraically closed field $\F$ of arbitrary characteristic and hope to convince the reader through results and examples that this is a rich theory worth investigating further. Paoluzzi and Porti ~\cite{Paoluzzi-nonstandard-2013, Paoluzzi-invariant-2018, Paoluzzi-examples-2020} have similarly studied the character varieties of knot complements in odd characteristic and found discrepancies across characteristics (for example, in the number of components, dimension, intersection points). It remains a tantalising problem to link such ramification phenomena to geometric, algebraic or topological properties of knots or 3--manifolds.

The contents of this paper is as follows:

\Cref{sec:varieties} \textbf{(The representation variety and the variety of characters)}
This section establishes some results on $\SLF$--representation varieties of finitely generated groups in arbitrary characteristic. Let $\Gamma$ be a finitely generated group. The \textbf{$\SLF$--representation variety} of $\Gamma$ is the space of homeomorphisms $\rho\co \Gamma \to \SLF$,
\begin{equation*}
	\R(\Gamma,\F) = \Hom(\Gamma,\SLF)
\end{equation*}
Each element $\rho$ is called a representation. 

We imbue $\R(\Gamma,\F)$ with the structure of an algebraic set.
For the constructions in Culler-Shalen theory, we require interesting curves in $\R(\Gamma,\F)$. More specifically, we require 1--dimensional irreducible subvarieties that meet each conjugacy class of representations in at most finitely many points. One approach is to work with the \emph{character variety} arising from geometric invariant theory, another is to work with the \emph{variety of characters}. This paper follows the latter route by first introducing a set-up and definition that is not standard, but motivated by our intention to formulate and apply Culler-Shalen theory.
 
The group $\SLF$ acts on $\R(\Gamma,\F)$ by conjugation and the action is algebraic. We say that $\rho, \sigma \in \R(\Gamma,\F)$ are \textbf{closure-equivalent} if the Zariski closures of their orbits intersect. The quotient space of this equivalence relation $\sim_c$ 
\[
	\Hom(\Gamma,\SLF) / \sim_c
\]
is a topological space, which we imbue with the extra structure of an affine variety by showing that the equivalence relation $\sim_c$ is precisely the equivalence relation induced by the trace functions. We thus turn the quotient space into the  \textbf{variety of $\SLF$--characters} of $\Gamma,$ written $\X(\Gamma,\F).$ Hence the set of points of $\X(\Gamma,\F)$ is 
$\Hom(\Gamma,\SLF) / \sim_c$ and its coordinate ring is naturally generated by the trace functions.

\begin{remark}
\label{rem:vars}
There is a related construction giving the affine general invariant theory quotient of the representation variety. 
Here, all conjugation invariant functions on $\R(\Gamma,\F)$ are considered. Following \cite{Martin-reductive-2003} denote the GIT quotient by 
\[
C(\Gamma, \F) = \Hom(\Gamma,\SLF) \sslash \SLF
\]
and call this the \textbf{$\SLF$--character variety}.
The character variety $C(\Gamma, \F)$ has coordinate ring the ring of all $\SL_2$--invariant regular functions on $\Hom(\Gamma,\SLF)$. 
The variety of characters  $\X(\Gamma,\F)$ has coordinate ring the ring of trace functions. 
The inclusion of the ring of trace functions in the full ring of invariants gives a map $C(\Gamma, \F) \to \X(\Gamma, \F)$.

It is classical that for the complex numbers, these two constructions define the same object in terms of points and structure~\cite{Procesi-invariant-1976}. That is, there is a natural identification $C(\Gamma, \C) = \X(\Gamma, \C)$ as affine varieties. 
We thank Benjamin Martin for pointing out to us that $C(\Gamma, \F) = \X(\Gamma, \F)$ also holds when $\F$ has characteristic zero or when $\Gamma$ is the free group; see \cite{Martin-reductive-2003}. 
However, in positive characteristic and $\Gamma$ an arbitrary finitely generated group, it is an interesting question to determine whether the map $C(\Gamma, \F) \to \X(\Gamma, \F)$ is an isomorphism of varieties. See also \cite{Lawton-varieties-2017} for related discussion and examples.
\end{remark}

Many of the results in \Cref{sec:varieties} may be well-known or are direct generalisations from the theory over the complex numbers. Of particular interest is characteristic $2$, as some of the standard trace identities that hold over characteristic $0$ change. In particular, the trace ring is generated by the trace functions of single and ordered double and triple products of generators unless the characteristic is 2 (\Cref{cor:ordered_up_to_three}). We show that these trace functions nevertheless give a natural coordinate system for the variety of characters in any characteristic (\Cref{thm:character_variety_is_characters}). We also determine a birational inverse from this parameterisation of the variety of characters to the trace ring in characteristic 2. 
An overview of notation and main results of this section is given in \Cref{sec:def_varieties}.
Another technical result that may be new is a finite set of equations defining the reducible characters of the free group (\Cref{cor:redChar}).

\Cref{sec:Splittings} \textbf{(Splittings of groups detected by ideal points)} The approach used by Culler and Shalen \cite{Culler-Shalen-varieties-1983} to define ideal points of curves in the variety of characters  applies verbatim for fields of positive characteristic, and so does the application of Bass-Serre theory~\cite{Bass-covering-1993, Serre-trees-2003} of group actions on trees to produce splittings of groups. The relevant definitions and results are summarised in this section.

\Cref{sec:CullerShalen} \textbf{(Essential surfaces in 3--manifolds detected by ideal points)} The approach used by Culler and Shalen \cite{Culler-Shalen-varieties-1983} to associate essential surfaces to ideal points of curves in the $\SLC$--character variety also applies almost verbatim to the variety of $\SLF$--characters in our set-up. In particular, the key results from \cite{Boyer-finite-2001, Cooper-Plane-1994, Culler-Dehn-1987, Culler-Shalen-varieties-1983, Culler-bounded-1984} linking valuations of trace functions to topological properties of associated essential surfaces are at our disposal. We describe the basic ingredients, and also include a well-known bound on the dimension of the variety of characters of a 3--manifold that goes back to Thurston~\cite{Thurston-notes}. 

\Cref{sec:APoly} \textbf{($A$--polynomials and eigenvalue varieties)} Assume $M$ is a compact 3--manifold with a torus boundary component. We show the definition of the $A$--polynomial still holds and prove a version of the \emph{boundary slopes} theorem in characteristic $p$, following the same approach as in \cite{Cooper-Plane-1994}. The standard $A$--polynomial over $\C$ and the $A$--polynomial in characteristic $p$ can be related. Comparing across characteristics, we prove that for all but finitely many primes $p$, a slope on the boundary of a knot manifold is strongly detected in characteristic $p$ if and only if it is strongly detected over $\C$. The excluded primes are those that divide leading terms in the algorithm for computing a Gr\"obner basis that is used to calculate the $A$--polynomial. We also describe the natural generalisation from \cite{Tillus-boundary-2005} to eigenvalue varieties of manifolds with more than one torus boundary component.

\Cref{examples} \textbf{(A trip to the zoo)} The last section gives some applications and examples. We first discuss free products of groups and detecting prime decomposition (\Cref{subsec:direct_and_prime}). We then include three extended examples with different behaviour. The first is a change of the $A$--polynomial (\Cref{sec:m137}) in characteristic 2. The second and third examples are closed, Haken, hyperbolic 3--manifolds. One has a decrease of the dimension of the variety of characters in characteristic 2 and no essential surface is detected in any characteristic (\Cref{sec:m019}). The other has an increase of the dimension of the variety of characters in characteristic 2 and in particular essential surfaces are only detected by ideal points in characteristic 2 (\Cref{sec:m188}).

For the remainder of this paper, let $\F$ be an algebraically closed field of arbitrary characteristic. We write $\Z_p = \Z/p\Z$ and sometimes use $\F_p$ for an arbitrary algebraically closed field of characteristic $p$. 
In general, we may denote terms calculated in a specific characteristic with a subscript $p$. 
We use $\idMat$ to refer to the $2\times2$ identity matrix, and $A^\intercal$ to denote the transpose of $A.$

\textbf{Acknowledgements.} 
Research of the first author is supported by an Australian Government Research Training Program scholarship.
Research of the second author is supported in part under the Australian Research Council's ARC Future Fellowship FT170100316. The second author thanks Benjamin Martin for stimulating conversations in Brisbane in 2011. The authors also thank Eric Chesebro, Benjamin Martin, and Daniel Mathews for helpful comments on an earlier draft of the manuscript, and Yue Ren for helpful mathematical discussions. 


\section{The representation variety and the variety of characters}
\label{sec:varieties}

We start by discussing the $\SLF$--representation variety and the variety of  $\SLF$--characters for $\F$ an algebraically closed field. We recall their definitions, provide some results for traces, and use this to describe coordinates for the variety of characters. This includes both classic results for traces and new results for traces in characteristic $2$. 
We end with a useful result on normal forms and conjugacy classes.


\subsection{Definitions}
\label{sec:def_varieties}

Let $\F$ be an algebraically closed field and $\Gamma$ be a finitely generated group. 
The \textbf{$\SLF$--representation variety} of $\Gamma$ is the space of representations $\rho\co \Gamma \to \SLF$,
\begin{equation}
	\R(\Gamma,\F) = \Hom(\Gamma,\SLF)
\end{equation}
Given a finite, ordered generating set $( \gamma_1, \ldots, \gamma_n)$ of $\Gamma,$ the Hilbert basis theorem implies that we can imbue $\R(\Gamma,\F)$ with the structure of an affine algebraic subset of $\SLF^n \subset \F^{4n}$ via the natural embedding
\begin{align*}
	\varphi\co \R(\Gamma,\F) &\hookrightarrow \SLF^n \subset \F^{4n}\\
	\rho &\mapsto (\rho(\gamma_1), \ldots, \rho(\gamma_n))
\end{align*}
The Zariski topology on $\F^{4n}$ induces a topology on $\R(\Gamma,\F).$ Choosing a different ordered generating set $(\gamma'_1, \ldots, \gamma'_m)$ results in another embedding 
\[
	\varphi'\co \R(\Gamma,\F) \hookrightarrow \SLF^m \subset \F^{4m}
\]
There is a natural polynomial isomorphism between the images of $\varphi$ and $\varphi'$ obtained by expressing the generators of one generating set in terms of the other and vice versa. In particular, the topology on $\R(\Gamma,\F)$ induced by the Zariski topology is independent of the choice of generating set. 

The group $\SLF$ acts by conjugation on $\SLF^n.$ The action preserves the natural embedding of $\R(\Gamma,\F)$ and is algebraic. We say that $\rho, \sigma \in \R(\Gamma,\F)$ are \textbf{closure-equivalent} if the Zariski closures of their orbits intersect. We denote the quotient space of this equivalence relation $\sim_c$ by
\[
	\X(\Gamma,\F) = \Hom(\Gamma,\SLF) / \sim_c
\]
We denote the natural quotient map $q\co \R(\Gamma,\F) \to \X(\Gamma,\F).$ We show in \Cref{thm:character_variety_is_characters} and \Cref{cor:trace_function_characterisation} that $\X(\Gamma,\F)$ can be imbued with the structure of an affine algebraic set with coordinate ring the ring of trace functions. This justifies calling $\X(\Gamma,\F)$  the
 \textbf{variety of $\SLF$--characters} of $\Gamma,$ which we start using from now onwards.
 
A key observation in \Cref{sec:coordinates_for_X} is that the variety of $\SLF$--characters of $\Gamma$ has a natural set of coordinates, given by taking traces of ordered products of generators. 
Thus, the embedding $\varphi\co\R(\Gamma,\F) \hookrightarrow \SLF^n\subset \F^{4n}$ induces a natural embedding ${\varphi}_n\co\X(\Gamma,\F) \hookrightarrow \F^{2^n-1}.$ We call this the \textbf{long trace coordinates} of the variety of characters,
\[
	\ltrn\co \varphi(\R(\Gamma,\F)) \to {\varphi}_n(\X(\Gamma,\F))
\]
In particular, the Zariski topology on $\X(\Gamma,\F)$ from the embedding in $\F^{2^n-1}$ coincides with the quotient topology inherited from $\R(\Gamma,\F)$ and we have the following commutative diagram:
\[
	\xymatrix{
	\R(\Gamma,\F) \ar[d]^{}_{\varphi} \ar[r]^{q}_{} 	& \X(\Gamma,\F) \ar[d]^{}_{{\varphi}_n}	\\
	\SLF^{n} \ar[r]^{\ltrn}_{} 		& \F^{2^n-1}}		
\]
The \textbf{character of a representation $\rho$} is the trace function
\begin{equation}
\begin{split}
	\tau_\rho\co \Gamma &\to \F \\
	\gamma&\mapsto\tr(\rho(\gamma))
\end{split}	
\end{equation}
Conjugate representations have the same trace function, and since the trace functions are algebraic, it follows that closure-equivalent representations have the same trace function. In \Cref{cor:trace_function_characterisation}, we show that 
\begin{align*}
					& \tau_\rho = \tau_\sigma \\
	\Leftrightarrow\quad  &  \text{$\tau_\rho$ and $\tau_\sigma$ agree on all ordered products of generators}\\
	\Leftrightarrow\quad  &  \text{$\rho$ and $\sigma$ are closure-equivalent}.
\end{align*}
Hence the variety of characters is naturally identified with the space of characters of representations.
Classical work of Fricke and Vogt (\Cref{lem:trace_generators}) implies that, given two different embeddings $\X(\Gamma,\F) \hookrightarrow \F^{2^n-1}$ and $\X(\Gamma,\F) \hookrightarrow \F^{2^m-1}$ arising from different generating sets of $\Gamma$ as above, there is a polynomial isomorphism with integral coefficients taking the image of one embedding one to the other. (Throughout, we call the additive group generated by $1\in \F$ the integers in $\F$.) Hence any two long trace coordinates are naturally isomorphic.

In fact, we prove the following stronger statements in \Cref{sec:coordinates_for_X}. Compose $\X(\Gamma,\F) \hookrightarrow \F^{2^n-1}$ with the projection onto the $d = n+\binom{n}{2}+\binom{n}{3}$ coordinates that correspond to the traces of the images of all generators, and of all ordered double and triple products of generators. We call this the \textbf{short trace coordinates} of the variety of characters.
We show that this composition gives an embedding ${\varphi}_d\co\X(\Gamma,\F) \hookrightarrow \F^d$
and hence our commutative diagram extends to:
\[
	\xymatrix{
	\R(\Gamma,\F) \ar[d]^{}_{\varphi} \ar[r]^{q}_{} 	& \X(\Gamma,\F)  \ar[d]^{}_{{\varphi}_n} \ar[rd]^{\varphi_d}_{} & 	\\
	\SLF^{n} \ar[r]^{\ltrn}_{} 		& \F^{2^n-1} \ar[r]^{r}_{}  & \F^{d}}		
\]
We write
\[
	\strn = r\circ \ltrn \co \varphi(\R(\Gamma,\F)) \to {\varphi}_d(\X(\Gamma,\F))
\]

\begin{remark}
\label{rem:inverseCharP}
	In characteristic $p\neq2$, the restriction of the projection map $\varphi_n(\X(\Gamma,\F))$ $\to $ $\varphi_d(\X(\Gamma,\F))$ has a polynomial inverse. This arises from the trace identities in \Cref{sec:traceRing}. In particular, there is a polynomial automorphism between the short trace coordinates of the variety of characters arising from different ordered generating sets.
\end{remark}

\begin{remark}
\label{rem:inverseChar2}
	In characteristic 2, for each Zariski component $Z \subseteq \X(\Gamma,\F),$ the projection map has a birational inverse taking ${\varphi}_d(Z)$ to ${\varphi}_n(Z).$ See \Cref{sec:traces_char_2}.
\end{remark}

For a manifold $M$ with finitely generated fundamental group, we take $\Gamma=\pi_1(M)$. The representation varieties are written $\R(M,\F)=\R(\pi_1(M),\F)$ and the varieties of characters are written $\X(M,\F)=\X(\pi_1(M),\F)$. 


\subsection{Classical results on traces}
\label{sec:traceRing}

We focus on trace identities that will be used to later describe the variety of characters in \Cref{sec:coordinates_for_X}. 
Much of this section outlines classical work that goes back to Vogt~\cite{Vogt-invariants-1889} and Fricke~\cite{Fricke1896}. 
The main purpose is to highlight the aspects that apply in arbitrary characteristic and use this to determine the short trace coordinates in arbitrary characteristic. This will give a polynomial inverse to ${\varphi}_n(\X(\Gamma,\F)) \to {\varphi}_d(\X(\Gamma,\F))$ in characteristic $p\neq2$ as discussed in \Cref{rem:inverseCharP}.

We follow the preliminary discussion in \cite[\S 4]{Gonzalez-character-1993}; see also \cite{Horowitz-characters-1972} and \cite{Magnus-rings-1980}. We have not seen the identity in \Cref{eq:trace_four_char2} elsewhere in the literature.

For $A,B,C\in\SLF$ with $\F$ an algebraically closed field of characteristic $p$ we have the following trace identities
\begin{align}
\label{eq:tr_single_inverse}	\tr\left(A^{-1}\right)		&=\tr\left(A\right)\\
\label{eq:tr_conjugate}		\tr\left(BAB^{-1}\right)	&=\tr\left(A\right)\\
\label{eq:tr_invert}			\tr\left(A)\tr(B\right)		&=\tr\left(AB\right)+\tr\left(AB^{-1}\right)
\end{align}
We often apply \Cref{eq:tr_conjugate} in the equivalent form $\tr(AB)=\tr(BA).$
We also have
\begin{equation}\label{eq:tr_permute_three}
\begin{split}
	\tr(ACB)=	&\tr(A)\tr(BC)+\tr(B)\tr(AC)+\tr(C)\tr(AB) \\ 
			&-\tr(A)\tr(B)\tr(C)-\tr(ABC)
\end{split}
\end{equation}
These can all be derived from the Cayley-Hamilton theorem.

Applying \Cref{eq:tr_permute_three} with the convention $[\;A,\;B\;] = ABA^{-1}B^{-1}$, we have for all $A, B \in \SLF,$
\begin{equation}
\label{eq:commutator_trace}
	\tr[\;A,\;B\;] = (\tr (A))^2 + (\tr (B))^2 + (\tr (AB))^2 - \tr(A)\tr(B)\tr (AB) -2
\end{equation}
Then for all $A, B, C \in \SLF,$ we have
\begin{equation}
\label{eq:commutator_trace_three}
\begin{split}
	\tr[\;A,\;BC\;] = 	& (\tr (A))^2 + (\tr (BC))^2 + (\tr (ABC))^2 \\
				&- \tr(A)\tr (BC)\tr (ABC) -2
\end{split}
\end{equation}
Suppose we are given a word $W$ of length $L$ in the matrices $A_1,\ldots, A_n\in \SLF$ and their inverses. We regard the identity element as a word of length zero. The trace identities can be used to write $\tr(W)$ as a polynomial with integral coefficients in the traces of words of length at most $L$ that have the property that each word contains each letter $A_i$ at most once and with the $A_i$ respecting the natural cyclic order of the index set. This is recorded in the below result.

\begin{lemma}[Vogt, Fricke]
\label{lem:trace_generators}
Let $\F$ be an algebraically closed field. The trace of any word in the matrices $A_1,\ldots, A_n\in\SLF$ is a polynomial with integral coefficients in the $2^n-1$ traces
\[
	\tr (A_{j_1}A_{j_2} \ldots A_{j_m}), \quad 1\le j_1<j_2 < \ldots < j_m \le n, \quad m \le n
\]
Moreover, the constant term of the polynomial is an integral multiple of the trace of the identity matrix.
\end{lemma}

An argument in the spirit of~\cite{Culler-Shalen-varieties-1983} that implies finite generation of the trace ring without an explicit generating set is~\cite[Theorem 1.4]{Martin-reductive-2003}. Here, we give an elementary argument to prove \Cref{lem:trace_generators} using the trace identities.

\begin{proof}[Sketch of proof]
A method to determine the polynomial claimed in the statement proceeds as follows.
Using \Cref{eq:tr_conjugate} in the equivalent form $\tr(BA)=\tr(AB)$ and induction, we see that we may cyclically permute the argument in the trace function.

Using \Cref{eq:tr_permute_three}, the previous observation regarding cyclic permutation, and induction, we may write the trace of $W$ as a polynomial with integral coefficients in traces of words of length at most $L$ that have the property that they are of the form $A_1^{m_1}\cdots A_n^{m_n}$ with $m_i \in \Z.$ Note that $\sum |m_i| \le L.$

Using \Cref{eq:tr_invert} in the equivalent form $\tr\left(AB^{-1}\right) = \tr(A)\tr(B) - \tr(AB)$ and cyclic permutation, we may now arrange for all exponents to be positive, at the expense of increasing the number of terms in the polynomial expression for $\tr(W)$ but maintaining the property that each term has the required cyclic order.

Now \Cref{eq:tr_invert} also implies 
\begin{equation}
\label{eq:tracePowers}
	\tr(A^{m+1}C) = \tr(A)\tr(A^mC) - \tr(A^{m-1}C), \quad m\ge 1
\end{equation}
and hence we may inductively transform the words in the terms of the polynomial to contain each letter at most once, again maintaining the order property.
This gives the desired result.
\end{proof}

Applying \Cref{eq:tr_permute_three} three times to different partitions of the product $ABCD$ as in \cite[Lemma 4.1.1]{Gonzalez-character-1993} implies the following:
\begin{align}\label{eq:tr_four_term}
	 \notag 2 \tr (ABCD) = & \tr(AC) \tr B \tr D - \tr A \tr B \tr (CD) - \tr B \tr C \tr (DA)  \\
					  & - \tr(AC) \tr (BD) + \tr(AB) \tr (CD) + \tr(BC) \tr (DA) \\
			        \notag & - \tr(ACB) \tr D + \tr A \tr BCD + \tr B \tr CDA  + \tr C \tr DAB
\end{align}
\Cref{eq:tr_four_term} involves the term $\tr(ACB)$, which does not respect the natural cyclic order but can be substituted using \Cref{eq:tr_permute_three}.
Hence if $p\neq 2$, one may use induction on \Cref{eq:tr_four_term} and the above observations to write the trace of any word $W$ of length $L\ge 4$ in the matrices $A_1,\ldots, A_n$ and their inverses as a polynomial in the traces of the $A_i$ and their cyclically ordered double and triple products. 
Moreover, the constant term of the polynomial is an integral multiple of 2. 
\begin{corollary}
\label{cor:ordered_up_to_three}
Let $\F$ be an algebraically closed field of characteristic $p\neq2$. 
The trace of any word in the matrices $A_1,\ldots, A_n\in\SLF$ and their inverses is a polynomial in the $d=n+\binom{n}{2}+\binom{n}{3}$ traces
\[
	\tr (A_{j_1}\ldots A_{j_m}), \quad 1\le j_1<j_2 < \ldots < j_m\le n, \quad m \le 3
\]
\end{corollary}
The corollary constructs the polynomial inverse for $\varphi_n(\X(\Gamma,\F))\to \varphi_d(\X(\Gamma,\F))$ mentioned in \Cref{rem:inverseCharP}.

In contrast, in characteristic 2 the left hand side in \Cref{eq:tr_four_term} vanishes and we obtain a special equation satisfied by the traces involving four matrices. We collect this and other useful identities that hold in this characteristic in \Cref{sec:traces_char_2}.


\subsection{Reducible and irreducible representations}
\label{sec:coordinates_for_X}

From now on, identify $\R(\Gamma,\F)$ with its image in $\SLF^n.$
Since there is a natural epimorphism from the free group $F_n$ in $n$ letters to $\Gamma$ and $\R(\Gamma,\F)\subset \R(F_n,\F) = \SLF^n$ is an algebraic set consisting of conjugacy classes of representations of the free group in $n$ letters, it suffices to analyse the varieties over the free group. 

We say that a subgroup of $\SLF$ is \textbf{reducible} if its action on $\F^2$ has an invariant 1--dimensional subspace. Otherwise it is \textbf{irreducible}. 
We call a \emph{representation} (or equivalently an $n$--tuple of matrices) \textbf{irreducible} (resp. \textbf{reducible}) if it generates an {irreducible} (resp. {reducible}) subgroup of $\SLF$. This property is invariant under conjugation, and hence inherited by conjugacy classes.

In particular, all abelian subgroups of $\SLF$ are reducible. We call a representation (or equivalently an $n$--tuple of matrices) \textbf{abelian} if it generates an {abelian} subgroup of $\SLF.$

We know the representation variety of $\Gamma$ forms an algebraic set. For the free group $F_n$, view $\SLF^n \subset \F^{4n}$ and use the \emph{Zariski topology} on $\F^{4n}.$ 
We embed $\SLF^n \subset \F^{4n}$ by setting
\[
	A_k = \begin{pmatrix}  w_k & x_k \\ y_k & z_k \end{pmatrix}.
\]
Then using coordinates $(w_1, x_1, y_1, z_1, \ldots, w_n, x_n, y_n, z_n)$ for $\F^{4n},$ $\SLF^n$ is a Zariski closed subset of $\F^{4n}$ with coordinates defined by the equations 
$w_k z_k-x_ky_k=1$ for each $1\le k \le n.$

We collect some useful results here about reducible and irreducible representations and later characters. 
The results are classical or {are due to \cite{Culler-Shalen-varieties-1983}}, and hold in any characteristic.

\begin{remark}
Since we are interested in representations up to conjugacy, we can often conjugate pairs of matrices to be in their simplest form. 

If $A, B \in\SLF$ form a reducible pair, then we may conjugate so that 
\begin{equation}
\label{eqn:redForm}
	A=\begin{pmatrix} s & \ast \\ 0 & s^{-1} \end{pmatrix} \quad\text{and}\quad
	B=\begin{pmatrix}  t & \ast \\ 0 & t^{-1} \end{pmatrix}, \quad s,t \in\F\setminus\{0\}.
\end{equation}

If $A, B \in\SLF$ form an irreducible pair, then we may conjugate so that 
\begin{equation}
\label{eqn:irredForm}
	A=\begin{pmatrix} s & 1 \\ 0 & s^{-1} \end{pmatrix} \quad\text{and}\quad
	B=\begin{pmatrix}  t & 0 \\ u & t^{-1} \end{pmatrix}, \quad s,t,u \in\F\setminus\{0\}
\end{equation}
where $u\neq-\left(s-s^{-1}\right)\left(t-t^{-1}\right)$.
This is useful for calculations.

Note the requirement $u \neq -\left(s-s^{-1}\right)\left(t-t^{-1}\right)$ in \Cref{eqn:irredForm} is because the invariant subspaces of $A$ are spanned by $(1,0)^\intercal$ and $\left(1, s^{-1}-s\right)^\intercal$ and those of $B$ by $(0,1)^\intercal$ and $\left(t-t^{-1},u\right)^\intercal.$ 
Then $u = -\left(s-s^{-1}\right)\left(t-t^{-1}\right)$ gives
$$\left(t-t^{-1}\right)\;\left(1, s^{-1}-s\right)^\intercal = \left(t-t^{-1},u\right)^\intercal$$
which is a contradiction to irreducibility.
\end{remark}

\begin{lemma}
\label{lem:reducible_commutator2}
$A, B \in\SLF$ have a common invariant 1--dimensional subspace if and only if 
\begin{equation}
\label{eq:commutator_of_reducible}
	\tr\left(ABA^{-1}B^{-1}\right) = 2.
\end{equation}
\end{lemma}

\begin{proof}
If $A, B \in\SLF$ have a common invariant 1--dimensional subspace, then we may conjugate them to the form given in \Cref{eqn:redForm}.
Then $\tr\left(ABA^{-1}B^{-1}\right) = 2.$ 

Now suppose $A, B \in\SLF$ have no common 1--dimensional invariant subspace. We may conjugate so they are in the form given in \Cref{eqn:irredForm}.
Then 
\[
	\tr\left(ABA^{-1}B^{-1}\right) = 2 + u^2 + u\left(s-s^{-1}\right)\left(t-t^{-1}\right).
\]
If $\tr\left(ABA^{-1}B^{-1}\right)=2,$ then $u\neq 0$ implies $0 \neq u = -\left(s-s^{-1}\right)\left(t-t^{-1}\right)$. This contradicts \Cref{eqn:irredForm}. Hence $\tr\left(ABA^{-1}B^{-1}\right)\neq2.$
\end{proof}

We can extend this idea to a triple of matrices.
\begin{lemma}
\label{lem:3_pairwise_reducible}
Suppose $(A, B, C) \in\SLF^3$ is irreducible, but any two of the matrices have a common invariant 1--dimensional subspace. Then $\tr[\;A,\; BC\;] \neq 2.$
\end{lemma}

\begin{proof}
Let $e_1$ be a common eigenvector of $A$ and $B,$ and $e_2$ be a common eigenvector of $A$ and $C.$ Since $(A, B, C) \in\SLF^3$ is irreducible, $(e_1, e_2)$ is a basis of $\F^2$, and hence up to conjugation we may assume
\[
	A = \begin{pmatrix} v & 0 \\ 0 & v^{-1} \end{pmatrix}, \quad
	B=\begin{pmatrix} s & 1 \\ 0 & s^{-1} \end{pmatrix}, \quad
	C=\begin{pmatrix}  t & 0 \\ u & t^{-1} \end{pmatrix}, \quad s, t, u, v \in\F\setminus\{0\}.
\]
Now $\tr [\;B,\; C\;] =2$ and $u \neq 0$ implies $0 \neq u = -\left(s-s^{-1}\right)\left(t-t^{-1}\right)$. 
Then we get $\tr[\;A,\; BC\;] = 2$ if and only if $\left(v^2-1\right)\left(s^2-1\right)\left(t^2-1\right)=0.$ 
Since $(A, B, C) \in\SLF^3$ is irreducible, we have $v^2\neq1$ and thus $$\left(v^2-1\right)\left(s^2-1\right)\left(t^2-1\right) \neq 0$$ 
This gives $\tr[\;A,\; BC\;] \neq 2$ as required.
\end{proof}
Another way of formulating the statement in \Cref{lem:3_pairwise_reducible} is: if $(A, B, C) \in\SLF^3$ is irreducible, but any two of the matrices have a common invariant 1--dimensional subspace, then $A$ and $BC$ do not share a common invariant 1--dimensional subspace.

Let $(A_1,\ldots, A_n) \in\SLF^n$ and define the $d = n+\binom{n}{2}+\binom{n}{3}$ traces
\begin{equation}
\label{eq:traces_of_three}
	t_{j_1j_2\ldots j_m} = \tr (A_{j_1}\ldots A_{j_m}), \quad 1\le j_1<j_2 < \ldots < j_m\le n, \quad m \le 3
\end{equation}
Define 
\begin{equation}\label{eq:defn_of_trn}
\strn \co \SLF^n \to \F^d
\end{equation}
 to be the map

\[
	\SLF^n \ni (A_1,\ldots, A_n) \mapsto 
	(t_{j_1j_2\ldots j_m})_{1\le j_1<j_2 < \ldots < j_m\le n,\; m \le 3} \in  \F^d
\]
as in \Cref{sec:def_varieties} for $\Gamma = F_n$. 
The indices are ordered first by length and then lexicographically. 
We identify the coordinate ring of $\F^d$ with
\[
	\F[t_1, t_2, \ldots, t_d]
\]
and note that this gives natural identifications $t_{n+1} = t_{12},$ \ldots , $t_d = t_{n-2\; n-1\; n}.$
 
The group $\SLF$ acts by conjugation on $n$--tuples. We write $A = (A_1,\ldots, A_n) \in \SLF^n$ and the action of $Q \in \SLF$ is 
\[
	Q \cdot A = QAQ^{-1} = (QA_1Q^{-1},\ldots, QA_nQ^{-1})
\]
Note that $\strn(Q \cdot A) = \strn(A)$ since traces are invariant under conjugation.

\begin{proposition}[Reducible is Zariski closed]
\label{prop:reducible_closed}
$(A_1,\ldots, A_n) \in\SLF^n$ is reducible if and only if 
\begin{enumerate}
	\item $\tr [\;A_i,\; A_j\;] = 2$ for all $1 \le i < j \le n$, and
	\item $\tr [\;A_i,\; A_jA_k\;] = 2$ for all $1 \le i < j <k \le n.$
\end{enumerate}
In particular, the set $R^{red}_n$ of all reducible $(A_1,\ldots, A_n) \in\SLF^n \subset \F^{4n}$ is a Zariski closed set.
\end{proposition}

\begin{proof}
If $(A_1,\ldots, A_n) \in\SLF^n$ is reducible, then the equations are clearly satisfied by the previous results. 
Hence suppose that $(A_1,\ldots, A_n) \in\SLF^n$ is not reducible. 

If there is a pair $(A_i, A_j)$ that is irreducible, then $\tr [\;A_i,\; A_j\;] \neq 2$. 
Hence assume that each pair  $(A_i, A_j)$ is reducible but $(A_1,\ldots, A_n)$ is irreducible. Then each $A_i$ must have at least two invariant 1--dimensional subspaces. 

Note that if $(A_1,\ldots, A_n) \in\SLF^n$ is irreducible and $A_1 = \pm \idMat,$ then $(A_2,\ldots, A_n) \in\SLF^{n-1}$ is irreducible. Hence without loss of generality, we may assume that  $A_1\neq \pm \idMat.$

Suppose the 1--dimensional invariant subspaces of $A_1$ are spanned by $e_1$ and $e_2,$ and, after relabelling, suppose $A_2, \ldots, A_{i_0-1}$ all have invariant subspaces spanned by $e_1$ and $e_2.$ Since $(A_1,\ldots, A_n) \in\SLF^n$ is irreducible, we have $i_0-1<n.$ If all other matrices have an invariant subspace spanned by $e_2,$ then the representation is reducible. Hence there is a matrix $A_{i_0}$ with no invariant subspace spanned by $e_2.$ 
Since $A_1$ and $A_{i_0}$ have a common invariant subspace, this is spanned by $e_1.$ 

Let $A_{i_0}, \ldots, A_{j_0-1}$ be the matrices with invariant subspaces spanned by $e_1$ but not $e_2.$ Again, $j_0-1<n,$ since otherwise 
$(A_1,\ldots, A_n) \in\SLF^n$ is reducible. Hence let $A_{j_0}, \ldots, A_n$ be those with invariant subspaces spanned by $e_2$ but not $e_1.$ In particular, $A_k\neq \pm \idMat$ for all $i_0 \le k \le n.$

With respect to the basis $(e_1, e_2),$ we have
\begin{equation}
\label{eq:irreducible_with_pw_red}
	A_{1} = \begin{pmatrix} v & 0 \\ 0 & v^{-1} \end{pmatrix} \quad\text{and}\quad
	A_{i_0}=\begin{pmatrix} s & 1 \\ 0 & s^{-1} \end{pmatrix} \quad\text{and}\quad
	A_{j_0}=\begin{pmatrix}   t & 0 \\ u & t^{-1} \end{pmatrix}
\end{equation}
where $u\neq 0,$ $v^2 \neq 1,$ $s^2 \neq 1,$ $t^2 \neq 1$ and $1 < i_0 < j_0.$
The condition on $A_{i_0}$ and $A_{j_0}$ having a common invariant 1--dimensional subspace  implies
\[ 
	u = (s^{-1}-s)(t-t^{-1})
\]
We compute $\tr [\;A_1,\; A_{i_0}A_{j_0}\;] - 2 =  (v-v^{-1})^2(1-s^{-2})(1-t^{-2})\neq 0.$
This completes the proof.
\end{proof}

We can also better understand irreducible representations and their the image under $\strn$. 
\begin{proposition}[Irreducible uniquely determined by traces of at most three]
\label{lem:traces_of_three_enough}
Suppose $A = (A_1,\ldots, A_n)\in\SLF^n$ is irreducible. Then $A$ is determined uniquely up to conjugation by $\strn(A).$ Hence the orbit of $A$ under the conjugation action is the Zariski closed set 
\[
	\SLF\cdot A = \{ B \in \SLF^n \mid \strn(B)=\strn(A)\} \subset \F^{4n}
\]
\end{proposition}

\begin{proof}
The hypothesis implies $n\ge 2.$ First suppose that there are two matrices which have no common invariant 1--dimensional subspace. We can identify such a pair $A, B$ amongst the matrices $A_1,\ldots, A_n$ using the given traces since they satisfy 
\[
	2 \neq \tr [\;A,\; B\;] = (\tr (A))^2 + (\tr (B))^2 + (\tr (AB))^2 - \tr (A)\tr (B)\tr (AB) -2
\]
according to \Cref{lem:reducible_commutator2}, 
\Cref{eq:commutator_trace} and noting $\tr (AB) = \tr (BA).$

First assume that $A = A_1$ and $B = A_2$. Up to conjugation, they are of the form 
\[
	A_1=\begin{pmatrix} s_1 & 1 \\ 0 & s_1^{-1} \end{pmatrix} \quad\text{and}\quad
	A_2=\begin{pmatrix}  s_2 & 0 \\ u_2 & s_2^{-1} \end{pmatrix}
\]
with $u_2\neq 0.$
The equations $t_1=\tr (A_1)=s+s^{-1}$ and $t_2=\tr (A_2)=t+t^{-1}$ have up to four distinct pairs of solutions $(s,t),$ $(s,t^{-1}),$ $(s^{-1},t),$ $(s^{-1},t^{-1}).$ The pair $(s_1, s_2)$ equals one of these four pairs, and for each such choice, $u_2$ is uniquely determined by $t_{12}=\tr (A_1A_2)=st+u_2+s^{-1}t^{-1}.$ 
The triple $(t_1, t_2, t_{12})$ therefore determines up to four triples $(s_1, s_2, u_2)$:
\begin{equation*}
\begin{split}
	(s_1, s_2, u_2)  \in \{\; &(s, t, t_{12}-st-s^{-1}t^{-1}),\; (s, t^{-1}, t_{12}-st^{-1}-s^{-1}t),\; \\
	&(s^{-1}, t, t_{12}-st^{-1}-s^{-1}t),\; (s^{-1}, t^{-1}, t_{12}-st-s^{-1}t^{-1})\;\}
\end{split}
\end{equation*}
Any two these four triples give conjugate pairs of matrices.
This proves the lemma for the case $n=2.$ 

Suppose $n \ge 3,$ and that we have chosen 
$(s_1, s_2, u_2) = (s, t, t_{12}-st-s^{-1}t^{-1})$ in the above. This fixes the pair $(A_1, A_2)$ in its conjugacy class. We now need to show that the entries of all other matrices $A_3, \ldots, A_n$ are uniquely determined by the set of traces  in \Cref{eq:traces_of_three}.

Let $3 \le k \le n.$ The four entries of $A_k = \begin{pmatrix}  w_k & x_k \\ y_k & z_k \end{pmatrix}$ satisfy a quadratic equation from $\det A_k = 1,$ and four linear equations from $\tr (A_k) = t_{k},$ $\tr (A_1A_k) = t_{1k},$ $\tr (A_2A_k)=t_{2k},$ and $\tr (A_1A_2A_k)=t_{12k}.$ The linear equations give:
\[
	\begin{pmatrix}
	1 & 0 & 0 & 1\\
	s & 0 & 1 & s^{-1}\\
	t&u&0&t^{-1}\\
	st+u&s^{-1}u&t^{-1}&s^{-1}t^{-1}
	\end{pmatrix}
	\begin{pmatrix}
	w_k\\x_k\\y_k\\z_k
	\end{pmatrix}
	=
	\begin{pmatrix}
	t_{k}\\t_{1k}\\t_{2k}\\t_{12k}
	\end{pmatrix}
\]
This system of linear equations has a unique solution if and only if 
\[ 
	u \left( u - \left(s^{-1}-s\right)\left(t-t^{-1}\right)\right) \neq 0
\]
This is equivalent to $\tr\left(A_1A_2A_1^{-1}A_2^{-1}\right) \neq 2.$

Hence if $A_1$ and $A_2$ have no common invariant 1--dimensional subspace and $n\ge 3,$ then $A_1,\ldots, A_n\in\SLF$ are uniquely determined up to conjugation by the $n + (n-1) + (n-2) + (n-3) = 4n-6$ values
\[
	\tr (A_{1}),\; \tr (A_{2}),\; \tr (A_{k}),\; \tr (A_{1}A_{2}),\; \tr (A_{1}A_{k}),\;  \tr (A_{2}A_{k}), \;\tr (A_{1}A_2A_{k}) \quad 3\le k < n
\]
This is clearly a subset of the traces given in \Cref{eq:traces_of_three}.

Now given the pair of matrices $A,B$ with no common invariant $1$--dimensional subspace, if $A = A_i$ and $B = A_j,$ we may assume $i<j$ and determine the entries as above from the values of
\[
	\tr (A_{k}),\; \tr (A_{i}A_{k}),\;  \tr (A_{j}A_{k}), \;\tr (A_{i}A_jA_{k}) \quad 1\le k < n, \; k \notin \{ i, j\}
\]
It follows from \Cref{eq:tr_conjugate,eq:tr_permute_three} that all of these traces are polynomials in integral coefficients in the traces given in \Cref{eq:traces_of_three}. This completes the proof of the lemma if there are at least two matrices without a common invariant 1--dimensional subspace.

Suppose any two matrices have a common invariant 1--dimensional subspace, but there is no such subspace invariant under all matrices. This is exactly the set-up in the second half of the proof of \Cref{prop:reducible_closed}. Now \Cref{lem:3_pairwise_reducible} and \Cref{prop:reducible_closed} imply that amongst $A_1,\ldots, A_n$ there are three matrices $A,$ $B$ and $C$ with the property that $\tr[\;A, \; BC\;]\neq 2.$ Again, suppose first that 
$A = A_1,$ $B=A_2$ and $C=A_3.$ We make the Ansatz 
\[
	A_{1} = \begin{pmatrix} v & 0 \\ 0 & v^{-1} \end{pmatrix} \quad\text{and}\quad
	A_{2}=\begin{pmatrix} s & 1 \\ 0 & s^{-1} \end{pmatrix} \quad\text{and}\quad
	A_{3}=\begin{pmatrix}  t & 0 \\ u & t^{-1} \end{pmatrix}
\]
with $t_1 = \tr (A_1),$ $t_2 = \tr (A_2),$ $t_3 = \tr (A_3),$ 
$t_{12}= \tr (A_1A_2)$,
$t_{13}= \tr (A_1A_3)$,
$t_{23}= \tr (A_2A_3)$,
$t_{123}= \tr (A_1A_2A_3)$. Here, the left hand sides are the given values for the traces, and the right hand sides are given in terms of the indeterminants $v, s, t, u.$ We need to show that we can express them uniquely in terms of the values of the traces.

As in the proof of \Cref{prop:reducible_closed}, $\tr [\;A_{2},\; A_{3}\;] = 2$ implies
$u = (s^{-1}-s)(t-t^{-1}),$ and then 
\[
	0\neq  \tr [\;A_1,\; A_{2}A_{3}\;] - 2 = (v-v^{-1})^2(1-s^{-2})(1-t^{-2})
\]
This implies $v^{2}\neq 1,$ $s^{2}\neq 1$ and $t^{2}\neq 1.$ 

Combining $t_{12} = \tr (A_1A_2)$ and $t_2 = \tr (A_2)$ gives 
\[
	s = \frac{v-v^{-1}}{t_2 v - t_{12}}
\]
Combining $t_{13}=\tr (A_1A_3)$ and $t_3=\tr (A_3)$ gives
\[
	t = \frac{v-v^{-1}}{t_3 v - t_{13}}
\]
In particular, all entries of $A_1$ and $A_2$ are uniquely determined by the traces $t_2$, $t_3$, $t_{12}$, $t_{13}$, $t_{23}$ and choosing a zero $v$ of $t_1 = x + x^{-1}$ as upper left entry for $A_1$. We need to identify the correct zero. Note that repeated application of $v^2 = t_1 v-1$ allows us to express all higher powers of $v$ as linear polynomials with coefficients in $\Z[t_1].$ Hence we obtain $0 = t_{123} - \tr (A_{1}A_{2}A_{3})$ if and only if $0 = f_1 - f_2 v,$ where $f_1, f_2 \in \Z[t_1,t_2, t_3, t_{12}, t_{13}, t_{23},t_{123}].$ This uniquely determines $v$ unless $f_1=0=f_2.$ In terms of our indeterminants, we have 
\[
	f_1 = v^{-5} (s^2-1)(t^2-1)(v^2-1)^4 \neq 0
\]
This proves the claim for $A_1, A_2, A_3.$

For every $A_k$ where $k\ge 4,$ we again let $A_k = \begin{pmatrix}  w_k & x_k \\ y_k & z_k \end{pmatrix}$ and consider the four equations arising from 
$\tr (A_k) = t_k,$ 
$\tr (A_1A_k) = t_{1k},$
$\tr (A_2A_k) = t_{2k},$
$\tr (A_3A_k) = t_{3k}.$
As above this system of linear equations in the matrix entries, 
\[
	\begin{pmatrix}
	1 & 0 & 0 & 1\\
	v & 0 & 0 & v^{-1}\\
	s&0&1&s^{-1}\\
	t&u&0&t^{-1}
	\end{pmatrix}
	\begin{pmatrix}
	w_k\\x_k\\y_k\\z_k
	\end{pmatrix}
	=
	\begin{pmatrix}
	t_{k}\\t_{1k}\\t_{2k}\\t_{12k}
	\end{pmatrix}
\]
has a unique solution under our hypotheses on the eigenvalues.

Hence the representative in the conjugacy class is uniquely determined by (a subset of) the given traces. The result in the general case (where we find three matrices amongst the generators satisfying the required inequality on the trace of a commutator) now again follows by application of trace identities.

For the second statement note that \Cref{prop:reducible_closed} implies that $\strn(A)$ is not the image of a reducible $n$--tuple.
This proves the result.
\end{proof}

We can elaborate on some of these steps.
\begin{remark}
It is a pleasant exercise to explicitly determine the $n$--tuple in the second part of the proof.
Up to conjugation, we may assume that all of $A_1, \ldots, A_{i_0-1}$ are diagonal (with $e_1$ representing the first standard basis vector) and such that $A_1^2 \neq \idMat$, all of $A_{i_0}, \ldots, A_{j_0-1}$ are upper triangular (but not diagonal), and all of $A_{j_0}, \ldots, A_n$ are lower triangular (but not diagonal). We further fix $(A_1, \ldots, A_n)$ in its conjugacy class by stipulating that the off-diagonal entry of $A_{i_0}$ equals one. This completely fixes the matrix entries in $(A_1, \ldots, A_n)$ without ambiguity. 

We now show that all matrix entries are uniquely determined by the values of the traces given in \Cref{eq:traces_of_three}. Write
\[
	A_{k} = \begin{pmatrix} s_k & 0 \\ 0 & s_k^{-1} \end{pmatrix} \quad\text{and}\quad
	A_{i}=\begin{pmatrix} s_i & w_i \\ 0 & s_i^{-1} \end{pmatrix} \quad\text{and}\quad
	A_{j}=\begin{pmatrix}  s_j & 0 \\ u_j & s_j^{-1} \end{pmatrix}
\]
where $1\le k \le i_0-1,$ $i_0\le i \le j_0-1,$ $j_0\le j \le n,$ and $w_i\neq 0 \neq u_j$ and $w_{i_0}=1,$ $s_1^2 \neq 1,$ $s_i^2 \neq 1,$ $s_j^2 \neq 1.$
The condition on $A_{i_0}$ and $A_{j}$ having a common invariant 1--dimensional subspace implies for all $j_0\le j \le n,$
\[ 
	u_j = (s_{i_0}^{-1}-s_{i_0})(s_j-s_j^{-1}).
\]
The condition on $A_{j_0}$ and $A_{i}$ having a common invariant 1--dimensional subspace implies for all $i_0\le i \le j_0-1,$
\[ 
	w_i =  \frac{(s_{i}^{-1}-s_{i})(s_{j_0}-s_{j_0}^{-1})}{u_{j_0}} 
	= \frac{(s_{i}^{-1}-s_{i})(s_{j_0}-s_{j_0}^{-1})}{(s_{i_0}^{-1}-s_{i_0})(s_{j_0}-s_{j_0}^{-1})}
	= \frac{s_{i}^{-1}-s_{i}}{s_{i_0}^{-1}-s_{i_0}}
\]
This agrees with $w_{i_0}=1.$
In particular, the $n$--tuple $(A_1, \ldots, A_n)$ is completely determined by the values and position of the eigenvalues in the matrices. 
\end{remark}

We denote the orbit space under the conjugation action by $\SLF^n / \SLF,$ and call it the \textbf{character stack of the free group in $n$ letters}. We write $[A] = \SLF \cdot (A_1,\ldots, A_n)\in \SLF^n / \SLF.$
The map $\strn \co \SLF^n \to \F^d$ is constant on conjugacy classes of $n$--tuples, and hence factors through the character stack. Let
\[
	\qn \co \SLF^n / \SLF \to \F^{d}
\]
denote the intermediate map taking conjugacy classes of $n$--tuples of matrices to the $d=n+\binom{n}{2}+\binom{n}{3}$ traces in \Cref{eq:traces_of_three}.

It follows from \Cref{lem:reducible_commutator2} and \Cref{eq:commutator_trace} that for each $x\in \F^{d},$ the preimage $\qn^{-1}(x)$ contains either reducible or irreducible conjugacy classes, but not both. \Cref{lem:traces_of_three_enough} implies that $\qn$ is injective on orbits of irreducible $n$--tuples. We record this as:

\begin{corollary}[Irreducible orbit is Zariski closed]
\label{cor:irreducible_orbit_closed}
Let $[A], [B] \in \SLF^n / \SLF$ with $[A]$ irreducible. If $\qn([A]) = \qn([B])$, then $[A] = [B].$ 

In particular, $\SLF^n\cdot (A_1,\ldots, A_n) \subset \SLF^n$ is a Zariski closed set, and two such orbits are either disjoint or equal.
\end{corollary}

We note that for reducible $[A] \in \SLF^n / \SLF$, the preimage of $\qn([A])\in \F^{d}$ may contain infinitely many orbits in $\SLF^n / \SLF.$

We say that $A, B \in \SLF^n$ are \textbf{closure-equivalent} if and only if the Zariski closures of their orbits intersect. The quotient space associated with this equivalence relation $\sim_c$ is the  \textbf{variety of $\SLF$--characters of the free group in $n$ letters},  
$$X(F_n, \F)= \SLF^n / \sim_c$$ We denote the equivalence class of $A \in \SLF^n$ by $[A]_c.$

\begin{corollary}[Irreducible orbit closures]
\label{cor:irreducible_orbit_closures}
Let $A \in \SLF^n$ be irreducible. Then $[A]_c=[A]= \strn^{-1}(\strn(A)).$
\end{corollary}

\begin{proof}
According to \Cref{cor:irreducible_orbit_closed}, $[A]$ is Zariski closed and if $B$ is irreducible, then either $[A]\cap [B] = \emptyset$ or $[A]= [B].$ Now if $B$ is reducible, then the orbit of $B$ is contained in the Zariski closed set $R^{red}_n$ from \Cref{prop:reducible_closed}. Hence the closure of the orbit of $B$ is contained in $R^{red}_n$. But $[A]\cap R^{red}_n = \emptyset$. Hence $[A]_c$ only contains the orbit of $A.$
\end{proof}

We can similarly look at reducible orbit closures.
\begin{lemma}[Reducible orbit closures]
\label{lem:reducible_orbit_closures}
Suppose $A = (A_1,\ldots, A_n)\in\SLF^n$ is reducible. Then there are diagonal matrices $D_1,\ldots, D_n\in\SLF$ such that $\strn(A) = \strn(D)$, where $D = (D_1,\ldots, D_n).$ Moreover, $D \in [A]_c$ and $D^{-1} = (D_1^{-1},\ldots, D_n^{-1})\in [A]_c.$ In particular, $[A]_c = \strn^{-1}(\strn(A)).$
\end{lemma}

\begin{proof}
If $A = (A_1,\ldots, A_n)\in\SLF^n$ is reducible, then we may conjugate so that each $A_k$ is upper triangular,
\[
	A_k = \begin{pmatrix} s_k & u_k \\ 0 & s_k^{-1} \end{pmatrix}
\]
Define 
\[
	D_k = \begin{pmatrix} s_k & 0 \\ 0 & s_k^{-1} \end{pmatrix}, \quad
	Q = \begin{pmatrix} 0 & 1 \\ -1 & 0 \end{pmatrix}, \quad
	Q_u =  \begin{pmatrix} u & 0 \\ 0 & u^{-1} \end{pmatrix}, \quad u \in \F\setminus \{0\}
\]
Let $D = (D_1,\ldots, D_n).$ Then $\strn(A) = \strn(D).$

We next show $D \in [A]_c$ and $D^{-1} \in [A]_c.$ Note $D \in [A]$ if and only if $D^{-1}$$=$ $QDQ^{-1}\in [A]$, in which case there is nothing to show. Hence suppose $D \notin [A].$

Now
\begin{equation}
\label{eqn:redOrbitClosures}
	Q_uA_kQ_u^{-1} = A_k = \begin{pmatrix} s_k & u_ku^2 \\ 0 & s_k^{-1} \end{pmatrix}
\end{equation}
for every $u \in \F \setminus \{0\}$.

Since $D \notin [A],$ there is an index $i_0$ with $u_{i_0} \neq 0.$
Given the entries $s_k, u_k$, consider the algebraic subset $X$ of $\F^{4n}$ with variables $w_k, x_k, y_k, z_k$ defined by the equations
\[
	w_k = s_k,\;
	u_{i_0} x_k  = u_k x_{i_0},\;
	y_k=0, \;
	z_k =  s_k^{-1} \quad \text{ for all } 1\le k \le n.
\]
These represent entries of some matrix that solves \Cref{eqn:redOrbitClosures}. 
Then all variables are specified except for $x_{i_0}$ and 
$X\cong  \F$. We also observe that $X = \{ Q_uA_kQ_u^{-1} \mid u \in \F \setminus \{0\} \} \cup \{ D\}.$ 
This shows that the Zariski closure of $\{ Q_uA_kQ_u^{-1} \mid u \in \F \setminus \{0\} \}$ contains $\{D\}.$ 
Since the Zariski closure of a set contains the Zariski closure of any of its subsets, this implies that the Zariski closure of $[A]$ contains $D.$ Hence $D \in [A]_c.$ Conjugation by $Q$ shows $D^{-1} \in [A]_c.$

For the last part, we need to show that $D$ and $D^{-1}$ are the only $n$--tuples of diagonal matrices in $\strn^{-1}(\strn(A)).$
Suppose $\strn(A) = (t_1, \ldots, t_{n-2\; n-1\; n}).$ We show that there are at most two $n$--tuples of diagonal matrices with the specified traces.

If $t_k= \varepsilon_k 2,$ where $\varepsilon \in \{ \pm 1\},$ then let $D_k = \varepsilon_k E.$ Note that $D_k = D_k^{-1}$ in this case. If each $t_k= \varepsilon_k 2,$ then there is no ambiguity in the choices and there is exactly one $n$--tuple of diagonal matrices with the specified traces.

Suppose $t_i$ is the first coordinate not equal to $\pm 2.$ Then $t_i = x + x^{-1}$ has two distinct roots. Choose a root $d_i$ of $t_i = x + x^{-1}$ to be the upper left entry of $D_i.$ Let $t_j$ be the next coordinate not equal to $\pm 2.$ Then $t_j = x+x^{-1}$ has two distinct roots, and we need to choose the upper left entry $d_j$ of $D_j$ amongst them. We also have the requirement on the upper left entry $d_j$ of $D_j$ to satisfy
$t_{ij} = d_id_j + d_i^{-1}d_j^{-1}.$ If this is satisfied for both choices of roots, then
$d_id_j + d_i^{-1}d_j^{-1} = d_id_j^{-1} + d_i^{-1}d_j$, which is equivalent to 
$(d_i-d_i^{-1})(d_j-d_j^{-1})=0$, which contradicts $t_i^2 \neq 4 \neq t_j^2.$ Hence the root $d_j$ of $t_j = s_j + s_j^{-1}$ is uniquely determined by $t_{ij} = d_id_j + d_i^{-1}d_j^{-1}.$

Inductively, each matrix $D_k$, $i<k\le n$ is uniquely determined by the first choice of upper left entry in $D_i$ and the values of $t_k$ and $t_{ik}$ or it is determined by $t_k=\pm 2.$ It also follows that making a different choice for the upper left entry of $D_i$ results in inverses of all $D_k.$ This completes the proof.
\end{proof}


\subsection{Reducible and irreducible characters}

We now revert to the notation in \Cref{sec:def_varieties}. Recall that we have identified $\R(F_n,\F) = \SLF^n$ and for each $\rho\in \R(F_n,\F),$ we have the trace function $\tau_\rho\co F_n\to \F.$ \Cref{cor:irreducible_orbit_closures} and \Cref{lem:reducible_orbit_closures} imply:
\begin{theorem}[Traces are isomorphic to characters]
\label{thm:character_variety_is_characters}
The closure-equivalence classes in $\R(F_n,\F)=\SLF^n$ are precisely the fibres of $\strn\co\SLF^n\to \F^d.$ Hence the image of $\strn$ is naturally isomorphic with $\X(F_n,\F).$
\end{theorem}

We get the following corollary.
\begin{corollary}
\label{cor:trace_function_characterisation}
Let $\Gamma$ be a finitely generated group.
Suppose $\rho, \sigma \in \R(\Gamma,\F).$ Then the following four statements are equivalent:
\begin{enumerate}
	\item $\tau_\rho = \tau_\sigma$;
	\item $\tau_\rho$ and $\tau_\sigma$ agree on all ordered products of distinct generators;
	\item $\tau_\rho$ and $\tau_\sigma$ agree on all ordered single, double and triple products of distinct generators;
	\item $\rho$ and $\sigma$ are closure-equivalent.
\end{enumerate}
\end{corollary}

\begin{proof}
It suffices to prove this statement for $\Gamma=F_n$.

We first note that the equivalence of (1) and (2) is the content of \Cref{lem:trace_generators}.
Now (2) implies (3), since we restrict to a smaller set.
The equivalence of (3) and (4) is the content of \Cref{thm:character_variety_is_characters}.
Finally, (4) implies (1) since every trace function is an algebraic function on $\R(F_n,\F).$
\end{proof}

From now on, we will often identify $\X(\Gamma,\F)$ with its image in $\F^d$. 
We say that an element of the \emph{variety of characters} is \textbf{irreducible} (resp. \textbf{reducible}) if it is associated with an irreducible (resp. {reducible}) $n$--tuple of matrices. 

Reducible characters form a Zariski closed set since they are the image of a Zariski closed set (see \Cref{prop:reducible_closed}; specific equations are given in \Cref{cor:redChar}). We denote $\Xred(\Gamma,\F)$ the union of all reducible characters. The union of irreducible characters may not be Zariski closed. We denote $\Xirr(\Gamma,\F)$ the Zariski closure of the union of all irreducible characters.


\subsection{Traces in characteristic 2}
\label{sec:traces_char_2}

Recall from the introduction that $\F_2$ denotes an (arbitrary) algebraically closed field of characteristic 2. 
We introduce extra trace identities in characteristic 2 and use them to determine the short trace coordinates and birational inverse for $\varphi_n(\X(\Gamma,\F_2))\to \varphi_d(\X(\Gamma,\F_2))$ in characteristic 2, as mentioned in \Cref{rem:inverseChar2}.

Note that in all previously established identities, we have $2=0.$ 
In particular, \Cref{eq:tr_four_term} reads:
\begin{align}\label{eq:tr_four_term_c2}
	\notag0=& \tr(AC) \tr B \tr D + \tr A \tr B \tr (CD) + \tr B \tr C \tr (DA)  \\
		 & + \tr(AC) \tr (BD) + \tr(AB) \tr (CD) + \tr(BC) \tr (DA) \\
	\notag & + \tr(ACB) \tr D + \tr A \tr (BCD) + \tr B \tr (CDA)  + \tr C \tr (DAB)
\end{align}

As before, write $t_{i_0\ldots i_m} = \tr(A_{i_0}\cdots A_{i_m})$. We also use the notation $t_A = \tr(A)$ for all $A\in \SLFt$. 

\begin{lemma}\label{lem:reduction_using_single_c2}
Let $A, B, C, D, E \in \SLFt.$ Then 
\begin{equation}
\begin{split}
	t_A t_{BCDE}	&=(t_{BD}+t_Bt_D)(t_{ACE}+t_Ct_{AE})+(t_{CE}+t_Ct_E)(t_{BAD}+t_Bt_{AD})\\
				&+t_{BC}t_{ADE}+t_{BE}(t_{ACD}+t_Ct_{AD}+t_Dt_{AC}+t_At_Ct_D)\\
				&+t_{CD}(t_{BAE}+t_Bt_{AE}+t_Et_{AB}+t_At_Bt_E)+t_{DE}(t_{BAC}+t_Bt_{AC}+t_Ct_{AB}+t_At_Bt_C)
\end{split}
\end{equation}
In particular,
\begin{equation}
\label{eq:trace_four_char2}
\begin{split}
	t_A t_{ABCD}	&=(t_{AC}+t_At_C)t_{ABD}+t_{AB}t_{ACD}+t_Ct_{AB}t_{AD}+t_{AD}t_{ABC}\\
			&+t_Bt_Ct_D+t_Bt_{CD}+t_Ct_{BD}+t_Dt_{BC}
\end{split}
\end{equation}
\end{lemma}

\begin{proof}
The first equation can be verified via direct computation in characteristic 2.

The second follows from the first because, in characteristic 2, for all $X, Y \in \SLFt$ we have $t_{XX}=t_X^2$ and $t_{XXY}+t_{X}t_{XY}=t_Y$ using \Cref{eq:tr_invert,eq:tracePowers} respectively.
\end{proof}

It follows from \Cref{lem:reduction_using_single_c2} and induction that if $Z \subseteq \X(\Gamma,\F_2)$ is a component with the property that for some generator $\gamma_i$ of $\Gamma$ the trace function $t_i$ is not in the ideal defining ${\varphi}_d(Z),$ then we can define a birational inverse of the projection map when restricted to $\F_2^{2^n-1} \supset {\varphi}_n(Z) \to {\varphi}_d(Z)  \subset \F_2^{d}$ of the form
\[
	f=(f_1, \ldots, f_{2^n-1}) \co \F_2^{d}  \to \F_2^{2^n-1}
\]
with $f_k = t_k$ for all $1 \le k \le d$, and $f_k \in \F_2[t_1, \ldots, t_{i-1}, t_i^{\pm 1}, t_{i+1}, \ldots t_d]$ for all $k > d$. 
In particular, if $\Gamma$ is the free group, then we may take $i=1$ and apply $f$ to all of ${\varphi}_d(\X(\Gamma,\F_2)).$

\begin{lemma}\label{lem:trivial_traces_trivial_char_c2}
Let $Z \subseteq \X(\Gamma,\F_2)$ be a component with the property that for each generator $\gamma_i,$ the trace function $t_i$ vanishes on $Z.$ If all trace functions $t_{ij}$ also vanish on $Z$, then $Z$ only contains the character of the trivial representation. In particular, $Z$ is mapped to the origin in $\F_2^{d}.$
\end{lemma}

\begin{proof}
Suppose $\rho$ is a non-trivial representation with character $\chi \in Z.$
Then some generator, say $\gamma_i$ has non-trivial image and since $t_i(\chi)=0,$ it is parabolic and hence has a unique invariant subspace.
Now for every $1\le j\le n,$ we have $t_j(\chi) = t_{ij}(\chi) = 0.$ By computation modulo 2, it follows that for every $1\le j,k\le n,$, $t_{ijk}=0$. \Cref{eq:commutator_trace,eq:commutator_trace_three} and \Cref{prop:reducible_closed} imply that the image under $\rho$ of the subgroup generated by $\gamma_i$ and $\gamma_j$ is reducible. Since $\rho(\gamma_j)$ is either parabolic or trivial, this implies that it also leaves the unique 1--dimensional invariant subspace of $\rho(\gamma_i)$ invariant. 
Hence $\rho$ is reducible and conjugate into the upper triangular group. Since all generators have trace zero, they are either parabolic or the identity. Hence $\rho$ has the same character as the trivial representation.
\end{proof}

We get the following immediate corollary.
\begin{corollary}
Let $Z \subseteq \X(\Gamma,\F_2)$ be a component that contains at least two (and hence infinitely many) characters. Then at least one of the coordinate functions $t_i$ ($1\le i \le n$) or $t_{ij}$ ($1\le i < j \le n$) does not vanish on all of ${\varphi}_d(Z).$
\end{corollary}

We are now in a position to describe the birational inverse to the projection map in the case where all coordinate functions $t_i$, $1\le i \le n$, vanish on an arbitrary component $Z$. If ${\varphi}_d(Z) = \{ 0\} \subset \F_2^d,$ then this is mapped to the origin in $\F_2^{2^n-1}.$ Otherwise at least one of the functions  $t_{ij}$ does not vanish on all of ${\varphi}_d(Z)$.

Consider $t_{j_1j_2\ldots j_m}$ where $m \ge 4.$ 
Recall $A_i\in\SLFt$ is the trace of the image of the generator $\gamma_i$, $1\leq i \leq n$. 
Suppose for some pair of indices $j_a, j_b$ satisfying $j_1 \leq j_a < j_b \leq j_m$, the function $t_{j_aj_b}$ does not vanish on all of ${\varphi}_d(Z)$. If $j_b<j_m$ we apply \Cref{eq:tr_four_term_c2} with 
$A = A_{j_1}A_{j_2}\ldots A_{j_{m-1}}$, $B = A_{j_m},$ $C = A_{j_a}$ and $D = A_{j_b}.$
Since the traces of generators vanish, we obtain the equality on ${\varphi}_d(Z)$:
\[
	\tr (CD) \tr(AB)  =  \tr(AC) \tr (BD)  + \tr (AD) \tr(BC) + \tr A \tr (BCD)
\]
We note that on ${\varphi}_d(Z)$, \Cref{eq:tr_permute_three} implies that we may permute images of generators without changing traces, and applying \Cref{eq:tracePowers} gives $t_Y = t_{A_kA_kY}+t_{A_k}t_{A_kY}=t_{A_kA_kY}$, for any $1\leq k \leq n$ and any $Y\in\SLFt.$ 
This implies that adding or removing even powers of generators does not change the trace. 
Substituting and simplifying gives: 
\begin{equation}
\label{eq:trace_four_char2_ver2}
	t_{j_aj_b} t_{j_1j_2\ldots j_m} = t_{j_1\ldots \hat{j_a} \ldots j_{m-1}} t_{j_b j_m} + t_{j_1\ldots \hat{j_b} \ldots j_{m-1}} t_{j_a j_m} + 
	t_{j_1j_2\ldots j_{m-1}}t_{j_aj_b j_m}
\end{equation}
where the hat means that an index does not appear.

Hence $t_{j_1j_2\ldots j_m}$ can be expressed as a rational function in shorter traces and with denominator involving only $t_{j_aj_b}$. If $j_b=j_m$, we use $A = A_{j_m}A_{j_1}A_{j_2}\ldots A_{j_{m-2}}$ and $B = A_{j_{m-1}}$ instead, and get a similar expression to \Cref{eq:trace_four_char2_ver2}.

Suppose for every pair of indices $j_a < j_b$ the function $t_{j_aj_b}$ vanishes on the whole component ${\varphi}_d(Z)$. Then \Cref{lem:trivial_traces_trivial_char_c2} implies that the character restricted to the subgroup generated by $A_{j_1}, A_{j_2}, \ldots, A_{j_{m}}$ has the same character as the trivial representation and hence all trace functions with indices in this subset vanish identically on ${\varphi}_d(Z).$ Hence the birational inverse is identically zero on the corresponding coordinates in $\F_2^{2^n-1}.$ 

Combining the results from this section we can define the birational inverse. If there exists a generator $\gamma_i$ $1\leq i\leq n$ with $t_i$ nonzero on $\varphi_d(Z)$, then we can use \Cref{eq:trace_four_char2} and invert $t_i$ to define the birational inverse for $i=A$. If, for all generators $\gamma_i$, $1\leq i\leq n$, $t_i$ vanishes on $\varphi_d(Z)$ and there exists $1\leq k< l \leq n$ such that $t_{kl}$ is nonzero on $\varphi_d(Z)$, then we can use \Cref{eq:trace_four_char2_ver2} and invert $t_{kl}$ to define the birational inverse with $k=j_a,l=j_b$. If, for all generators $\gamma_i, \gamma_j$, $1\leq i \leq j \leq n$, $t_i$ and $t_{ij}$ vanish on $\varphi_d(Z)$, then the birational inverse is zero on the corresponding coordinates in $\F_2^{2^n-1}.$


\subsection{Two generator groups}

Consider $\Gamma=\langle \alpha, \beta \rangle$ the free group on two elements. We can fully describe $\X(\Gamma,\F)$ following well-known classical arguments.

\begin{lemma}
\label{cor:twoGenChar}
	The embedding $\varphi_2\co \X(\Gamma,\SLF)\hookrightarrow \F^{2^2-1}=\F^3$ is an isomorphism. 
	That is, for all $(a, b, c) \in \F^3$ there exists some $(A,B)\in \SLF^2$ such that $\tr A = a,$ $\tr B = b,$ and $\tr AB = c.$
	Moreover, $(A,B)$ is reducible if and only if $a^2 + b^2 + c^2 -abc = 4.$
\end{lemma}

\begin{proof}
	Consider the representation variety $R(\Gamma,\F)=\{(A,B)\in\SLF^2\}$.
	It follows from \Cref{thm:character_variety_is_characters} that 
	\[
		\X(\Gamma,\F)=\{ (\tr(A),\tr(B),\tr(AB))\in\F^3 \mid (A,B)\in \R(\Gamma,\F)\}
	\]
	We want to show $\X(\Gamma,\F) \cong \F^3$.
	
	Let $\chi=(a,b,c) \in\F^3$. 
	We need to show there is a pair of matrices $A, B$ such that $(\tr(A),\tr(B),\tr(AB))=\chi$. 
	
	We may choose $d_a, d_b \in \F\setminus\{0\}$ such that $a=d_a+d_a^{-1}$ and $b=d_b+d_b^{-1}.$ Let
	\[
		A = \begin{pmatrix} d_a & 1 \\ 0 & d_a^{-1} \end{pmatrix}, \quad 
		B = \begin{pmatrix} d_b & 0 \\ u & d_b^{-1} \end{pmatrix} 
	\]
with $u\in\F$ to be determined.
	Then 
	\[
	\begin{split}
		AB &= \begin{pmatrix} d_ad_b+u & d_b^{-1} \\ d_a^{-1}u & d_a^{-1}d_b^{-1} \end{pmatrix}
	\end{split}
	\]
	We can choose $u=c-(d_ad_b+d_a^{-1}d_b^{-1})$. Then $\tr(AB)=(d_ad_b+d_a^{-1}d_b^{-1})+u=c$ and $(\tr(A),\tr(B),\tr(AB))$ $=$ $(a,b,c)$. This completes the first part of the proof.
	
	Consider now the reducible representations. 	
	By \Cref{prop:reducible_closed}, $(A,B)$ is reducible if and only if $\tr[\;A,\;B\;]=2$. 
	The trace identity \Cref{eq:commutator_trace} gives $\tr[\;A,\;B\;]=a^2 + b^2 + c^2 -abc -2$.
	Then the reducible representations are precisely $a^2 + b^2 + c^2 -abc = 4$ for $(a,b,c)$ $=$ $(\tr(A),\tr(B),\tr(AB))$. 
	This finishes the proof.
\end{proof}


\subsection{Equations for reducible characters}

A method to derive equations for the variety of characters for $\F = \C$ is given in \cite{Gonzalez-character-1993}. 
Here, we describe equations for the reducible characters of the free group. 
The difference with the arguments in the proof of \Cref{lem:reducible_orbit_closures} is that there we knew that a point in $\F^d$ is the image of a reducible $n$--tuple under $\strn.$ 
The situation now is that we specify points in $\F^d$ using some equations and then need to show that a suitable $n$--tuple with the specified image exists. 
A priori, the set we specify may be too large.

\begin{corollary}
\label{cor:redChar}
For $n\geq 3$, the image under $\strn \co \SLF^n \to \F^d$ of the set of all reducible $n$--tuples is the Zariski closed set
\begin{align*}
	X^{red}_n = \{ &(\;t_1, \ldots,\; t_{n-2 \; n-1 \; n}) \in \F^d \mid \\
	&\text{for all } 1 \le i < j \le n,\\ 
	&\quad(1)\quad 	t_i^2+t_j^2+t_{ij}^2-t_it_jt_{ij}=4\;\\  
	& \text{and for all } 1 \le i < j<k \le n,\\ 
	&\quad(2)\quad 	t_{ij}^2 + t_{ik}^2 + t_{jk}^2 - 2 t_j t_{jk} t_k + t_j^2 t_k^2 + t_{ij} t_{ik} (t_{jk} - t_j t_k)=4 \;\\ 
	&\quad(3)\quad 	t_{ij}^2 + t_{ik}^2 + t_{jk}^2 - 2 t_i t_{ik} t_k + t_i^2 t_k^2 + t_{ij} t_{jk} (t_{ik} - t_i t_k)=4 \;\\ 
	&\quad(4)\quad 	t_{ij}^2 + t_{ik}^2 + t_{jk}^2 - 2 t_i t_{ij} t_j + t_i^2 t_j^2 + t_{ik} t_{jk} (t_{ij} - t_i t_j) =4 \;\\ 
	&\quad(5)\quad 	t_i^2+t_{jk}^2+t_{ijk}^2-t_it_{jk}t_{ijk}=4\;\\
	&\quad(6)\quad 	t_j^2+t_{ik}^2+t_{ijk}^2-t_jt_{ik}t_{ijk}=4\;\\
	&\quad(7)\quad 	t_k^2+t_{ij}^2+t_{ijk}^2-t_kt_{ij}t_{ijk}=4\;\\
	&\quad(8)\quad 	t_i^2 + t_{jk}^2 - t_i t_{jk} ( t_i t_{jk} + t_j t_{ik}+ t_k t_{ij} - t_i t_j t_k-t_{ijk} ) \\
			&\qquad \qquad+ ( t_i t_{jk} + t_j t_{ik}+ t_k t_{ij} - t_i t_j t_k-t_{ijk} )^2=4 \;\\
	&\quad(9)\quad 	t_j^2 + t_{ijk}^2 - t_j t_{ijk} (t_{ik} + t_{ij} t_{jk} - t_i t_k) + (t_{ik} + t_{ij} t_{jk} - t_i t_k)^2=4 \;
	\}
\end{align*}
\end{corollary}

The sets of equations all arise from traces of commutators. Equation (1) arise from the first criterion of \Cref{prop:reducible_closed};  equations (2), (3), (4) come from equations of the form $\tr[\;AB,\;AC\;]=2$; equations (5), (6), (7) arise from the second criterion of \Cref{prop:reducible_closed} and cyclic permutation; equation (8) comes the second criterion of \Cref{prop:reducible_closed} and changing the cyclic order; and equation (9) comes from equations of the form $\tr[\;B,\;ABC\;]=2$. 
Note that the right-hand side of each equation equates to $0$ in characteristic $2$. 
In fact, the last equation (9) is only required in characteristic 2. 

If $n=1$ there is only one generator and all $n$--tuples are reducible. 
If $n=2$ there are only two generators and we only need consider the equation (1). This is covered in the previous section.

\begin{proof}
We start without any assumptions on the characteristic.
It follows from the characterisation of reducible $n$--tuples in \Cref{prop:reducible_closed} and \Cref{eq:commutator_trace,eq:commutator_trace_three} that the image of each reducible $n$--tuple satisfies all equations and hence lies in $X^{red}_n.$

Let $\chi=(\;t_1, \ldots,\; t_{n-2 \; n-1 \; n})\in X^{red}_n.$ We need to show there is a reducible $n$--tuple such that $\strn(D_1,\ldots,D_n)=\chi$. We construct the $n$--tuple to be diagonal. 

Without loss of generality, we can rearrange the terms so there exists $\ii$, $1\leq \ii \leq n+1$ such that $t_m^2=4$ for all $m\in\{1,\ldots,\ii-1\}$ and $t_m^2\neq4$ for all $m\in\{\ii,\ldots,n\}$. Then $\ii$ is the lowest integer such that $t_{\ii}^2\neq4$. In particular, if $\ii=n+1$, then $t_m^2=4$ for all $m$, and if $\ii=1$, $t_m^2\neq4$ for all $m$. 

We go through the cases for substituting pairs and triples from $1\leq a \leq b \leq c \leq \ii \leq i \leq j \leq k \leq n$ into equations (1)--(9). 

Consider $1\leq a\leq \ii-1$. There is $\varepsilon_a \in \{ \pm 1\}$ such that $t_a =\varepsilon_a 2$. For all $1\leq a < \ii$, define 
\begin{equation}
\begin{split}
\label{eqn:option1}
	D_a &= \varepsilon_a \idMat
\end{split}
\end{equation}
Then $\tr(D_a)=t_a$.

For each $1\leq a<b\leq \ii-1$, consider $t_{ab}$. Equation (1) is equivalent to 
\begin{equation}
\begin{split}
\label{eqn:doubleAllTrivial}
	(t_{ab} -\varepsilon_a\varepsilon_b 2)^2 	&=0, \\ 
	\implies t_{ab} 						&= \varepsilon_a\varepsilon_b 2 = \tr(D_aD_b)
\end{split}
\end{equation}

For each $1\leq a<b<c\leq \ii-1$, consider $t_{abc}$. Equation (5) is equivalent to 
\begin{equation}
\begin{split}
\label{eqn:tripleAllTrivial}
	(t_{abc} - \varepsilon_a\varepsilon_b\varepsilon_c 2)^2 	&=0,\\
	\implies t_{abc} 								&= \varepsilon_a\varepsilon_b \varepsilon_c 2 = \tr(D_aD_bD_c)
\end{split}
\end{equation}
Substituting the results into the remaining equations gives $0=0$.

Consider $\ii$. Choose $d_{\ii}$ such that $t_{\ii}=d_{\ii}+d_{\ii}^{-1}$ and define 
\begin{equation}
\begin{split}
	D_{\ii} &=\begin{pmatrix} d_{\ii} & 0 \\ 0 & d_{\ii}^{-1} \end{pmatrix}
\end{split}
\end{equation}
Then $\tr(D_{\ii})=t_{\ii}.$

For $1\leq \ii\leq i\leq n$, consider $t_i$ and $t_{{\ii}i}$. Let $x \in \F$ be such that $t_i=x+x^{-1}$. Then equation (1) applied to $\ii, i$ gives
\begin{equation}
\begin{split}
\label{eqn:doubleOneTrivial}
	\left(\left(d_{\ii}x+d_{\ii}^{-1}x^{-1}\right) - t_{{\ii}i}\right) \left(\left(d_{\ii}x^{-1}+d_{\ii}^{-1}x\right) - t_{{\ii}i}\right)	&= 0
\end{split}
\end{equation}
Note that $t_{\ii}^2 \neq 4 \neq t_i^2$ implies that only one of these factors is equal to zero. This uniquely determines
$d_i\in\{x,x^{-1}\}$ such that 
\[
\begin{split}
	d_{\ii}d_i+d_{\ii}^{-1}d_i^{-1}-t_{{\ii}i} 	&=0,\\
	\implies t_{{\ii}i}					&=d_{\ii}d_i+d_{\ii}^{-1}d_i^{-1}
\end{split}
\]	
We let
\begin{equation}
\begin{split}
	D_{i} &=\begin{pmatrix} d_{i} & 0 \\ 0 & d_{i}^{-1} \end{pmatrix}
\end{split}
\end{equation}
Then $\tr(D_{i})=t_{i}$ and $\tr(D_{\ii}D_{i})=t_{\ii i}$. The $n$--tuple $(D_1,\ldots,D_n)$ is now completely defined, and we need to show that the traces of their double and triple products have the correct values.

For each $1\leq a < \ii \leq i \leq n$, consider $t_{ai}$. Equation (1) is equivalent to 
\begin{equation}
\begin{split}
\label{eqn:doubleTwoTrivial}
	\left(\left(d_{i}+d_{i}^{-1}\right) - \varepsilon_a t_{a{i}}\right)^2 	&= 0,\\
	\implies 					t_{a{i}} 		&=\varepsilon_a\left(d_{i}+d_{i}^{-1}\right)=\tr(D_aD_{i})
\end{split}
\end{equation}

For each $1\leq a < b < \ii \leq i \leq n$, consider $t_{abi}$. Equation (5) is equivalent to 
\begin{equation}
\begin{split}
\label{eqn:tripleTwoTrivial}
	\left(\left(d_{i}+d_{i}^{-1}\right) - \varepsilon_a\varepsilon_b t_{ab{i}}\right)^2 	&= 0,\\
	\implies t_{ab{i}} 									&= \varepsilon_a\varepsilon_b \left(d_{i}+d_{i}^{-1}\right)=\tr(D_aD_bD_{i})
\end{split}
\end{equation}
Substituting the results into the remaining equations gives $0=0$.

For each $1\leq a < \ii \leq i < j \leq n$, consider $t_{aij}$. Equation (5) is equivalent to 
\begin{equation}
\begin{split}
	\left(t_{ai j} -\varepsilon_a \left(d_{i}d_j+d_{i}^{-1}d_j^{-1}\right)\right)^2	&= 0 \\
	\implies t_{a{i}j} 							&= \varepsilon_a\left(d_{i}d_j+d_{i}^{-1}d_j^{-1}\right)=\tr(D_aD_{i}D_j)
\end{split}
\end{equation}
Substituting the results into the remaining equations gives $0=0$.

For each $1\leq \ii < j < k \leq n$, consider $t_{jk}$. Equation (1) is equivalent to 
\begin{equation}
\begin{split}
	\left(d_{j}d_{k}+d_{j}^{-1}d_k^{-1} - t_{jk}\right) \left(d_{j}d_k^{-1}+d_{j}^{-1}d_k - t_{jk}\right)	&= 0
\end{split}
\end{equation}

Then equations (2), (3), and (4) applied to $\ii,j,k$ gives
\begin{equation}
\begin{split}
\label{eqn:redChar234}
	\left(d_{j}d_{k}+d_{j}^{-1}d_k^{-1} - t_{jk}\right) \left(\left(1 - d_{\ii}^{-2}\right)d_j^{-1}d_k^{-1} \right. \quad& \\ \left.+\left(1-d_{\ii}^2\right)d_jd_k+d_j^{-1}d_k+d_j d_k^{-1}-t_{jk}\right) 	&=0 \\
	\left(d_{j}d_{k}+d_{j}^{-1}d_k^{-1} - t_{jk}\right) \left(d_{\ii}^2 d_j d_k^{-1} + d_{\ii}^{-2} d_j^{-1} d_k - t_{jk}\right) 							&=0 \\
	\left(d_{j}d_{k}+d_{j}^{-1}d_k^{-1} - t_{jk}\right) \left(d_{\ii}^{-2} d_j d_k^{-1} + d_{\ii}^2 d_j^{-1} d_k - t_{jk}\right) 							&=0
\end{split}
\end{equation}

Assume $t_{jk}\neq d_{j}d_{k}+d_{j}^{-1}d_k^{-1}$. Then,
\[
\begin{split}
	t_{jk}	&=d_{j}d_k^{-1}+d_{j}^{-1}d_k, \\
		&=\left(1 - d_i^{-2}\right)d_j^{-1}d_k^{-1}+\left(1-d_i^2\right)d_jd_k+d_j^{-1}d_k+d_j d_k^{-1} \\
		&=d_i^2 d_j d_k^{-1} + d_i^{-2} d_j^{-1} d_k \\
		&=d_i^{-2} d_j d_k^{-1} + d_i^2 d_j^{-1} d_k
\end{split}
\]

Since $t_i^2\neq4$, we have $d_i\neq d_i^{-1}$. 
The equations can be compared to get
\[
\begin{split}
	d_i^2 d_j^2 d_k^2 					&=1\\
	d_i^2 d_j^2 						&= d_k^2\\
	d_i^2 d_k^2 						&= d_j^2\\
	\left(d_j +d_j^{-1}\right)\left(d_i d_k  - d_i^{-1} d_k^{-1}\right)	&=0\\
	\left(d_i+d_i^{-1}\right)\left(d_jd_k^{-1}-d_j^{-1}d_k\right)	&=0.
\end{split}
\]
The last two equations give multiple possibilities.
Together with the first three equations,
\begin{equation}
\begin{split}
	d_j +d_j^{-1}=0 \text{ and } d_i+d_i^{-1} = 0 					&\implies d_k^2=1\\
	d_j +d_j^{-1}=0 \text{ and } d_jd_k^{-1}-d_j^{-1}d_k = 0 			&\implies d_i^2=1 \\
	d_jd_k^{-1}-d_j^{-1}d_k=0 \text{ and } d_i+d_i^{-1} = 0 			&\implies d_j^2=1 \\
	d_jd_k^{-1}-d_j^{-1}d_k=0 \text{ and } d_jd_k^{-1}-d_j^{-1}d_k = 0 	&\implies d_i^2=1
\end{split}
\end{equation}
which all give contradictions to $t_m^2\neq4$ for each $m\in\{i,j,k\}$.

Thus to satisfy \Cref{eqn:redChar234} we must have $t_{jk} = d_{j}d_{k}+d_{j}^{-1}d_k^{-1} = \tr(D_j D_k)$. This means equations (2), (3), (4) evaluate to $0$.

For each $1\leq \ii \leq i < j < k \leq n$, now consider $t_{ijk}$. 
Equations (5), (6), and (7) are equivalent to
\begin{equation}
\begin{split}
\label{eqn:redChar567}
	\left(d_{i} d_j^{-1} d_k^{-1} + d_{i}^{-1} d_j d_k - t_{{i}jk}\right) \left(d_{i}^{-1} d_j^{-1} d_k^{-1} + d_{i} d_j d_k - t_{{i}jk}\right) 	&= 0 \\
	\left(d_{i}^{-1} d_j d_k^{-1}  + d_{i} d_j^{-1} d_k - t_{{i}jk}\right) \left(d_{i}^{-1} d_j^{-1} d_k^{-1} + d_{i} d_j d_k - t_{{i}jk}\right) 	&= 0 \\
	\left(d_{i} d_j d_k^{-1}  + d_{i}^{-1} d_j^{-1} d_k - t_{{i}jk}\right) \left(d_{i}^{-1} d_j^{-1} d_k^{-1} + d_{i} d_j d_k - t_{{i}jk}\right) 	&= 0
\end{split}
\end{equation}

Assume $t_{{i}jk} \neq d_{i}^{-1} d_j^{-1} d_k^{-1} + d_{i} d_j d_k$. Then,
\[
\begin{split}
	t_{i jk}	&=d_{i} d_j^{-1} d_k^{-1} + d_{i}^{-1} d_j d_k\\
			&=d_{i}^{-1} d_j d_k^{-1}  + d_{i} d_j^{-1} d_k \\
			&=d_{i} d_j d_k^{-1}  + d_{i}^{-1} d_j^{-1} d_k
\end{split}
\]

Since $t_m^2\neq4$ for each $m\in\{i,j,k\}$, we have $d_m\neq d_m^{-1}$ for each $m\in\{i,j,k\}$. 
The equations can be compared to get $d_{i}^2=d_j^2=d_k^2$.

Then equations (8) and (9) applied to $i,j,k$ give
\begin{align}
	\label{eqn:redChar8}
	\frac{2 \left(d_k^2-1\right)^4 \left(d_k^2+1\right)^2}{d_k^6} &=0 \\
	\label{eqn:redChar9}
	\frac{ \left(d_k^2-1\right)^4 \left(d_k^2+1\right)^2 \left(1 + d_k^2 +  d_k^4\right)}{d_k^8} &=0
\end{align}
Again since $t_k^2\neq4$, we have $d_k\neq d_k^{-1}$ and $d_k^2-1\neq0$.

In characteristic $p\neq2$, \Cref{eqn:redChar8} implies $d_k^2=-1$. Given $d_{i}^2=d_j^2=d_k^2$ we also have $d_{i}^2=-1=d_j^2$. Then $t_{{i}jk}=0=\tr(D_{i}D_jD_k)$. Substituting the results into \Cref{eqn:redChar9} gives $0=0$. 

In characteristic $2$, \Cref{eqn:redChar8} gives $0=0$. \Cref{eqn:redChar9} this implies either $d_k^2=1$, which gives a contradiction, or $1 + d_k +  d_k^2=0$.
Then $d_k^{-1}=d_k+1$ and $t_{{i}jk}=0=\tr(D_{i}D_jD_k)$.

Assume instead $t_{{i}jk} = d_{i}^{-1} d_j^{-1} d_k^{-1} + d_{i} d_j d_k$ from \Cref{eqn:redChar567}. Then equations (8) and (9) evaluate to $0$ and $t_{{i}jk}=\tr(D_{i}D_jD_k)$.

Thus, given $\chi \in X^{red}_n$, for all possible combinations of $1\leq a \leq b \leq c \leq \ii \leq i \leq j \leq k \leq n$, we can construct a representation $(D_1,\ldots,D_n)$ such that $\strn(D_1,\ldots,D_n)=\chi$. The proves the result. 
\end{proof}


\subsection{Normal forms; conjugacy classes}

We record a result on normal forms and conjugacy classes that will be useful later. The following observation is stated in \cite[Lemma 3.1]{Paoluzzi-examples-2020} in terms of $p\ge 3$ and $\PSLF$ with reference to the structure of projective special groups over finite fields. We give a different argument for $\SLF$ using minimal polynomials and Jordan normal form.

\begin{lemma}
\label{lem:JNF_in_SLk}
Let $\F$ be an algebraically closed field of characteristic $p$ and $A \in \SLF$ be an element of finite order. If $p\ge 3,$ then either the order of $A$ is coprime with $p$ or it equals $p$ or $2p.$ If $p=2,$ then the order of $A$ is either odd or equal to two. In addition:
\begin{enumerate}
	\item If the order of $A$ is coprime with $p$, then $A$ is diagonalisable.
	\item If the order of $A$ equals $p$, then $A$ is conjugate to $\begin{pmatrix} 1 & 1 \\ 0 & 1 \end{pmatrix}.$
	\item If $p$ is odd and the order of $A$ equals $2p$, then $A$ is conjugate to $ \begin{pmatrix} -1 & \phantom{-}1 \\ 0 & -1 \end{pmatrix}.$
\end{enumerate}
\end{lemma}

\begin{proof}
Suppose $A^n = \idMat.$ Then the minimal polynomial of $A$ divides $x^n-1$ in $\F[x].$ Suppose $n=p^mq$, where $q$ is coprime with $p.$ Applying the Frobenius map gives $x^n-1 = x^{p^mq} -1^{p^m} = (x^q -1)^{p^m}.$
In particular, the eigenvalues of $A$ are roots of $x^q-1.$ Let $\lambda$ be an eigenvalue of $A.$

 If $\lambda \neq \lambda^{-1},$ then $A$ is diagonalisable. Since $\lambda^q = 1= \lambda^{-q},$ we have $A^q=\idMat$ and hence the order of $A$ divides $q,$ which is co-prime with $p.$ This is the first case of the lemma.

Hence suppose $\lambda = \lambda^{-1}.$ Then $\lambda \in \{\pm 1\}$ and $A$ has normal form either $\pm \idMat$ or 
$\begin{pmatrix} 1 & 1 \\ 0 & 1 \end{pmatrix}$
or
$ \begin{pmatrix} -1 & \phantom{-}1 \\ 0 & -1 \end{pmatrix}.$

Now if $p=2,$ then $\idMat = -\idMat$ is diagonal of order one which is coprime with $2$ and hence is in the first case. If $p\ge 3,$ then $\idMat$ has order one and $-\idMat$ has order two; both are diagonal with order coprime with $p$ and are again covered by the first case.

For any prime $p,$ the matrix $\begin{pmatrix} 1 & 1 \\ 0 & 1 \end{pmatrix}$ has order equal to $p$, and hence this is the second case of the lemma. If $p=2,$ then $-1=1$ and we are done. If $p\ge 3$, then $ \begin{pmatrix} -1 & \phantom{-}1 \\ 0 & -1 \end{pmatrix}$ has order $2p$ and we are in the third case.

This completes the proof of the lemma.
\end{proof}


\section{Splittings of groups detected by ideal points}
\label{sec:Splittings}

Following the treatment in \cite{Culler-Shalen-varieties-1983}, we define ideal points of curves in the variety of characters for arbitrary algebraically closed fields in such a way that we can apply Bass-Serre theory~\cite{Bass-covering-1993, Serre-trees-2003} as in \cite{Culler-Shalen-varieties-1983}. The up-shot is that an ideal point of a curve in $\X(\Gamma, \F)$ gives rise to a \textbf{splitting} of $\Gamma,$ that is, an isomorphism of $\Gamma$ with the fundamental group of a graph of groups. 


\subsection{Ideal points and valuations}

We call a 1--dimensional irreducible subvariety a \textbf{curve}. We briefly recall the definitions of ideal points of a curve and their relationship with non-trivial discrete rank 1 valuation on the function field of the curve.

Let $\K$ be a field. A \textbf{discrete rank 1 valuation} on $\K$ is a map $v\co \K \to\Z \cup \{\infty\}$ that is the extension of a map $\K\setminus \{0\} \to\Z$ by letting $v(0)=\infty$ satisfying
\begin{enumerate}
	\item $v(ab)=v(a)+v(b)$ for all $a,b\in \K\setminus\{0\}$,
	\item $v(a+b)\geq\min\{v(a),v(b)\}$,
\end{enumerate}
where $z\leq\infty$ and $z+\infty=\infty=\infty+z$ for all $z\in\Z$.

Two varieties $V, W$ are \textbf{birationally equivalent} is there are rational, dense maps $\varphi \co V\to W$, $\psi \co W\to V$, such that $\varphi \circ \psi=1=\psi\circ\varphi$ where defined. 
The \textbf{projective completion} of a variety $V$ over $\F$ is the closure of $V$ under the map $J: \F^m\to \F P^m$ defined by $J(z_1,\ldots,z_m)=[1,z_1,\ldots,z_m]$. 
When $V$ is a 1--dimensional irreducible variety there is a unique non-singular projective variety $\tilde{V}$, which is birationally equivalent to the projective completion $\overline{J(V)}$. 
This is called the \textbf{smooth projective closure} of $V$, and exists over all algebraically closed fields $\F$ (and more generally, for all perfect fields). 

Let $\Gamma$ be a finitely generated group and consider a curve in the variety of characters, $C\subset\X(\Gamma,\F)$. 
The \textbf{ideal points} of $C$ are the points of the smooth projective closure $\tilde{C}$ that correspond to the set $\overline{J(C)}-J(C)$ under the birational equivalence. Intuitively, these are the points added to $C$ by the process of taking the smooth projective closure. Note $\F(\tilde{C})\cong \F(C)$. 

Take an ideal point $\xi\in\tilde{C}$. This defines a non-trivial discrete rank 1 valuation on $\F(C)$ via,
\begin{equation}
	\begin{split}
	v_\xi\co \F(C)&\to\Z\cup\{\infty\}\\
	v_\xi(f)&=
		\begin{cases}
		q & f\text{ has a zero of order }q\text{ at }\xi\\
		\infty & f=0 \\
		-q & f\text{ has a pole of order }q\text{ at }\xi
		\end{cases}
	\end{split}
\end{equation}

Take $R_C\subset \R(\Gamma,\F)$ to be an irreducible subvariety with the property that $q(R_C)=C$ under the natural map $q\co \R(\Gamma,\F)\to\X(\Gamma,\F)$. 
We can extend the valuation to $\F(R_C)$ using the following result from Alperin and Shalen~\cite{Alperin-linear-1982}.
\begin{lemma}[Alperin-Shalen]
\label{lem:valuationExtension}
Let $E$ be a discretely valued field with valuation $v$. If $F$ is a finitely generated extension field of $E$ then there is an extension of $v$ to a discrete valuation $\tilde{v}$ of $F$.
\end{lemma}

The natural algebraic map $q\co \R(\Gamma,\F)\to\X(\Gamma,\F)$ induces
\[
	\F(\X(\Gamma,\F))=\F[t_i]/I\subseteq \F[z_i]/I=\F(\R(\Gamma,\F))
\]
for ideal $I$ and $z_i$ coordinates for $\R(\Gamma,\F)$ such that $t_i\in\text{span}\{z_j\}$ for all $i$. 
We can restrict this to
\[
	\F(C)=\F[t_i]/I_C\subseteq \F[z_i]/I_C=\F(R_C)
\]
for the ideal $I_C$ defining the curve $C$, with $I \subset I_C$. 
Then setting $E=\F(C)$, $v=v_\xi$ and $F=\F(R_C)$, we can use \Cref{lem:valuationExtension} to extend the valuation $v_\xi$ to a discrete rank 1 valuation on $\F(R_C)$,
\begin{equation}
	\tilde{v}_\xi\co \F(R_C)\to\Z\cup\{\infty\}
\end{equation}

Let $\K=\F(R_C).$ The \textbf{tautological representation} is 
\begin{equation}
\begin{split}
	\P\co \pie(M)&\to \SL_2(\K),\\
	\P(\gamma)&=\begin{pmatrix} a & b\\ c& d \end{pmatrix} \text{ for }\rho(\gamma)=\begin{pmatrix} a(\rho) & b(\rho)\\ c(\rho)& d(\rho) \end{pmatrix}
\end{split}
\end{equation}
Here, $a, b, c, d$ are polynomial functions in the ambient coordinates defining $\R(M,\F)$ and can hence be viewed as elements in $\K=\F(R_C)$.


\subsection{Actions of groups on trees}

We continue with the notation from the previous section. Let $\K=\F(R_C)$, $\xi$ an ideal point of $C$ and $v=v_\xi$ the associated valuation. We recall the construction in \cite[\S 2]{Culler-Shalen-varieties-1983} of the \textbf{Bass-Serre tree} $\tree_\xi$ with an action of $\SL_2(\K)$. More details on the construction can be found in \cite{Baumslag-topics-1993, Serre-trees-2003, Shalen-representations-2002}.

The \textbf{valuation ring} is
\begin{equation}
	\O=\{a\in \K\mid v(a)\geq0\}
\end{equation}
The valuation ring is a principal ideal domain and will contain a maximal ideal $\mathcal{M}$ consisting of the non-units of $\O$ with $\mathcal{M}=\langle \pi \rangle$ for $v(\pi)=1$. 

Consider $V=\K^2$ a left $\O$-module. An \textbf{$\O$-lattice} $L$ is an $\O$-submodule of the form
\begin{equation}
	L=\O x+\O y
\end{equation}
for $x,y\in V$ linearly independent. The multiplicative group $\K^\times$ acts on $\O$-lattices by multiplication via
\begin{equation}
	aL=\O ax+\O ay
\end{equation}
with orbits of this action defining equivalence classes $\Lambda=[L]$. 

The key observation is that given $\O$-lattices $L_1$ and $L_2$, there is $m\in\Z$ such that $\pi^m L_2\subset L_1$. Hence for $\pi^m L_2\subset L_1$, $m\in\Z$, there is a basis $\{x,y\}$ for $L_1$ such that $\{\pi^f x,\pi^g y\}$ is a basis for $\pi^m L_2$ for some $f,g\in\Z$. One obtains a well-defined distance between the 
equivalence classes of $\O$-lattices $\Lambda_1=[L_1]$ and $\Lambda_2=[L_2]$ by letting
\begin{equation}
	d\left(\Lambda_1,\Lambda_2\right)=\lvert f-g \rvert
\end{equation}
This defines a metric on the set of all equivalence classes of $\O$-lattices.

We construct a graph (which turns out to be a tree) $\tree_\xi$ by the following set of vertices $\V$ and edges $\E$,
\begin{equation}
\begin{split}
	\V &= \{\Lambda=[L] \mid L \text{ an $\O$-lattice}\}\\
	\E &= \{\left(\Lambda_1,\Lambda_2\right) \mid d\left(\Lambda_1,\Lambda_2\right)=1\}
\end{split}
\end{equation}

The group $\GL(V)$ acts on $\O$-lattices via
\begin{equation}
	A\circ L=A\circ\left(\O x+\O y\right)=\O Ax+\O Ay
\end{equation}
for $A\in\GL(V)$. The action is well-defined on equivalence classes, giving an action on the tree $\tree_\xi$. In fact, this action is via isometries. That is, $d\left(A\circ\Lambda_1,A\circ\Lambda_2\right)=d\left(\Lambda_1,\Lambda_2\right)$, for all $A\in\GL(V)$ and for all $\Lambda_1, \Lambda_2$.

If we restrict the action to $\SL_2(\K)$, it can be proven that the action of $\SL_2(\K)$ on $\tree_\xi$ is non-trivial, simplicial, and without inversions.
We use the tautological representation $\P$ to pull back the action of $\SL_2(\K)$ on $\tree_\xi$ to an action of $\Gamma$ on $\tree_\xi$ by $\gamma\circ\Lambda=\P(\gamma)\circ\Lambda$ for all $\gamma\in\Gamma$. As per the action of $\SL_2(\K)$, the action of $\Gamma$ is simplicial and without inversions. 

We use the tautological representation to define a function $I_\gamma=\tr(\P(\gamma)) \in \F(\R(M,\F))$ for each $\gamma \in \pi_1(M)$:
\begin{equation}
\begin{split}
	I_\gamma \co \R(M,\F) 	&\to \F \\
					\rho	&\mapsto\tr(\rho(\gamma))
\end{split}	
\end{equation}
We say that $I_\gamma$ \textbf{blows up} (at $\xi$) if $v_\xi(I_\gamma)<0.$ Otherwise we say that $I_\gamma$ is \textbf{bounded} (at $\xi$).

The proofs of \cite[Property 5.4.2]{Shalen-representations-2002} and \cite[Theorem 2.2.1]{Culler-Shalen-varieties-1983} apply verbatim with the current set-up, and hence we have the following result:
\begin{theorem}[Culler-Shalen]
\label{pro:stabilisers}
Let $C$ be a curve in $\X(\Gamma, \F).$ To each ideal point $\xi$ of $\tilde{C},$ one can associate a splitting of $\Gamma$ with the property that for each element $\gamma \in \Gamma,$ the following are equivalent:
\begin{enumerate}
	\item $v_\xi(\trFunc_\gamma)\ge 0$
	\item A vertex of the Bass-Serre tree $\tree_\xi$ is fixed by $\gamma.$
\end{enumerate}
Thus, in particular, the action of $\Gamma$ on $\tree_\xi$ is non-trivial and the associated splitting of $\Gamma$ is non-trivial.
\end{theorem}


\section{Essential surfaces in 3--manifolds detected by ideal points}
\label{sec:CullerShalen}

We outline the tools and approach used by Culler and Shalen \cite{Culler-Shalen-varieties-1983} now over $\F$ to associate essential surfaces in $3$--manifolds to ideal points of curves in the variety of $\SLF$--characters. A brief outline is given that follows~\cite{Shalen-representations-2002}, where further details and proofs can be found. We conclude with remarks on bounds on the dimension of the variety of characters, which were first provided in \cite{Culler-Shalen-varieties-1983} based on work by Thurston \cite{Thurston-notes}. 


\subsection{Essential surfaces in 3--manifolds}
\label{sec:EssentialSurfaces}

We start by defining essential surfaces following standard terminology from Jaco \cite{Jaco-lectures-1980}.
A {surface} $S$ in a compact 3--manifold $M$ will always mean a 2--dimensional piecewise linear submanifold \emph{properly embedded} in $M$. That is, a closed subset of $M$ with $\partial S = S \cap \partial M$. If $M$ is not compact, we replace it by a compact core. 

An embedded sphere $S^2$ in a 3--manifold $M$ is called \textbf{incompressible} or \textbf{essential} if it does not bound an embedded ball in $M$, and a 3--manifold is \textbf{irreducible} if it contains no incompressible 2--spheres. 
A surface $S$ without 2--sphere components in the 3--manifold $M$ is called \textbf{incompressible} if for each disc $D\subset M$ with $D \cap S = \partial D$ there is a disc $D' \subset S$ with $\partial D' = \partial D$. Note that \emph{every} properly embedded disc in $M$ is incompressible.

A surface $S$ in a 3--manifold $M$ is \textbf{$\partial$--compressible} if either 
\begin{enumerate}
\item $S$ is a disc and $S$ is parallel to a disc in $\partial M,$ or
\item $S$ is not a disc and there exists a disc $D\subset M$ such that $D\cap S = \alpha$ is an arc in $\partial D,$ $D\cap \partial M = \beta$ is and arc in $\partial D,$ with $\alpha \cap \beta = \partial \alpha = \partial \beta$ amnd $\alpha \cup \beta = \partial D$ and either $\alpha$ does not  separate $S$ or $\alpha$ separates $S$ into two components and the closure of neither is a disc.
\end{enumerate}
Otherwise $S$ is \textbf{$\partial$--incompressible}.

\begin{definition}  \cite{Shalen-representations-2002}\label{def:essential}
A surface $S$ in a compact, irreducible, orientable 3--manifold is said to be \textbf{essential} if it has the following properties:
    \begin{enumerate}
       \item $S$ is bicollared;
       \item the inclusion homomorphism $\pi_1(S_i) \to \pi_1(M)$ is
          injective for
           every component $S_i$ of $S$;
       \item no component of $S$ is a 2--sphere;
       \item no component of $S$ is boundary parallel;
       \item $S$ is non-empty.
    \end{enumerate}
\end{definition}
In the first condition, bicollared means $S$ admits a map $h\co S\times [-1,1] \to M$ that is a homeomorphism onto a neighbourhood of $S$ in $M$ such that $h(x,0)=x$ for every $x \in S$ and $h(S\times [-1,1])\cap \bound M = h(\bound S\times [-1,1])$.
The surface $S$ being bicollared in orientable $M$ implies $S$ is orientable. 
The second condition is equivalent to saying that there are no compression discs for the surface (cf.~\cite[Lemma 6.1]{Hempel-manifolds-1976}). Hatcher~\cite[Lemma 1.10]{Hatcher-notes} implies that if each boundary component of the essential surface $S$ lies on a torus boundary component of $M$, then $S$ is both incompressible and $\partial$--incompressible.

See \cite{Shalen-representations-2002} for a discussion of the definition and its ramifications. A frequently used alternative definition is that an orientable surface $S$ properly embedded in $M$ is said to be {essential} if it is incompressible, boundary-incompressible, and not boundary-parallel. 
A compact, irreducible 3--manifold that contains an essential surface is called \textbf{Haken}.


\subsection{Splittings and surfaces}

Assume $M$ is compact and orientable, but not necessarily irreducible.
Stallings~\cite{Stallings-topological-1965} previously showed that a free-product decomposition of a fundamental group of a 3--manifold gives rise to a system of essential 2--spheres in the manifold and Epstein~\cite{Epstein-free-1961} and Waldhausen~\cite{Waldhausen-gruppen-1967} previously showed that a decomposition of the fundamental group as a free product with amalgamation or HNN extension gives rise to a system of incompressible surfaces. 
\Cref{pro:stabilisers} states that an ideal point of a curve in the variety of characters of $\pi_1(M)$ gives a non-trivial splitting of $\pi_1(M).$ 

As a consequence of this result, Culler and Shalen~\cite[Proposition 2.3.1]{Culler-Shalen-varieties-1983} distill the earlier work of Stallings~\cite{Stallings-topological-1965}, Epstein~\cite{Epstein-free-1961} and Waldhausen~\cite{Waldhausen-gruppen-1967} into the following result.

\begin{proposition}[Culler-Shalen] \cite[Proposition 2.3.1]{Culler-Shalen-varieties-1983}
\label{pro:CSsplittingsurface}
Let $M$ be a compact, orientable 3--manifold. For any non-trivial splitting of $\pi_1(M)$ there exists a non-empty system $S = S_1 \cup \ldots \cup S_m$ of orientable incompressible surfaces in $N,$ none of which is boundary parallel, such that $\im(\pie(S_i) \to \pi_1(M))$ is contained in an edge group for $i = 1, \ldots, m,$ and 
$\im(\pie(R) \to \pie(M))$ is contained in a vertex group for each connected component of $M\setminus S.$ 

Moreover, if $K \subset \partial M$ is a subcomplex such that $\im(\pie(K) \to \pie(M))$ is contained in a vertex group for each connected component of $K$, we may take $S$ to be disjoint from $K.$
\end{proposition}


\subsection{Detecting essential surfaces}

From now onwards, we assume that $M$ is irreducible. Following Culler and Shalen, we now obtain finer information on surfaces detected by the ideal point $\xi$ of a curve $C$ in the variety of characters. We outline the general approach and refer the reader again to \cite{Culler-Shalen-varieties-1983, Shalen-representations-2002} for a detailed account relevant to this paper.

The action by $\pie(M)$ on the Bass-Serre tree $\tree=\tree_\xi$ is used to find an essential surface in $M$ via the following construction that was described by Stallings~\cite{Stallings-topological-1965}. Given a triangulation on $M$, lift the triangulation to the universal cover $\tilde{M}$ so that $\pie(M)$ acts simplicially on $\tilde{M}$. Construct a simplicial, $\pie(M)$--equivariant map $f\colon \tilde{M}\to\tree$. Denote the set of midpoints of edges of $\tree$ to be $E$. Then $f^{-1}(E)$ is a surface $\tilde{S}$ in $\tilde{M}$ that descends to a surface $S$ in $M$.

The surface associated with the ideal point $\xi$ depends on the choice of triangulation of $M$ and the choice of map $f$. A surface found in this way is called \textbf{dual to the action} of $\pi_1(M)$ on $\tree$. 
 The surface may contain finitely many parallel copies of some of its components, which we implicitly discard. The definition of a surface being dual to the action combined with \Cref{pro:CSsplittingsurface} has the following consequence.

\begin{corollary}
If $S$ is a surface dual to the action of $\pi_1(M)$ on $\tree$, then,
\begin{enumerate}[label=(\roman*)]
	\item for each component $M_i$ of $M-S$, the subgroup $\text{im}(\pi_1(M_i)\to \pi_1(M))$ of $\pi_1(M)$ is contained in the stabiliser of a vertex of $\tree$; and 
	\item for each component $S_j$ of $S$, the subgroup $\text{im}(\pi_1(S_j)\to \pi_1(M))$ of $\pi_1(M)$ is contained in the stabiliser of an edge of $\tree$.
\end{enumerate}
\end{corollary}

If the surface $S$ dual to the action is not essential we can find a map homotopic to $f$ such that it is (see \cite{Culler-Shalen-varieties-1983, Shalen-representations-2002}). 
An essential surface that arises in this way is said to be \textbf{detected} by an ideal point. 


\subsection{Boundary slopes and essential surfaces}

A \textbf{boundary slope} is the slope of the boundary curve of an essential surface $S$ in $M$ that has non-empty intersection with $\bound M$. 
The relationship between boundary slopes of essential surfaces in $M$ and valuations of trace functions of peripheral elements is given in the following reformulation of \cite[Proposition 1.3.9]{Culler-Dehn-1987}.

\begin{proposition}
Let $M$ be a compact, orientable, irreducible 3--manifold with $\partial M$ a torus. 
Suppose there is an element $\gamma \in \pi_1(M)$ such
that $v_\xi(I_\gamma) < 0.$
  Let 
  $\alpha \in \im (\pi_1(T) \to \pi_1(M))$ such that
  $v_\xi (I_\alpha) \ge 0.$  Then either
  \begin{enumerate}
    \item $v_\xi (I_\beta) \ge 0$ for all
      $\beta \in \im (\pi_1(T) \to \pi_1(M))$ and there is a
      closed essential surface in $M$, or
    \item $\alpha$ determines a boundary slope of $M$.
  \end{enumerate}
\end{proposition}
\begin{proof}
Since $v_\xi (I_\alpha) \ge 0,$ we know that $\alpha$ stabilises a vertex in the associated Bass-Serre tree. Thus, there is an
associated essential surface $S$ disjoint from any curve $C$ on
$T \subseteq \partial M$ representing $\alpha$. If $\partial S \ne \emptyset$ then
$C$ is a boundary slope. Otherwise $\im (\pi_1(\partial M)\to \pi_1(M))$ stabilises a vertex and the first part
is true.
\end{proof}

\begin{corollary}
\label{pro:detected_surface}
Let $M$ be a compact, orientable, irreducible 3--manifold with $\partial M$ a torus. Let $C$ be a curve in the variety of characters with an ideal point $\xi$.
We have the following mutually exclusive cases.
\begin{enumerate}
  \item If there is an element $\gamma$ in $\im (\pi_1(\partial M) \to \pi_1(M))$
     such that $v_\xi(I_\gamma)<0$, then up to inversion there
     is a unique primitive element $\alpha \in \im(\pi_1(\partial M)\to \pi_1(M))$ such that 
     $v_\xi(I_\alpha)\ge 0$. Then every essential surface $S$ detected by $\xi$ has non-empty boundary and 
     $\alpha$ is parallel to its boundary
     components.  In this case, $\alpha$ is called a \textbf{strongly detected boundary slope}.
  \item If $v_\xi(I_\gamma)\ge 0$ for all 
     $\gamma \in \im (\pi_1(\partial M)\to\pi_1(M))$, 
     then an essential surface $S$ detected by $\xi$ may be chosen that is disjoint from $\partial M$. Moreover, if $S$ is detected by $\xi$ and has non-empty boundary, then its boundary slope is \textbf{weakly detected} by $\xi.$
\end{enumerate}
\end{corollary}


\subsection{Remarks on detected surfaces}

Not all essential surfaces can be detected by the $\SLC$--character variety. Much work has been done to further characterise which surfaces are detected by which ideal points (for example, see \cite{Chesebro-closed-2013,Dunfield-incompressibility-2012,Segerman-detection-2011,Tillus-character-2004,Tillus-degenerations-2012, Yoshida-ideal-1991}). 

This paper gives an example of a closed, Haken, hyperbolic 3--manifold with no essential surface detected in any characteristic (\Cref{sec:m019}). Essential surfaces in closed hyperbolic 3--manifolds have genus at least two. Forthcoming work~\cite{Garden-essential-2024} of the authors with Ben Martin shows that there are essential tori in infinitely many graph manifolds and infinitely many Seifert fibered spaces that are not detected by any ideal point of the variety of $\SLF$--characters for any algebraically closed field $\F.$

In contrast, the theory has been extended to $\SLnC$--character varieties \cite{Hara-character-2021}, and it was shown that for  every essential surface in a 3--manifold there is some $n$ with the property that the surface is detected by the $\SLnC$--character variety~\cite{Friedl-representation-2018}. 

\subsection{Bounds on dimension}
\label{sec:dim}

Based on Thurston~\cite[Theorem 5.6]{Thurston-notes}, Culler and Shalen~\cite[Proposition 3.2.1]{Culler-Shalen-varieties-1983} provided a bound on the dimension of the $\SLC$--character variety. The proofs apply almost verbatim in our context.

\begin{lemma}~\cite[Lemma 1.5.1]{Culler-Shalen-varieties-1983}
	\label{lem:irrTraceNotTwo}
	Let $\F$ be an algebraically closed field of characteristic $p\neq2$ and let $\Gamma$ be a finitely generated group. Suppose $\rho\co \Gamma \to \SL_2(\F)$ is an irreducible representation and $\alpha \in \Gamma$ such that $\rho(\alpha) \neq  \pm E$. Then there exists $\gamma \in \Gamma$ such that the restriction of $\rho$ to the subgroup generated by $\alpha$ and $\gamma$ is irreducible and $\tr(\rho(\gamma))^2\neq4$.
\end{lemma}

\begin{proof}
	Take $\alpha \in \Gamma$ such that $\rho(\alpha) \neq  \pm E$. As $\rho$ is irreducible, by definition there exists $\gamma\in\Gamma$ such that $\rho$ restricted to the subgroup generated by $\alpha$ and $\gamma$ is irreducible. 
	We show $\gamma$ can be chosen so that $\tr(\rho(\gamma))^2\neq4$. 
	
	Suppose $\gamma_0\in\Gamma$ is such that $\rho$ restricted to the subgroup generated by $\alpha$ and $\gamma_0$ is irreducible and $\tr(\rho(\gamma_0))^2=4$. 
	Then $\rho(\gamma_0)\neq\pm E$ and we may assume $\rho(\gamma_0)$ is parabolic.
	Up to conjugation we may assume
	\begin{equation}
	\label{eqn:irrTraceNotTwo}
		\rho(\gamma_0)=\pm\begin{pmatrix} 1 & 1 \\ 0 & 1 \end{pmatrix},\quad
		\rho(\alpha)=\begin{pmatrix} a & 0 \\ c & a^{-1} \end{pmatrix}, \quad c\neq 0
	\end{equation}
	Then for all $n\in\Z$ we have $\tr(\rho(\alpha \gamma_0^{n}))=(\pm 1)^n\left(a+a^{-1}+ cn\right)$.
	
	For characteristic $3$, there are only two options for $n=\pm1$. We have for $n=1$, 
	\[
		\tr(\rho(\alpha \gamma_0))=(\pm 1)\left(a+a^{-1}+ c\right)
	\] 
	and for $n=-1$, 
	\[
		\tr(\rho(\alpha \gamma_0^{-1}))=(\pm 1)\left(a+a^{-1}- c\right)
	\]
	If at least one of $\tr(\rho(\alpha \gamma_0^n))\neq\pm1$ then take $\gamma=\alpha\gamma_0^{n}$. Otherwise we have $a+a^{-1}=0$ and $c^2=1$.
	In this case, consider instead the commutator 
	\[
		\gamma = \alpha \gamma_0 \alpha^{-1} \gamma_0^{-1} = \begin{pmatrix} 1 - a c & a c \\ -1 & -1  + a c \end{pmatrix}
	\]
	 Then $\tr(\rho(\gamma))= 0 \neq \pm1$. It remains to prove that $\rho(\alpha)$ and $\rho(\gamma)$ form an irreducible pair. 
	We calculate $\rho([\alpha,\gamma])=-E$ and $\tr (\rho([\alpha,\gamma])) = 1\neq2\Mod{3}$. By \Cref{lem:reducible_commutator2}, the restriction of $\rho$ to the subgroup generated by $\alpha$ and $\gamma$ is irreducible and $\gamma$ satisfies the requirements. 
	Note, in characteristic $3$ this only works assuming $c^2=1$ and $a+a^{-1}=0$, otherwise the conditions may not be satisfied. 
\end{proof}

The issue with the result of \Cref{lem:irrTraceNotTwo} in characteristic $2$ is that we only have one option for $n\in\Z_p\setminus\{0\}$, $n=1$, to get combinations $\alpha \gamma_0^n$. 
In order to refine the result for characteristic $2$ we use the infinite dihedral group $\D \cong \Z_2 \ast \Z_2 \cong \Z \rtimes \Z_2 = \langle s, t\mid s^2=1,\ sts^{-1}=t^{-1}\rangle$.
We say a group $G$ is \textbf{$\D$--like} if it is isomorphic to the generalised dihedral group $\Dn$, written
\begin{equation*}
\begin{split}
	\Dn 	&=\Dih\cong \Z^n \rtimes \Z_2 \\ 
		&= \langle s, t_1,\ldots t_n \mid s^2=1,\ st_is^{-1}=t_i^{-1},\ t_it_j=t_jt_i\ \text{ for all } i,j=1,\ldots,n\rangle
\end{split}
\end{equation*}
 for some $n\in\N$.
	
\begin{lemma}
	\label{lem:irrTraceTwo}
	Let $\F$ be an algebraically closed field of characteristic $2$ and let $\Gamma$ be a finitely generated group.
	Suppose $\rho\co \Gamma \to \SL_2(\F)$ is an irreducible representation and $\alpha \in \Gamma$ such that $\rho(\alpha) \neq E$. 
	Assume $\im(\rho)$ is not $\D$--like.
	Then there is $\gamma\in \Gamma$ such that the restriction of $\rho$ to the subgroup generated by $\alpha$ and $\gamma$ is irreducible and $\tr(\rho(\gamma))\neq0$.
\end{lemma}

\begin{proof}
	We follow the same initial steps as \Cref{lem:irrTraceNotTwo}.
	
	Take $\alpha \in \Gamma$ such that $\rho(\alpha) \neq E$. As $\rho$ is irreducible, by definition there exists $\gamma_0\in\Gamma$ such that $\rho$ restricted to the subgroup generated by $\alpha$ and $\gamma_0$ is irreducible. 
	Denote a generating set of $\Gamma$ by $\{ \gamma_1, \ldots \gamma_n\}$. 
	Without loss of generality, we can consider $\gamma_0$ to be one of the generators, say $\gamma_0=\gamma_1$.
	
	If $\tr(\rho(\gamma_1))\neq0$ then we are done.
	Assume that $\tr(\rho(\gamma_1))=0$. Then up to conjugation we may assume
	\begin{equation}
	\label{eqn:irrTraceTwo}
		\rho(\gamma_1)=\begin{pmatrix} 1 & 1 \\ 0 & 1 \end{pmatrix},\quad
		\rho(\alpha)=\begin{pmatrix} a & 0 \\ c & a^{-1} \end{pmatrix}, \quad c\neq 0
	\end{equation}
	
	Consider $\tr(\rho(\alpha \gamma_1))=a+a^{-1}+c$ and
	assume $\im(\rho)$ is not $\D$--like. 
	If we have $\tr(\rho(\alpha \gamma_1))\neq0$ then we can again take $\gamma=\alpha\gamma_1$ and we are done.
	If $\tr(\rho(\alpha \gamma_1))=0$ then $c=a+a^{-1}$ and 
	\[
		\im\left(\rho\left(\langle \alpha, \gamma_1 \rangle \right)\right) = \left \langle \begin{pmatrix} 1 & 1 \\ 0 & 1 \end{pmatrix},  \begin{pmatrix} a & 0 \\ a+a^{-1} & a^{-1} \end{pmatrix} \right \rangle \cong \D
	\] 
	where the isomorphism is constructed by sending $s$ to the first matrix, $\rho(\gamma_1)$, and $t$ to the second matrix, $\rho(\alpha)$, for $s$ and $t$ the generators of $\D$.
		
	Then $\im\left(\rho\left(\langle \alpha, \gamma_1 \rangle \right)\right)$ cannot be the whole image, otherwise this would contradict our assumption that $\im(\rho)$ is not $\D$--like. Therefore there is an element $\delta\in\Gamma$ such that $\rho(\delta)\notin \im\left(\rho\left(\langle \alpha, \gamma_0 \rangle \right)\right)$. Say without loss of generality $\delta=\gamma_2$,
	\begin{equation*}
		\rho(\gamma_2)=\begin{pmatrix} e & f \\ g & h \end{pmatrix},\quad eh-fg=1
	\end{equation*}

	We get $\tr(\rho(\gamma_2\alpha\gamma_2^{-1}\alpha^{-1}))=(a+a^{-1})^2f(e+f+g+h)$. 
	Assume the restriction of $\rho$ to the subgroup generated by $\alpha$ and $\gamma_2$ is reducible. Then $(a+a^{-1})^2f(e+f+g+h)=0$. 
	Given $\tr(\rho(\alpha))\neq0$, we must have $f=0$ or $f=e+g+h$. 
	A simple calculation shows the restriction of $\rho$ to the subgroup generated by $\alpha$ and $\gamma_1\gamma_2$ is irreducible with trace $\tr(\rho(\gamma_1\gamma_2))=e+g+h$. 
	If $\tr(\rho(\gamma_1\gamma_2))=0$ then $f=0$ and $g=e+e^{-1}$, implying
	\[
		\im\left(\rho\left(\langle \alpha, \gamma_1, \gamma_2 \rangle \right)\right) = \left \langle \begin{pmatrix} 1 & 1 \\ 0 & 1 \end{pmatrix},  \begin{pmatrix} a & 0 \\ a+a^{-1} & a^{-1} \end{pmatrix},  \begin{pmatrix} e & 0 \\ e+e^{-1} & e^{-1} \end{pmatrix} \right \rangle \cong D_{\infty}^2
	\] 
	where the isomorphism is constructed by sending $s$ to the first matrix, $\rho(\gamma_1)$, $t_1$ to the second matrix, $\rho(\alpha)$, and $t_2$ to the third matrix, $\rho(\gamma_2)$, for $s$, $t_1$, and $t_2$ the generators of $D_{\infty}^2$.

	We can repeat this process over each generator $\gamma_i$. 
	Given our assumption that $\im(\rho)$ is not $\D$--like, there must be a generator $\gamma_k$, $2\leq k\leq n$ such that $\rho(\gamma_k)\notin \im\left(\rho\left(\langle \alpha, \gamma_1 \rangle \right)\right)$ and either the restriction of $\rho$ to the subgroup generated by $\alpha$ and $\gamma_k$ is reducible and $\tr(\rho(\gamma_1\gamma_k))\neq0$ or the restriction of $\rho$ to the subgroup generated by $\alpha$ and $\gamma_k$ is irreducible.
	Again let
	\begin{equation*}
		\rho(\gamma_k)=\begin{pmatrix} e & f \\ g & h \end{pmatrix},\quad eh-fg=1
	\end{equation*}
	
	In the first case take $\gamma=\gamma_1\gamma_k$. 
	In the second case consider the options $\gamma=\gamma_k$, $\gamma=\gamma_1\gamma_k$ or $\gamma=\alpha\gamma_k$. The restriction of $\rho$ to the subgroup generated by $\alpha$ and $\gamma$ is irreducible for each of these choices of $\gamma$. 
	We have $\tr(\rho(\gamma_k))=e+h$, $\tr(\rho(\gamma_1\gamma_k))=e+g+h$ and $\tr(\rho(\alpha\gamma_k))=ae+(a+a^{-1})f+a^{-1}h$. 
	If $\tr(\rho(\gamma_k))=\tr(\rho(\gamma_1\gamma_k))=\tr(\rho(\alpha\gamma_k))=0$ then $h=e$, $g=0$ and $f=e$. Combining with the determinant condition this would give $\rho(\gamma_k)=\rho(\gamma_1)$, which contradicts $\rho(\gamma_k)\notin \im\left(\rho\left(\langle \alpha, \gamma_1 \rangle \right)\right)$. 
	Thus one of $\tr(\rho(\gamma_k))$, $\tr(\rho(\gamma_1\gamma_k))$, $\tr(\rho(\alpha\gamma_k))$ is nonzero and we can choose $\gamma$ accordingly. 
\end{proof}

In fact, the converse statement of \Cref{lem:irrTraceTwo} is true if we also assume $\tr(\rho(\alpha)) \neq 0$.

\begin{lemma}
	Let $\F$ be an algebraically closed field of characteristic $2$ and let $\Gamma$ be a finitely generated group.
	Suppose $\rho\co \Gamma \to \SL_2(\F)$ is an irreducible representation and $\alpha \in \Gamma$ such that $\rho(\alpha) \neq E$ and $\tr(\rho(\alpha)) \neq 0$. 
	Assume $\gamma\in\Gamma$ such that the restriction of $\rho$ to the subgroup generated by $\alpha$ and $\gamma$ is irreducible and $\tr(\rho(\gamma))\neq0$. 
	Then $\im(\rho)$ is not $\D$--like.
\end{lemma}

\begin{proof}
	Assume for the purposes of a contradiction that $\im(\rho)$ is $\D$--like and assume first $\im(\rho) \cong \D = \langle s, t\mid s^2=1,\ sts^{-1}=t^{-1}\rangle$.
	
	As $\rho$ is irreducible there exists some $\gamma \in\Gamma$ such that the restriction of $\rho$ to the subgroup generated by $\alpha$ and $\gamma$ is irreducible.
	Assign $\mu, \nu\in \Gamma$ to be the elements such that $\rho(\mu)$ and $\rho(\nu)$ are mapped to $s,t$ respectively in the isomorphism. Then the restriction of $\rho$ to the subgroup generated by $\mu$ and $\nu$ must be irreducible and we can conjugate their image to
	\begin{equation*}
		\rho(\mu)=\begin{pmatrix} x & 1 \\ 0 & x^{-1} \end{pmatrix},\quad
		\rho(\nu)=\begin{pmatrix} y & 0 \\ u & y^{-1} \end{pmatrix}, \quad u\neq 0.
	\end{equation*}
	Then the isomorphism forces $\rho(\mu)^2=E$ and $\rho(\mu\nu\mu^{-1}\nu)=E$, which implies $x=1$ and $u=y+y^{-1}$.
	
	We will see that any element $\delta \in \Gamma$ then either has trace $0$ or the restriction of $\rho$ to the subgroup generated by $\nu$ and $\delta$ is reducible.
	Given $\tr(\rho(\alpha))\neq0$, the restriction of $\rho$ to the subgroup generated by $\nu$ and $\alpha$ must be reducible and the only option for an element $\gamma$ that is irreducible with $\alpha$ must have $\tr(\rho(\gamma))=0$. 
	
	To see all elements must be one of these two options, take an arbitrary element $\delta \in \Gamma$.
	Then 
	\[ 
	\rho(\delta)=\rho(\nu)^{k_1}\rho(\mu)^{l_1}\ldots \rho(\nu)^{k_n}\rho(\mu)^{l_n}
	\]
	for some sequences $k_i, l_i\in\Z$, $i=1,\ldots,n$. Without loss of generality, assume $k_i\neq 0$ for all $i\in\{2,\ldots,n\}$ and $l_i\neq 0$ for all $i\in\{1,\ldots,n-1\}$. The terms $\rho(\mu)^{l_i}$ are defined by the parity of $l_i$, so we can take $l_i=1$ for all $i\in\{1,\ldots,n-1\}$.
	Then the terms in $\rho(\delta)$ satisfy
	\[
	\begin{split}
		\rho(\nu)^{k_i}\rho(\mu) 					&= \begin{pmatrix} y^{k_i} 		& y^{k_i} \\  y^{k_i}+y^{-{k_i}} 	& y^{k_i} \end{pmatrix}, \\
		\rho(\nu)^{k_i}\rho(\mu)\rho(\nu)^{k_j}\rho(\mu)	&= \begin{pmatrix} y^{k_i-k_j} 	& 0 \\  y^{k_i-k_j}+y^{-k_i+k_j} 	& y^{-k_i+k_j} \end{pmatrix}.
	\end{split}
	\]
	We can deduce
	\begin{equation}	
	\label{eqn:delta1}
		\rho(\delta) 	= \begin{pmatrix} y^{k_1-k_2+\ldots -k_{n-1}+k_n} & 0 \\  y^{k_1-k_2+\ldots -k_{n-1}+k_n}+y^{-k_1+k_2-\ldots +k_{n-1}-k_n} & y^{-k_1+k_2-\ldots +k_{n-1}-k_n} \end{pmatrix} \begin{pmatrix} 1 & 1 \\  0 & 1 \end{pmatrix}^{l_n}
	\end{equation}	
	for $n$ odd or
	\begin{equation}		
	\label{eqn:delta2}
		\rho(\delta) 	= \begin{pmatrix} y^{k_1-k_2+\ldots +k_{n-1}-k_n} & y^{k_1-k_2+\ldots +k_{n-1}-k_n} \\  y^{k_1-k_2+\ldots +k_{n-1}-k_n}+y^{-k_1+k_2-\ldots -k_{n-1}+k_n} & y^{k_1-k_2-\ldots +k_{n-1}-k_n} \end{pmatrix} \begin{pmatrix} 1 & 1 \\  0 & 1 \end{pmatrix}^{l_n}
	\end{equation}	
	for $n$ even.
	
	There are four cases.
	If $n$ is odd and $l_n=1\Mod{2}$ or if $n$ is even and $l_n=0\Mod{2}$ then $\tr(\rho(\delta))=0$. 	
	If $n$ is odd and $l_n=0\Mod{2}$ or if $n$ is even and $l_n=1\Mod{2}$ then the restriction of $\rho$ to the subgroup generated by $\nu$ and $\delta$ is reducible.

	 Assume now 
	 \begin{equation*}
	 \begin{split}
	 	\im(\rho) 	&\cong D_{\infty}^q \\
				&= \langle s, t_1, \ldots t_q \mid s^2=1,\ st_is^{-1}=t_i^{-1},\ t_it_j=t_jt_i\  \text{ for all }\ i,j=1\ldots q\rangle
	\end{split}
	\end{equation*}
	 Assign $\mu, \nu_i \in \Gamma$, $i=1,\ldots,q$ to be the elements such that $\rho(\mu)$ and $\rho(\nu_i)$ are mapped to $s,t_i$ respectively in the isomorphism. Then  similar to the case for $q=1$, we can conjugate their image to,
	\begin{equation*}
		\rho(\mu)=\begin{pmatrix} 1 & 1 \\ 0 & 1 \end{pmatrix},\quad
		\rho(\nu_i)=\begin{pmatrix} y_i & 0 \\ y_i+y_i^{-1} & y_i^{-1} \end{pmatrix}, \quad \text{ for all } i=1,\ldots, q.
	\end{equation*}
	Again we will get that all elements $\delta\in\Gamma$ must be of one of the forms \Cref{eqn:delta1,eqn:delta2}, with the entries now in terms of combinations of powers of $y_1,\ldots, y_q$. 
	The same case analysis will show that any element $\delta \in \Gamma$ then either has trace $0$ or the restriction of $\rho$ to the subgroup generated by $\nu_i$ and $\delta$ is reducible for all $i=1,\ldots,q$.
	Given $\tr(\rho(\alpha))\neq0$, the restriction of $\rho$ to the subgroup generated by $\nu_i$ and $\alpha$ must be reducible for all $i=1,\ldots,q$ and the only option for an element $\gamma$ that is irreducible with $\alpha$ must have $\tr(\rho(\gamma))=0$, which gives the contradiction.
\end{proof}

We also use the following result relating the dimension of components of $\R(\Gamma,\F)$ and $\X(\Gamma,\F)$.

\begin{lemma}~\cite[Lemma 1.5.3]{Culler-Shalen-varieties-1983}
\label{lem:varDim}
	Let $\F$ an algebraically closed field of characteristic $p$ and let $\Gamma$ be a finitely generated group.
	If $R_0 \subseteq R(\Gamma,\F)$ is an irreducible component containing an irreducible representation $\rho_0 \in \R_0$, and if $X_0=\strn(R_0)$, then
\[
	\dim(X_0)=\dim(R_0)-3.
\]
\end{lemma}

\begin{proof}
We follow the proof of~\cite[Lemma 1.5.3]{Culler-Shalen-varieties-1983} verbatim. 
	By \Cref{cor:redChar}, the set of reducible representations has the form $\strn^{-1}(X^{red}_n)$ for the closed algebraic set $X^{red}_n\subseteq\F^d$.
	
	Set $U=X_0\cap(\F^d-X^{red}_n)$. For each $\chi\in U$, $\strn^{-1}(\chi)$ consists of irreducible representations. 
	By \Cref{cor:irreducible_orbit_closures}, $\strn^{-1}(\chi)$ must be a single equivalence class of representations.
	
	Fix $\rho\in\strn^{-1}(\chi)$. Define $\sigma_{\rho}\co \SLF \to \strn^{-1}(\chi)$ by $\sigma_{\rho}(A)=A\rho A^{-1}$. This is a covering map and for characteristic $p\neq2$ this will be a two-sheeting covering map. 
	This is because an irreducible representation has centraliser $\{\pm E\}$ in characteristic $p\neq2$ and $\{E\}$ in characteristic $2$.
	
	Then $\dim(\strn^{-1}(\chi))=\dim(\SLF) = 3$ for every point $\chi$ in a Zariski open set $U\subset X_0$ and thus $\dim(\strn^{-1}(X_0))$ $=$ $3$. We also have $\dim(\strn^{-1}(X_0))=\dim(R_0)-\dim(X_0)$. Together these give the result.
\end{proof}

We combine \Cref{lem:irrTraceNotTwo,lem:irrTraceTwo,lem:varDim} as in \cite[Proposition 3.2.1]{Culler-Shalen-varieties-1983} to get the following bound on dimension initially due to Thurston now over $\F$.

\begin{proposition}[Thurston]~\cite[Proposition 3.2.1]{Culler-Shalen-varieties-1983}
\label{pro:dimSL_Thurston}
Let $M$ be a compact orientable 3--manifold and $\F$ an algebraically closed field of characteristic $p$.
Let $\rho_0 \co \pi_1(M)  \to \SLF$ be an irreducible representation such that for each torus component $T$ of $\partial M$, $\rho_0(\im( \pi_1(T) \to \pi_1(M))) \not\subseteq \{ \pm E\}.$ 
In characteristic $2$, also assume that $\im(\rho_0)$ is not $\D$--like. 

Let $R_0$ be an irreducible component of $R(M)$ containing $\rho_0$. Then $X_0 = \strn(R_0)$ has dimension $\ge s - 3\chi(M),$ where $s$ is the number of torus components of $\partial M.$
\end{proposition}

\begin{proof}
	We again almost follow the proof of~\cite[Proposition 3.2.1]{Culler-Shalen-varieties-1983} verbatim, with the exception of characteristic $2$. 
	By \Cref{lem:varDim}, $\dim(X_0)\geq s-3\chi(M)$ if and only if  $\dim(R_0)\geq s-3\chi(M)+3$. 
	Continue by induction on $s$.
	
	Assume $s=0$. We assume the existence of an irreducible representation $\rho_0$. If $\partial M = \emptyset$ then $\chi(M)=0$ and the result is trivially satisfied. Hence we can assume $\partial M\neq\emptyset$. 
	
	Then $M$ will have the same homotopy type as a finite 2--dimensional CW-complex $K$ with one $0$--cell, $m$ 1--cells, and $n$ 2--cells. Then $\pi_1(M)$ has a presentation
	\[
		\pi_1(M)=\langle g_1, \ldots, g_m \mid r_1,\ldots r_n \rangle
	\]
	with $\chi(M)=\chi(K)=1-m+n$.
	
	Define $f$ to be the regular map 
	\begin{equation*}
	\begin{split}
		f \co \SLF^m &\to \SLF^n \\ 
		\left(x_1,\ldots,x_m\right)&\mapsto\left(r_1\left(x_1,\ldots,x_m\right),\ldots,r_m\left(x_1,\ldots,x_m\right)\right)
	\end{split}
	\end{equation*}
	 Then $R(M,\F)=f^{-1}(E,\ldots,E)$ and $R(M,\F)$ is the inverse image of a point under a regular map from a $3m$--dimensional affine variety to a $3n$--dimensional affine variety with irreducible components. This gives
	\[
		\dim(R_0)\geq 3m-3n = -3\chi(M)+3
	\]
	which proves the result for $s=0$.
	
	Consider $s>0$. Let $T$ be a torus component of $\partial M$. 
	Then $\rho_0(\im( \pi_1(T) \to \pi_1(M))) \not\subseteq \{ \pm E\}$ means there exists $\alpha\in\pi_1(M)$ represented by a simple closed curve on $T$ such that $\rho_0(\alpha)\neq \pm E$. In characteristic $2$, we also required $\im(\rho_0)$ is not $\D$--like.
	Then by \Cref{lem:irrTraceNotTwo} and \Cref{lem:irrTraceTwo} there exists $\gamma\in\pi_1(M)$ such that $\rho_0$ restricted to the subgroup generated by $\alpha$ and $\gamma$ is irreducible and $\tr(\rho_0(\gamma))^2\neq 4$.
 	
	We can construct a manifold $M'$ with one less torus boundary component than $M$ by removing a neighbourhood of an embedded loop on $T$ that represents $\gamma\in\pi_1(M)$. Then $T$ will become a genus two boundary component $T'$. Our original manifold $M$ can be obtained from $M'$ by adding a $2$--handle to $T'$. 
	
	We may choose a standard basis $\alpha', \beta', \gamma', \delta'$ of $\pi_1(T')$ such that $\alpha', \gamma'$ are mapped to $\alpha, \gamma$ under the natural surjection $i_*\co\pi_1(M')\to\pi_1(N)$. The $2$--handle is attached to $T'$ along a simple closed curve that represents $\delta'$. We write $\sigma=[\alpha', \beta']=[\gamma', \delta']$.
	
	Define a representation $\rho'_0=\rho_0\circ i_* \in R(M',\F)$. This is irreducible, as $\rho_0$ is irreducible and $i_*$ is surjective.
	Since the kernel of $i_*$ is the normal closure of $\delta'$, we can identify $R(M,\F)$ with $W \subset R(M',\F)$ the closed algebraic set of representations $\rho$ such that $\rho(\delta')=E$. Let $R'_0$ be an irreducible component of $R(M',\F)$ containing $\rho_0'$. Then $R_0$ is an irreducible component of $R'_0\cap W$.
	
	By the induction hypothesis, $M'$ has $s-1$ torus components and 
	\[
		\dim(R'_0)\geq (s-1)-3\chi(M')+3=s-3\chi(M)+5
	\]

	Let $g\co R'_0\to\F^2$ be defined by $g(\rho)=\left( \tr(\rho(\delta')), \tr(\rho(\sigma)) \right)$. Then $g(R_0) = (2,2)$ and $R_0\subset g^{-1}(2,2)$. 
	
	Let $Y$ be the set of all representations $\rho \in R'_0$ whose restriction to the subgroup in $\pi_1(M')$ generated by $\alpha'$, $\gamma'$ is reducible. By \Cref{cor:redChar}, $Y$ is a closed algebraic set. 
	Let $Z$ be the closed algebraic set consisting of all $\rho\in\R'_0$ such that $\tr(\rho(\gamma'))=\pm2$. Clearly, $\rho_0\notin Y\cup Z$.
	
	Take $\rho\in g^{-1}(2,2)$. Then $\tr(\rho(\delta'))= \tr(\rho(\sigma))=2$. By \Cref{lem:reducible_commutator2}, the restriction of $\rho$ 
to the subgroups generated by $\alpha', \beta'$ and by $\gamma', \delta'$ are reducible and $\rho(\sigma)$ has an eigenvector in common with $\rho(\alpha'), \rho(\gamma')$. 
	If $\rho(\sigma)\neq E$, then it has a unique, 1--dimensional invariant subspace and the pair $\rho(\alpha'), \rho(\gamma')$ is reducible. Then $\rho\in Y$. 
	If $\rho(\sigma) = E$, then $\rho(\gamma')$ and $\rho(\delta')$ commute. Since $\tr(\rho(\delta'))=2$, we either have $\rho(\delta')=E$, which implies $\rho\in W$, or $\tr(\rho(\gamma'))=\pm2$, which implies $\rho\in Z$.
	
	Putting this all together, $g^{-1}(2,2)\subset (W\cap R'_0)\cup Y \cup Z$.  
	Then
	\[
	\begin{split}
		\dim(R_0)	&=\dim(R'_0)-\dim(\F^2) \\
				&\geq s-3\chi(M)+3
	\end{split}
	\]
	This finishes the proof.
\end{proof}


\section{$A$--polynomials and eigenvalue varieties}
\label{sec:APoly}

The $A$--polynomial was introduced in~\cite{Cooper-Plane-1994} over $\SLC$, with an alternative, equivalent, definition presented in \cite{Cooper-polynomial-1996} using the eigenvalue map and eigenvalue varieties. We define the $A$--polynomials and eigenvalue varieties over $\F$ and revise the theory in this setting for both a single boundary component following following \cite{Cooper-polynomial-1996} and multiple boundary components following \cite{Tillus-boundary-2005}. We compare the $A$--polynomial in arbitrary characteristic $p$ to the standard $A$--polynomial over $\C$.


\subsection{A single boundary component}
\label{sec:APolySingle}

Consider manifold $M$ with boundary a torus $T$. The $A$--polynomial was introduced in~\cite{Cooper-Plane-1994} over $\SLC$. We can define the $A$--polynomial and eigenvalue variety over a general field $\F$. We have fundamental group 
\[ 
	\pi_1(M)=\langle \gamma_1, \ldots \gamma_n \mid r_1, \ldots, r_m \rangle
\]
Call generators for $\pi_1(T)$ a meridian $\M$ and longitude $\L$ and identify them with their images in $\pi_1(M)$ under the inclusion homomorphism.
Let $\R_U(M,\F)$ be the subvariety of $\R(M,\F)$ defined so that the lower left entries of $\rho(\M)$ and $\rho(\L)$ are zero,
\begin{equation*}
\begin{split}
	\R_U(M,\F) 	= \left\{ \rho\co \pi_1(M)\to\SLF \,\Big|\, \rho(\M)=\begin{pmatrix} m & \ast \\ 0 & m^{-1} \end{pmatrix}, \, \rho(\L)=\begin{pmatrix}  l & \ast \\ 0 & l^{-1} \end{pmatrix} \right\}
\end{split}
\end{equation*}
Every representation in $\R(M,\F)$ is conjugate to such a representation as the peripheral subgroup is abelian.

Define the \textbf{eigenvalue map} to be
\begin{equation*}
\begin{split}
	e\colon\R_U(M,\F)	&\to\left(\F\setminus\{0\}\right)^2\\
	\rho				&\mapsto \left(m,l\right)
\end{split}
\end{equation*}
taking representation $\rho$ to the eigenvalues of $\rho(\M)$ and $\rho(\L)$ corresponding to the common invariant subspace spanned by the first basis vector.

We may assume that $\gamma_1=\M$, $\gamma_2=\L$ in the presentation of the fundamental group. Then 
\begin{equation}
\label{eqn:upperRepVar}
\begin{split}
	\R_U(M,\F) 	&= \left\{\, (A_1, \ldots, A_n)\in\SLF^n \;\Big|\; R_1, \ldots, R_m \text{ and } \right. \\ 
				& \qquad \left. A_1=\begin{pmatrix} m & \ast \\ 0 & m^{-1} \end{pmatrix}, \; A_2=\begin{pmatrix}  l & \ast \\ 0 & l^{-1} \end{pmatrix} \,\right\} \subseteq \F^{4n}
\end{split}
\end{equation}
where $A_i=\rho(\gamma_i)=\begin{pmatrix} a_i & b_i \\ c_i & d_i \end{pmatrix}$ and $R_j = r_j(A_1, \ldots, A_n)$. The entries of $R_j$ are polynomials in $\Z_p[a_1,\ldots,d_n]$ where $m=a_1$ and $l=a_2$. Note we have $c_1=0=c_2$ and the determinant conditions $a_i d_i-b_i c_i=1$. 

Then the map $e$ is precisely the projection
\[
\begin{split}
	e\colon\R_U(M,\F)	&\to\left(\F\setminus\{0\}\right)^2\\
	(a_1,\ldots, d_n) 	&\mapsto \left(a_1,a_2\right)=\left(m,l\right)
\end{split}
\] 

The \textbf{eigenvalue variety} is the closure of the image of this map, $\overline{e(\R_U(M,\F))}$. By \cite{Tillus-boundary-2005}, the dimension of the eigenvalue variety is at most one. We write
\[
	\overline{e(\R_U(M,\F))} = V_0\cup V_1
\]
where $V_0$ is the union of all $0$--dimensional components and $V_1$ is the union of all $1$--dimensional components.

If $V_1$ is non-empty, then it is defined by a principal ideal generated by some polynomial $\tilde{A}_{\F}$ in variables $m$  and $l$. A generator for the radical of the ideal is the \textbf{$A$--polynomial}, written $A_{\F}$.

The $A$--polynomial contains information about the incompressible surfaces in the manifold $M$. In particular, the well-known boundary slopes theorem from~\cite{Cooper-Plane-1994} asserts that the slopes in the Newtown polygon of the $A$--polynomial of $M$ are precisely the slopes of incompressible surfaces in $M$ associated to an action of $\pie(M)$ on a tree $\tau$. Recall the Newtown polygon of a polynomial $f(x,y)$ is the convex hull of integer lattice points $(s,t)$ in the plane that are degrees of monomials $x^sy^t$ that appear in $f(x,y)$.

The $A$--polynomial can be calculated using elimination theory by way of resultants or \textbf{Gr\"obner bases} (see \cite{Cooper-polynomial-1996, Cox-ideal-2015}), the latter of which give a standard basis for an ideal that behaves well with polynomial division.
Once a basis has been chosen for the boundary torus, the generator of the principal ideal $\tilde{A}_{\F}$ and the $A$--polynomial are both well-defined up to multiplication by units in $\F(m^{\pm1},l^{\pm1})$. 

We make some remarks about Gr\"obner bases and how they are calculated. 
\begin{remark}
\label{rem:buchAlg}
	The Gr\"obner basis of a polynomial ideal $I \subset k[x_1,\ldots,x_n]$ is calculated using \emph{Buchberger's algorithm} (see \cite[Theorem 2]{Cox-ideal-2015}). 
	Using the ascending chain condition for ideals we know we can do this using finitely many polynomials. 	
	
	For our purposes we use a modified version of {Buchberger's algorithm} where there is no division by coefficients in $k$. 
	{In particular, if an element has coefficients in a ring $R\subset k$ then they will remain in $R$ under our algorithm.}
	We describe the differences in the versions of the algorithm below.
	
	The main ingredients of Buchberger's algorithm are the \emph{division algorithm} (given in \cite[Theorem 3]{Cox-ideal-2015}) and the {$S$--polynomial}. 
	Given a non-zero polynomial $f\in k[x_1,\ldots x_n]$, we can write $f$ as a sum of monomials
	\[
		f=\sum_\alpha a_\alpha x^\alpha,  \quad a_\alpha \in k
	\]
	for $x=(x_1,\ldots,x_n)$, $\alpha=(\alpha_1,\ldots \alpha_n) \in \Z^n_{\geq0}$, $x^\alpha=x_1^{\alpha_1} \ldots x_n^{\alpha_n}$.
	For a given monomial ordering $>$, we say the multidegree of $f$ is
	\[
		\multi(f)=\max\{\alpha\in\Z^n_{\geq0} \mid a_\alpha\neq0\}
	\]
	Then the \textbf{leading monomial}, \textbf{leading coefficient}, and \textbf{leading term} of $f$ are respectively given by 
	\[
	\begin{split}
		\LM(f)=x^{\multi(f)},\quad \LC(f) = a_{\multi(f)},\quad \LT(f)=a_{\multi(f)}x^{\multi(f)}
	\end{split}
	\]
	
	Given a set of polynomials $f_1,\ldots, f_m$ and a monomial ordering, the \textbf{division algorithm} finds a way to write a polynomial $f$ in terms of $f_1,\ldots, f_m$, 
	\[
		f=a_1f_1+\ldots+a_mf_m+r, \quad a_i,r \in k[x_1,\ldots,x_n]
	\]
	with $r=0$ or no monomial in $r$ being divisible by $\LT(f_i)$ for all $1\leq i \leq m$. The algorithm involves dividing terms by $\LT(f_i)$. For example, in order to find the remainder $r$ in the steps of the algorithm, if $\LT(f_i)$ divides $\LT(f)$ then we take iterations of 
	\begin{equation}
	\label{eqn:divAlg1}
	\begin{split}
		f\mapsto f-\frac{\LT(f)}{\LT(f_i)}f_i
	\end{split}
	\end{equation}
	
	Given a pair of polynomials $f, g\in k[x_1,\ldots x_n]$ the \textbf{$S$--polynomial} of $f$ and $g$ cancels the leading terms in both $f$ and $g$ to get a polynomial of lower degree. Let $\beta_i$ be the powers of $x_i$ in $g$. Then
	\begin{equation}
	\label{eqn:Spoly1}
		S(f,g)=\frac{x^\gamma}{\LT(f)}f-\frac{x^\gamma}{\LT(g)}g \quad\text{for}\quad\gamma=(\gamma_1,\ldots,\gamma_n),\ \gamma_i=\max(\alpha_i,\beta_i)
	\end{equation}

	Approximately, \textbf{Buchberger's algorithm} starts with a generating set for the ideal, takes each pair of polynomials in the set, calculates the $S$--polynomial for this pair, and then performs the division algorithm of the $S$--polynomial with the generating set. If the remainder is non-zero this is added to the generating set, until all pairs have an $S$--polynomial with zero remainder under the division algorithm.
	The Gr\"obner basis can then be deduced from this. 

	In order to apply Buchberger's algorithm over $\Z$, we replace every instance of division by $\LT(h)$ for a polynomial $h$ to be instead division by $\LM(h)$ and multiply all other terms by $\LC(h)$. 
	For example \Cref{eqn:divAlg1} becomes
	\[
		f\mapsto \LC(f_i) f-\frac{\LT(f)}{\LM(f_i)}f_i
	\]
	and \Cref{eqn:Spoly1} becomes
	\[
		S(f,g)=\frac{\LC(g)x^\gamma}{\LM(f)}f-\frac{\LC(f)x^\gamma}{\LM(g)}g
	\]
	The result will be the same as the previous algorithm up to multiplication by a constant in $k$. 
	In particular the degrees of polynomials remain the same but no terms are ever divided by a coefficient. This makes it easier to compare the results of the algorithm across characteristics, which we will see later. 
\end{remark}
	
	Note when we refer to the Gr\"obner basis, we are typically thinking of the \emph{reduced} Gr\"obner basis, although this should not change any mathematical results. Typically, this is a Gr\"obner basis such that
	\begin{enumerate}
		\item \label{item:redGrob1} Every polynomial in the Gr\"obner basis has leading coefficient 1.
		\item For all elements $p$ of the Gr\"obner basis $G$, no monomial of $p$ lies in $\langle \LT(G\setminus\{p\})\rangle.$
	\end{enumerate}
	When using the modified Buchberger's algorithm described above we waive the requirement for \Cref{item:redGrob1}. Any nonzero polynomial ideal has a unique reduced Gr\"obner basis (see \cite[Proposition 6]{Cox-ideal-2015}), and under the modified algorithm this will be unique up to a common coefficient. 

\begin{lemma}
\label{lem:APolyChar}
	Let $\F$ be an algebraically closed field of characteristic $p$. 
	A generator for the radical of the principal ideal of $V_1$ can be chosen in $\Z_p[m^{\pm1},l^{\pm1}]$. Moreover, $A_{\F}$ only depends on $p$, not $\F$. 
\end{lemma}

\begin{proof}
	Let $\F$ be an algebraically closed field of characteristic $p$. 
	The argument is similar to a result in \cite{Cooper-Plane-1994}, where it is shown we may multiply $A_{\C}$ by a suitable nonzero constant so that the coefficients are integers. We may similarly multiply $A_\F$ by a suitable nonzero constant so that the coefficients are elements in $\Z_p$. 
	
	Consider the polynomial ideals associated with $R_U(M,\F)$. 
	Let $I_{\F} = \langle f_1,\ldots,f_k \rangle$, $f_i \in \Z_p[a_1,\ldots,d_n ]$ be the ideal defining the variety $R_U(M,\F)$ such that
	\begin{equation}
	\label{eqn:ideal}
	\begin{split}
		R_U(M,\F) 	&= V\left(I_{\F}\right)\\
				 	&= \left\{ \left(w_1, \ldots, w_{4n} \right) \in \F^{4n} \mid f_i\left(w_1, \ldots, w_{4n} \right) =0 \text{ for all } 1\leq i \leq k \right\}
	\end{split}
	\end{equation}
	Here, the polynomials $f_i$ are determined by the entries of the relations $R_i$ from \Cref{eqn:upperRepVar} together with the determinant conditions.
	By definition, the closure of the image under the eigenvalue map is $\overline{e(V(I_{\F}))}=V(\langle \tilde{A}_{\F}\rangle)$. 
	
	Elimination theory shows us that for any field $\F$ the polynomial $\tilde{A}_\F$ will appear as an entry in the Gr\"obner bases calculated with respect to lexicographic ordering of $I_\F$ (see \cite{Cox-ideal-2015}). In fact, the polynomial will be the unique entry with only variables $m$ and $l$. 
	Following Buchberger's algorithm, we know we can calculate the Gr\"obner basis with finitely many polynomials. 
	
	We follow Buchberger's algorithm so that all operations are done over $\Z$ as described in \Cref{rem:buchAlg}.
	Following this process, $\tilde{A}_\F$ will appear in the Gr\"obner basis as a the unique entry in $\Z_p[m^{\pm1},l^{\pm1}]$. 
	For all distinct algebraically closed fields $\F$ and $\F'$ both of characteristic $p$, $\tilde{A}_\F$ and $\tilde{A}_{\F'}$ are calculated the same way and will be the same. 
	Then the $A$--polynomials also agree.
	This proves the result.
\end{proof}

We omit the specific field and denote the different $A$--polynomials by characteristic, $\tilde{A}_{p}=\tilde{A}_{\F}$ and $A_{p}=A_{\F}$ for $\F$ a field of characteristic $p$. 
If we require the $A$--polynomial has no common integer factor, then we can consider the polynomial unique up to sign. 

We have a version of the boundary slopes theorem for characteristic $p$, which follows, as in characteristic 0, from \Cref{pro:detected_surface} and the relationship between valuations and boundary slopes of the Newton polygon:

\begin{theorem}
	Let $\F$ be an algebraically closed field of characteristic $p$. 
	The boundary slopes of the Newton polygon for $A_p$ are the strongly detected boundary slopes for $\X(M,\F)$.
\end{theorem} 

Another well-known theorem is the ones in the corners theorem from~\cite{Cooper-polynomial-1997}. This asserts that, if the $A$--polynomial is normalised so that the greatest common divisor of the coefficients is one, the corner vertices in the Newtown polygon of the $A$--polynomial of $M$ have associated coefficient $\pm1$. In this way we can think of the $A$--polynomial as being monic, in the sense that the dominant terms have coefficient $\pm1$. 

Considering fields of positive characteristic, a corollary to the ones in the corners theorem immediately follows. Let $\overline{A_0}=A_0\Mod p$ and more generally for $f$ a polynomial with integer coefficients, let $\overline{f}=f\Mod p$ for a given prime $p$ be the polynomial with coefficients modulo $p$. We will use this notation when $p$ is fixed. Since the coefficients in the corners will remain unchanged modulo $p$, we have:

\begin{corollary}
\label{cor:newtPolyModP}
	 The Newton polygon of $\overline{A_0}=A_0\Mod p$ is the same as the Newton polygon of $A_0$ for all primes $p$.
\end{corollary}

Recall $\tilde{A}_p$ is the polynomial that generates the ideal that defines the eigenvalue variety in characteristic $p$ (whilst ${A}_p$ is a generator of the radical of the ideal defined by $\tilde{A}_p$).

\begin{lemma}
\label{thm:tildeApDividesA0}
	For all but finitely many primes $p$, 
	\begin{equation}
	\label{eqn:tildeAPoly}
		\overline{\tilde{A}_0}=\tilde{A}_p.
	\end{equation}
\end{lemma} 

\begin{proof}
	Let $\F$ be an algebraically closed field of characteristic $p$. 
	By \Cref{lem:APolyChar} it is sufficient to only consider the characteristic of the field. 
	Consider the polynomial ideals associated with $R_U(M,\F)$ in characteristic 0 and characteristic $p$. 
	
	Let $I_0 = \langle f_1,\ldots,f_k \rangle$, $f_i \in \Z[a_1,\ldots,d_n ]$ be the ideal associated with the variety $\R_U(M,\F_0)$ such that
	\begin{equation}
	\label{eqn:I_0}
	\begin{split}
		\R_U(M,\F_0) 	&= V\left(I_0\right)\\
		 			&= \left\{ \left(w_1, \ldots, w_{4n} \right) \in \F_0^{4n} \mid f_i\left(w_1, \ldots, w_{4n} \right) =0\quad \text{ for all } 1\leq i \leq k \right\}
	\end{split}
	\end{equation}
	Here, the polynomials $f_i$ are determined by the relations in \Cref{eqn:upperRepVar}. By definition, the closure of the image under the eigenvalue map is $\overline{e(V(I_0))}=V(\langle \tilde{A}_0\rangle)$. 
	
	Then in characteristic $p$, write $I_p = \langle \overline{f}_1,\ldots,\overline{f}_k \rangle$, $\overline{f}_i \in \Z_p[a_1,\ldots,d_n ]$ for the corresponding ideal $I_0$ modulo $p$. We have
	\begin{equation}
	\label{eqn:I_p}
	\begin{split}
		\R_U(M,\F_p) 	&= V\left(I_p\right) \\
					& = \left\{ \left(w_1, \ldots, w_{4n} \right) \in \F_p^{4n} \mid \overline{f}_i\left(w_1, \ldots, w_{4n} \right) =0\quad \text{ for all } 1\leq i \leq k \right\}
	\end{split}
	\end{equation}
	Again by definition, the closure of the image under the eigenvalue map is $\overline{e(V(I_p))}=V(\langle \tilde{A}_p\rangle)$. 
	
	Elimination theory shows us that the respective polynomials $\tilde{A}_0$ and $\tilde{A}_p$ will appear as an entry in the Gr\"obner bases calculated with respect to lexicographic ordering of $I_0$ and $I_p$ (see \cite{Cox-ideal-2015}). In fact, they will be the unique entry with only variables $m$ and $l$. 
	
	Follow the modified Buchberger's algorithm from \Cref{rem:buchAlg} to calculate the Gr\"obner bases, $G_0$ and $G_p$ respectively, of $I_0$ and $I_p$ over $\Z$ and $\Z_p$ with respect to lexicographic ordering. 
	Suppose that the prime $p$ does not divide any of the leading terms of polynomials that appear in the Gr\"obner basis calculation for $G_0$. Then the Gr\"obner basis $G_0$ modulo $p$, $\overline{G_0}$, is the same as the Gr\"obner basis $G_p$. In this case, we get 
	\begin{equation*}
		\overline{\tilde{A}_0}=\tilde{A}_p
	\end{equation*}
	
	This proves the result for all primes $p$ that do not divide the leading term of a polynomial in Buchberger's algorithm over $\Z$. Since there are only finitely many leading terms, the proof is complete.
\end{proof}

This leads to the following.

\begin{lemma}
\label{thm:ApDividesA0}
	For all but finitely many primes $p$, the $A$--polynomial in characteristic $p$ divides the $A$--polynomial in characteristic 0 modulo $p$. Moreover,
	\begin{equation}
	\label{eqn:ApDividesA0}
		\overline{A_0}(m,l)=g(m,l)A_p(m,l)
	\end{equation}
	where $g \in\Z_p[m,l]$ consists only of polynomial factors that appear in $A_p$.
\end{lemma} 

\begin{proof}	
	This follows directly from \Cref{thm:tildeApDividesA0} and the definition of the $A$--polynomial. Let us break this down further. As in the proof of \Cref{thm:tildeApDividesA0}, consider the finitely many primes $p$ that do not divide any of the leading terms of polynomials that appear in the Gr\"obner basis calculations over $\Z$. 
	 
	The $A$--polynomial $A_0(m,l)$ is the radical of $\tilde{A}_0(m,l)$. Then looking at their irreducible factorisation over $\Z[m,l]$ we find,
	\[
	\begin{split}
		\tilde{A}_0 &=a_0 g_1^{k_1} \ldots g_r^{k_r}, \quad \,a_0\in\Z\setminus\{0\},\ k_i \in\Z_{>0} \\
		{A}_0 &=g_1 \ldots g_r.
	\end{split}
	\]
	for $g_i\in\Z[m,l]$ irreducible and pairwise distinct ($g_i\neq g_j$ for all $i\neq j$). Note \Cref{eqn:tildeAPoly} implies $\overline{a_0}\neq0 \Mod p$.
	
	Consider how this changes modulo $p$. We could find that there is some $i$ such that $\overline{g_i}$ is no longer irreducible, for example, $\overline{g_i}=h_{i_1}\ldots h_{i_s}$ for $h_{i_k} \in \Z[m,l]$ irreducible, or we could find that the collection of $\overline{g_i}$, $i\in\{1,\ldots,r\}$ is no longer distinct, for example, $\overline{g_i}=\overline{g_j}$ for $i\neq j$, or both. Then,
	\[
	\begin{split}
		\overline{\tilde{A}_0} 		&=\overline{a_0}\overline{g_1}^{k_1} \ldots \overline{g_r}^{k_r}, \\ 
							&=\overline{a_0}\overline{g_{\beta_1}}^{b'_1} \ldots \overline{g_{\beta_s}}^{b'_s}h_{1}^{c'_1}\ldots h_{t}^{c'_t}, \quad b'_i,c'_i \in{\Z_p}\setminus\{0\} \\
		\overline{{A}_0} 		&=\overline{g_1} \ldots \overline{g_r}, \\ 
							&=\overline{g_{\beta_1}}^{b_1} \ldots \overline{g_{\beta_s}}^{b_s}h_{1}^{c_1}\ldots h_{t}^{c_t}, \quad b_i,c_i \in{\Z_p}\setminus\{0\}
	\end{split}
	\]
	
	Here, $\{\beta_1,\ldots,\beta_s\}$ is the subset of $\{1,\ldots,r\}$ such that $\overline{g_{\beta_1}}, \ldots, \overline{g_{\beta_s}}$ are irreducible and distinct, and $h_j$ are the collection of irreducible and distinct factors of the remaining $\overline{g_i}$. We have $1\leq b_i\leq b'_i$ and $1\leq c_i\leq c'_i$.
	
	Then using \Cref{eqn:tildeAPoly}, taking the radical of $\overline{\tilde{A}_0}$ and removing common integer factors gives,
	\[
	\begin{split}
		A_p &=\overline{g_{\beta_1}} \ldots \overline{g_{\beta_s}}h_{1}\ldots h_{t}.
	\end{split}
	\]

	Then it is clear that 
	\begin{equation*}
		\overline{A_0}=\overline{g_{\beta_1}}^{b_1-1} \ldots \overline{g_{\beta_s}}^{b_s-1}h_{1}^{c_1-1}\ldots h_{t}^{c_t-1}A_p
	\end{equation*}
	and $g=\overline{g_{\beta_1}}^{b_1-1} \ldots \overline{g_{\beta_s}}^{b_s-1}h_{1}^{c_1-1}\ldots h_{t}^{c_t-1} \in\Z[m,l]$ satisfies the statement. 

	This proves the result for all primes $p$ that do not divide the leading term of a polynomial in Buchberger's algorithm over $\Z$.	
\end{proof}

\begin{remark} 
	\begin{enumerate}
		\item The factor $g(m,l)$ is important and the $A$--polynomial in characteristic $p$ is not necessarily the same as reducing the $A$--polynomial in characteristic $0$ modulo $p$. 
		\item There may be examples where the prime $p$ does divide the leading term of a polynomial in the Gr\"obner basis calculations over $\Z$ and the result still holds.
	\end{enumerate}
\end{remark}
We will see an example where the $A$--polynomial changes under different characteristics in \Cref{sec:m137} (although the detected boundary slopes does not change). \Cref{thm:ApDividesA0} is not applicable in the example because the prime $p$ divides leading terms when calculating the Gr\"obner basis.

\begin{theorem}
	Consider $M$ with a single boundary torus $\bound M=T$. 
	For all but finitely many primes $p$, the boundary curve of an essential surface in $M$ is strongly detected by $A_0$ if and only if it is strongly detected by $A_p$.
\end{theorem}

\begin{proof}
	This follows from \Cref{lem:APolyChar,thm:tildeApDividesA0,thm:ApDividesA0}. For any product of polynomials $f(x,y)=h(x,y)i(x,y)$, the Newton polygon of $f$ is the sum of the Newton polygons of $h$ and $i$. In particular, if neither $h(x,y)$ and $i(x,y)$ are monomials, then the slopes of the Newton polygon of $f$ will consist of the slopes of the Newton polygon of $h$ and the slopes of the Newton polygon of $i$ and, if one of $h(x,y)$ and $i(x,y)$ is a monomial, then the slopes of the Newton polygon of $f$ will consist of the slopes of the Newton polygon of the other polynomial. We have \Cref{eqn:ApDividesA0} from \Cref{thm:ApDividesA0} of this form.
	Recall by \Cref{lem:APolyChar}, the $A$--polynomial depends only on the characteristic of $\F$.
	
	Consider a slope on $\bound M$ that is strongly detected by $\X(M,\F)$. By the boundary slopes theorem in characteristic $p$, the slope corresponds to a boundary slope of the Newton polygon of $A_p(m,l)$. By \Cref{eqn:ApDividesA0} and above, for all but finitely many primes, $p$, this must also be a boundary slope of the Newton polygon of $\overline{A_0}(m,l)$. By \Cref{cor:newtPolyModP}, this is the same as a boundary slope of the Newton polygon of ${A_0}(m,l)$. Then, by \Cref{lem:APolyChar} and the boundary slopes theorem in characteristic $0$, this must be strongly detected by $X(M,\C)$.
		
	Consider a slope on $\bound M$ that is strongly detected by $X(M,\C)$. By the boundary slopes theorem in characteristic $0$, the slope corresponds to a boundary slope of the Newton polygon of $A_0(m,l)$. By \Cref{cor:newtPolyModP}, this is a boundary slope of the Newton polygon of $\overline{A_0(m,l)}$. Then any slope of the Newton polygon of $\overline{A_0(m,l)}$ is either a slope of the Newton polygon of $g(m,l)$ or a slope of the Newton polygon of ${A_p(m,l)}$. But $g(m,l)$ consists only of polynomial factors already present in ${A_p(m,l)}$. Thus the slope must be a slope of the Newton polygon of ${A_p(m,l)}$. By the boundary slopes theorem in characteristic $p$, this must be strongly detected by $X(M,\F)$.
	This proves the statement.
\end{proof}


\subsection{The case of multiple boundary components}

Consider $M$ with $h$ boundary tori $T_1, \ldots, T_h$. We can extend the definitions of the $A$--polynomial and the eigenvalue variety over a general field $\F$ for a manifold with  multiple boundary tori. Assume $\F$ is an algebraically closed field of characteristic $p$. We follow the approach in \cite{Tillus-boundary-2005}. 

Assume we have fundamental group 
\[ 
	\pi_1(M)=\langle \gamma_1, \ldots \gamma_n \mid r_1, \ldots, r_m \rangle
\]
Label the meridian and longitude for each boundary torus $\M_i,\L_i$. This forms a basis of $\pi_1(T_i)$ for each $i=1,\ldots, h$. Consider the representation variety,
\begin{equation}
\begin{split}
	\R(M,\F) 	&= \left\{ (A_1, \ldots, A_n)\in\SLF^n \mid R_1, \ldots, R_m \right \} \subseteq \F^{4n}
\end{split}
\end{equation}
where $A_j=\rho(\gamma_j)=\begin{pmatrix} a_j & b_j \\ c_j & d_j \end{pmatrix}$ and $R_k = r_k(A_1, \ldots, A_n)$.
The entries of $R_k$ are polynomials in $\Z_p[a_1,\ldots,d_n]$; we also have the determinant conditions $a_j d_j-b_j c_j=1$. 

For each $i$, we can identify $\M_i,\L_i$ with generators of $\im(\pi_1(T_i)\to\pi_1(M))$. Then $\M_i$ and $\L_i$ are words in the generators for $\pi_1(M)$. Label the eigenvalues of $\M_i$ and $\L_i$ as $m_i^{\pm1}$ and $l_i^{\pm1}$, respectively, such that $m_i$ and $l_i$ correspond to a common 1--dimensional invariant subspace.

Define the \textbf{eigenvalue map} to be
\begin{equation*}
\begin{split}
	e\colon\R(M,\F)	&\to\left(\F\setminus\{0\}\right)^{2h}\\
	\rho				&\mapsto \left(m_1,l_1,\ldots,m_h,l_h\right)
\end{split}
\end{equation*}
taking representation $\rho$ to eigenvalues of $\rho(\M_i)$ and $\rho(\L_i)$ for $i=1,\ldots,h$. 

We can frame this another way. The representation variety is the affine variety of some ideal $\R(M,\F)=V(J_{\F})$, $J_{\F} = \langle f_1,\ldots,f_k \rangle$, $f_i \in \Z_p[a_1,\ldots,d_n]$. 
For each $\gamma\in\pi_1(M)$, there is an element $I_\gamma\in\Z_p[a_1,\ldots,d_n]$ such that $I_\gamma(\rho)=\tr(\rho(\gamma))$. In particular, this holds for each $\M_i,\L_i$, for all $i =1,\ldots,h.$
Consider the ideal 
\[
\begin{split}
J'_{\F}	&=\langle\, m_i+m_i^{-1}-I_{\M_i},l_i+l_i^{-1}-I_{\L_i}, m_il_i+m_i^{-1}l_i^{-1}-I_{\M_i\L_i} \text{ for all } i=1,\ldots,h\,\rangle \\ 
	&\subseteq\Z_p[a_1,\ldots,d_n,m_1^{\pm1},l_1^{\pm1},\ldots,m_h^{\pm1},l_h^{\pm1}]
\end{split}
\]

We combine ideals $J_{\F}$ and $J'_{\F}$ to get the variety
\begin{equation}
\begin{split}
	\R_E(M,\F) 	&= V(\langle J_{\F},J'_{\F}\rangle) \subseteq \F^{4n}\times(\F\setminus\{0\})^{2h}
\end{split}
\end{equation}

Then the map $e$ can also be seen as the projection
\[
\begin{split}
	e\colon\R_E(M,\F)	&\to\left(\F\setminus\{0\}\right)^{2h}\\
	(a_1,\ldots, d_n, m_1,l_1,\ldots,m_h,l_h) 	&\mapsto \left(m_1,l_1,\ldots,m_h,l_h\right)
\end{split}
\] 

The \textbf{eigenvalue variety} is the closure of the image of this map, $\overline{e(\R_E(M,\F))}$. Again, the proofs in \cite{Tillus-boundary-2005} apply verbatim to show that the dimension of the eigenvalue variety is at most the number of torus boundary components. 

The eigenvalue variety is defined by an ideal $K_{\F}\subseteq\F[m_1,l_1,\ldots,m_h,l_h]$ with $\overline{e(\R_E(M,\F))}=V(K_{\F})$. 
The ideal $K_{\F}$ can be calculated using elimination theory by way of resultants or Gr\"obner bases and is the multiple boundary component analogue to the $A$--polynomial. 
As described in the previous section, the Gr\"obner basis can be calculated using Buchberger's algorithm (or a modified version of this described in \Cref{rem:buchAlg}) using finitely many polynomials.
Once a basis has been chosen for each boundary tori, $K_{\F}$  is well-defined up to multiplication by units in $\F(m_i^{\pm1},l_i^{\pm1})$. 

\begin{lemma}
\label{lem:APolyMultChar}
	Let $\F$ be an algebraically closed field of characteristic $p$. 
	The defining ideal for the eigenvalue variety can be chosen such that the generating set for $K_{\F}$ is inside $ 	\Z_p[m_1^{\pm1},l_1^{\pm1},\ldots,m_h^{\pm1},l_h^{\pm1}]$. 
\end{lemma}

\begin{proof}
	Let $\F$ be an algebraically closed field of characteristic $p$. 
	Similar to the proof in \cite{Cooper-Plane-1994} for a single torus, we may multiply $K_{\C}$ by a suitable nonzero constant so that the coefficients are integers. We may similarly multiply $K_\F$ by a suitable nonzero constant so that the coefficients are elements in $\Z_p$. 

	Consider the polynomial ideals associated with $R_E(M,\F)$, $J_{\F}$ and $J'_{\F}$. 
	Here, the $f_j$ in $J_{\F}$ are determined by the entries of the relations $R_j$ in \Cref{eqn:upperRepVar} together with the determinant conditions and the equations in $J'_{\F}$ are trace relations for the meridian and longitude of each boundary tori. 
	By definition, the closure of the image under the eigenvalue map is $\overline{e(V(I_{\F}))}=V(K_{\F})$. 
	
	Elimination theory shows us that for any field $\F$ the elements in $K_\F$ will appear as entries in the Gr\"obner basis calculated with respect to lexicographic ordering of $J_{\F}$ and $J'_{\F}$ (see \cite{Cox-ideal-2015}). Following Buchberger's algorithm, we know we can calculate the Gr\"obner basis with finitely many polynomials. 
	
	We follow Buchberger's algorithm so that all operations are done over $\Z$ as described in \Cref{rem:buchAlg}. Then $K_\F$ is the resulting Gr\"obner basis. 
	For all distinct algebraically closed fields $\F$ and $\F'$ both of characteristic $p$, $K_\F$ and $K_{\F'}$ are calculated the same way and will be the same. 
	This proves the result.
\end{proof}

Let $\F$ and $\F'$ be distinct algebraically closed fields of characteristic $p$.
The generating sets for $K_\F$ and $K_{\F'}$ can be chosen to be in $\Z_p[m_1^{\pm1},l_1^{\pm1},\ldots,m_h^{\pm1},l_h^{\pm1}]$ but we still consider the ideals in separate rings with $K_\F \subseteq \F[m_1^{\pm1},l_1^{\pm1},\ldots,m_h^{\pm1},l_h^{\pm1}]$ and  $K_{\F'} \subseteq \F'[m_1^{\pm1},l_1^{\pm1},\ldots,m_h^{\pm1},l_h^{\pm1}]$.

We recover the information about boundary slopes by taking the logarithmic limit set of the eigenvalue variety. 
Given an affine variety $V=V(I)$, we list three definitions for the \textbf{logarithmic limit set} following \cite{Bergman-logarithmic-1971}. 

Let $\F[X^{\pm}]=\F[X_1^{\pm1},\ldots,X_m^{\pm1}]$, $\left(x_1,\ldots,x_m \right)\in V$.

\begin{itemize}
	\item An absolute value function on $\F$ is a function $\lvert \cdot \rvert \co \F \to \R_{\geq 0}$ such that 
	$\lvert 0 \rvert = 0$, $\lvert 1 \rvert = 1$, $\lvert xy \rvert = \lvert x \rvert \lvert y \rvert$, $\lvert x+y \rvert \leq \lvert x \rvert + \lvert y \rvert$ and there exists $x \in\F$ such that $0<  \lvert x \rvert <1$.
	If $\F$ has an absolute value function, define $V^{(a)}_\infty(I)$ as the set of limit points on $S^{m-1}$ of points of the form
		\[
		\begin{split}
			\frac{\left(\log\lvert x_1\rvert,\ldots, \log\lvert x_m\rvert\right)}{\sqrt{1+\sum\left(\log\lvert x_i\rvert\right)^2}}.
		\end{split}
		\]
	\item Define $V^{(b)}_\infty(I)$ to be the set of $m$-tuples $\left(-v(X_1),\ldots, -v(X_m)\right)$ as $v$ ranges over all real-valued valuations on $\F[X^{\pm}]/I$ with $\sum v(X_i)^2=1$.
	\item For $f\in\F[X^{\pm1}]$, let the support of $f$ be
		\[
		\begin{split}
			s(f) = 	&\left\{ \alpha=\left( \alpha_1,\ldots,\alpha_m\right) \in\Z^m \; \big| \; x_1^{\alpha_1}\ldots x_m^{\alpha_m} \right.\\
					& \left. \quad\text{ occurs with nonzero coefficient in }f\right\}.
		\end{split}
		\]
		Define $V^{(c)}_\infty(I)$ to be the set of points $\zeta \in S^{m-1}$ such that, for all $0\neq f \in I$, $\underset{\alpha\in s(f)}{\max}\left(\zeta \cdot \alpha \right)$ is assumed at least twice. 
\end{itemize}

The last definition gives another interpretation of the logarithmic limit set as the intersection of the spherical duals of the Newtown polytopes of nonzero elements of $I$. That is, the intersection of all outward pointing unit normal vectors to the support places of the Newton polytopes of nonzero elements of $I$ (see \cite{Tillus-boundary-2005}). 
We write,
\[
	V^{(c)}_\infty(I) = \bigcap_{f\in I} \sphD(\newt(f))
\]
for $\newt(f)$ the Newton polytope of $f$ and $\sphD(P)$ the spherical dual of a polytope $P$. 
Recall, the \textbf{Newton polytope} of a polynomial $f(x_1,\ldots,x_m)$ is the convex hull of the support of $f$. This is a generalisation of the Newton polygon. 

Bergman \cite[Theorem 2]{Bergman-logarithmic-1971} proves that $V^{(b)}_\infty(I) = V^{(c)}_\infty(I)$. We write $V_\infty(I)=V^{(b)}_\infty(I)=V^{(c)}_\infty(I)$. Furthermore, if $\F$ is algebraically closed and equipped with an absolute value function, then $V^{(a)}_\infty(I)$ is a closed subset of $V_\infty(I)$, containing all points therein with rational coordinate ratios. This implies the definition of $V^{(a)}_\infty(I)$ does not depend on the choice of absolute value function and the definition of $V^{(b)}_\infty(I)$ does not depend on the choice of valuation.

A useful tool to understand and compute the logarithmic limit set through the third definition is \emph{tropical bases}.
An introduction to these concepts is in \cite{Maclagan-tropical-2015}.
The tropical basis of an ideal over $\F[x_1^\pm,\ldots,x_n^\pm]$ is the analogue of the \emph{universal} Gr\"obner basis of an ideal over $\F[x_1,\ldots,x_n]$ (which is a Gr\"obner basis over every monomial ordering).
Not every universal Gr\"obner basis is a tropical basis. 

Consider an ideal $I\subseteq \F[x_1^\pm,\ldots,x_n^\pm]$. 
Over $\C$, Bogart et al. \cite{Bogart-computing-2007} prove that $I$ has a finite tropical basis $\{f_1,\ldots,f_k\}$ with the property that
\[
	\bigcap_{f\in I} \sphD(\newt(f)) = \bigcap_{1\leq i \leq k} \sphD(\newt(f_i)).
\]
Maclagan and Sturmfels prove that this result holds over a field $\F$ with non-trivial valuation, 
see \cite[Theorem 2.6.5 and Corollary 3.2.3]{Maclagan-tropical-2015}. 
In \cite[Section 2]{Joswig-degree-2018} it is stated that this works over fields of positive characteristic. This can be obtained by using the results of Maclagan and Sturmfels and observing the results hold for a field with trivial valuation. In particular, an algebraically closed field $\F$ always has a trivial valuation and can be embedded into another valued field with a non-trivial valuation so that the new valuation extends the trivial valuation. For example, consider $\F$ embedded in $\F(t)$ with the valuation defined by $v(t^k)=k$. {Then the results of \cite{Maclagan-tropical-2015} can be applied using the extended field and will give the desired tropical basis over $\F$.} 

Recall for a polynomial $f$ with integer coefficients, $\overline{f}=f\Mod p$ is the polynomial with the same coefficients modulo $p$.  
Consider an ideal $I=\langle g_1,\ldots g_l \rangle  \subseteq \C[x_1,\ldots,x_n]$ with generators $g_i\in\Z[x_1,\ldots,x_n]$. We denote the corresponding ideal over an algebraically closed field $\F$ with characteristic $p$ as $\overline{I}^\F = \langle \overline{g_1},\ldots \overline{g_l} \rangle \subseteq \F[x_1,\ldots,x_n]$. Here, $\overline{g_i}\in\Z_p[x_1,\ldots,x_n]$. Then we can prove the following.

\begin{lemma}
\label{lem:newtPolytopeModP}
	Let $I$ be a finitely generated ideal in $\C[x_1,\ldots,x_n]$ with generators over the integers such that the greatest common divisor of all the coefficients is $1$. 
	Let $\F$ be an algebraically closed field of characteristic $p$. 
	For all but finitely many $p$, the logarithmic limit set of $I$ is the same as the logarithmic limit set of $\overline{I}^\F$,
	$$V_{\infty}(I) = V_{\infty}(\overline{I}^\F).$$
\end{lemma}

\begin{proof}
	Consider $I$ given by
	\[
		I = \langle g_1,\ldots,g_l \rangle \subseteq \C[x_1,\ldots,x_n],\ g_i \in \Z[x_1,\ldots,x_n]
	\]
	and assume the greatest common divisor of the coefficients of $g_i$ for all $i$ is $1$. 
	
	By \cite{Bogart-computing-2007,Maclagan-tropical-2015}, $I$ has a finite tropical basis $\{f_1,\ldots,f_k\}$.  
	This is done using initial forms.
	For $f\in\Z[x_1,\ldots,x_n]$, $\mathbf{w}\in\Z^n$, the \emph{initial form} $\text{in}_\mathbf{w}(f)$ is the terms in $f$ of lowest $\mathbf{w}$-weight (see  \cite{Bogart-computing-2007,Maclagan-tropical-2015} for formal definitions and examples).
	The construction of the tropical basis takes the original generating set $g_i$ and constructs additional polynomials $f$ such that their initial form $\text{in}_\mathbf{w}(f)$ is a unit for specific weight vectors $\mathbf{w}\in\Z^n$. 
	This can be done with operations for coefficients over $\Z$ similar to the modified version of Buchberger's algorithm discussed in \Cref{rem:buchAlg}.
	
	Using the third definition for the logarithmic limits and the existence of a tropical basis, $V_{\infty}(I)$ is precisely the intersection of the spherical duals of the Newton polytopes of elements of the tropical basis for $I$. 
	\[
		V_\infty(I) = \bigcap_{1\leq i \leq k} \sphD(\newt(f_i)).
	\]
	The Newton polytopes of nonzero elements of $I$ will only change modulo $p$ if the coefficient of a leading term of at least one of the elements of the tropical basis is divisible by $p$. 
	Because there are finitely many elements in the basis, there are finitely many primes $p$ that divide coefficients of the leading term. 
	
	Excluding these finitely many primes, the Newton polytopes of nonzero elements of $\overline{I}^\F$ will remain the same as the Newton polytopes of nonzero elements of $I$ and $V_{\infty}(I) = V_{\infty}(\overline{I}^\F)$.
\end{proof}

Following \cite{Tillus-boundary-2005}, consider the logarithmic limit set of the eigenvalue variety, $V_\infty(K_\F)$. For each torus boundary component the eigenvalue variety has symmetries in the choice of eigenvalues $\left(m_i^{\pm1},l_i^{\pm1}\right)$. This leads to symmetries in the logarithmic limit set. If $\left(x_1,\ldots,x_{2i-1},x_{2i},\ldots,x_{2h}\right) \in V_\infty(K)$ then $\left(x_1,\ldots,-x_{2i-1},-x_{2i},\ldots,x_{2h}\right) \in V_\infty(K)$. 

A quotient of the logarithmic limit set can be found by factoring by these symmetries inside $\mathbb{RP}^{2h-1}/\Z_2^{h-1}\cong S^{2h-1}$. Then the quotient map extends to a map $\varphi \co S^{2h-1} \to S^{2h-1}$ of degree $2^h$. Let $N_h$ be the $2h\times2h$ matrix
\[
	N_h=\begin{pmatrix} \begin{matrix} 0 & 1 \\ -1 & 0 \end{matrix} & & \\ & \ddots & \\ & & \begin{matrix} 0 & 1 \\ -1 & 0 \end{matrix} \end{pmatrix}
\]
 with $2\times2$ blocks given along the diagonal and $0$ otherwise. 

The corresponding boundary slopes theorem for multiple boundary tori proceeds loosely as follows. If $\zeta \in V_\infty(K)$ is a point with rational coordinate ratios then there is an essential surface with boundary in $M$ whose projectivised boundary curve coordinate is $\varphi(N_h\zeta)$. Boundary slopes detected this way are again called strongly detected boundary slopes. When $M$ has a single boundary torus, this is precisely the same boundary slopes theorem as stated in \Cref{sec:APolySingle}. 
We have a version of this theorem over multiple boundary tori for all algebraically closed fields $\F$.

\begin{theorem}
\label{thm:KBoundarySlopes}
	Let $\F$ be an algebraically closed field. 
	If $\zeta \in V_\infty(K_\F)$ is a point with rational coordinate ratios then there is an essential surface with boundary in $M$ whose projectivised boundary curve coordinate is $\varphi(N_h\zeta)$.
\end{theorem} 

The proof follows the same steps as that over $\C$ from \cite{Tillus-boundary-2005} also using the aforementioned theory for detecting essential surfaces through ideal points. 

We now compare the defining ideal $K_\F$ across characteristics. 
\begin{proposition}
\label{thm:tildeKpDividesK0}
	Let $\F$ be an algebraically closed field of characteristic $p$. 
	For all but finitely many primes $p$, 
	\begin{equation}
	\label{eqn:tildeK}
		{\overline{K_\C}^\F}={K}_\F.
	\end{equation}
\end{proposition} 

The proof follows the same steps as \Cref{thm:tildeApDividesA0}, applied to the ideals $K_\C$ and $K_\F$, rather than the polynomials $\widetilde{A}_0$ and $\widetilde{A}_p$. This proves the result for an algebraically closed field $\F$ of characteristic $p$ for all $p$ that do not divide the leading term of a polynomial in Buchberger's algorithm over $\Z$. This leads to the following.
\begin{theorem}
\label{thm:KBoundarySlopesP}
	Let M be an orientable, irreducible, compact 3--manifold with non-empty boundary consisting of a disjoint union of $h$ tori.
	For all but finitely many primes $p$, the boundary curve of an essential surface in $M$ is strongly detected by $\X(M, \F)$ for $\F$ an algebraically closed field of characteristic $p$ if and only if it is strongly detected by $\X(M, \C)$.
\end{theorem}

\begin{proof}
	 Let $\F$ be an algebraically closed field of characteristic $p.$
	
	Consider a slope on $\bound M$ that is strongly detected by $\X(M,\F)$. 
	By \Cref{thm:KBoundarySlopes}, there is a point $\zeta\in V_{\infty}(K_\F)$ with rational coordinate ratios associated with the slope with projectivised boundary curve coordinate is $\varphi(N_h\zeta)$.
	By \Cref{thm:tildeKpDividesK0}, for all but finitely many primes $p$, $K_\F = \overline{K_\C}^\F$ and thus $\zeta\in V_{\infty}(\overline{K_\C}^\F)$.
	By \Cref{lem:newtPolytopeModP}, for all but finitely many primes $p$, $V_{\infty}(K_\C) = V_{\infty}(\overline{K_\C}^\F)$ and thus $\zeta\in V_{\infty}({K_\C})$.
	Then, by \Cref{lem:APolyMultChar} and the boundary slopes theorem in characteristic $0$, the slope must also be strongly detected by $\X(M,\C)$.
	
	Consider a slope on $\bound M$ that is strongly detected by $\X(M,\C)$. 
	By \Cref{thm:KBoundarySlopes}, there is a point $\zeta\in V_{\infty}(K_\C)$ with rational coordinate ratios associated with the slope with projectivised boundary curve coordinate is $\varphi(N_h\zeta)$.
	By \Cref{lem:newtPolytopeModP}, for all but finitely many primes $p$, $V_{\infty}(K_\C) = V_{\infty}(\overline{K_\C}^\F)$ and thus $\zeta\in V_{\infty}(\overline{K_\C}^\F)$.
	By \Cref{thm:tildeKpDividesK0}, for all but finitely many primes $p$, $K_\F = \overline{K_\C}^\F$ and thus $\zeta\in V_{\infty}({K_\F})$.
	Then, by \Cref{lem:APolyMultChar} and the boundary slopes theorem in characteristic $p$, the slope must also be strongly detected by $\X(M,\F)$.
	This proves the statement.
	
	Note the finitely many primes $p$ excluded in \Cref{lem:newtPolytopeModP} are those that divide the leading terms in certain elements of the Gr\"obner basis and thus a subset of the finitely many primes excluded by  \Cref{thm:tildeKpDividesK0}.
\end{proof}


\section{A trip to the zoo}
\label{examples}

We survey several applications and examples of Culler-Shalen theory over algebraically closed fields of arbitrary characteristic. 
We start with free products of groups and detecting prime decomposition, and then include three extended examples with different behaviour. The first example ${\tt m137}\rm$ shows a change of the $A$--polynomial; the second example ${\tt m019(3,4)}\rm$ shows a decrease of the dimension of the variety of characters; the third example ${\tt m188(2,3)}\rm$ shows an increase of the dimension of the variety of characters. In particular, we show that no essential surface in ${\tt m019(3,4)}\rm$ is detected by the variety of characters in every characteristic, and an essential surface in ${\tt m188(2,3)}\rm$ is only detected in characteristic 2.

In the examples, we let $\F$ be an algebraically closed field of characteristic $p$ and let $\X_p=\X^\text{irr}(M,\F)$ denote the components of the variety of characters that contain irreducible representations. 


\subsection{Free products of groups and detecting prime decomposition}
\label{subsec:direct_and_prime}

We can characterise a certain type of group that admits curves in the variety of characters in positive characteristic: free products of groups that admit non-central representations. 

Let $Z(H)$ be the centre of a group $H$ and consider a group $G$. We say $G$ \textbf{admits a non-central representation} into $H$ if there exists a representation $\rho:G\to H$ such that $\im(\rho)$ is not central. That is, there exists $\gamma \in G$ such that $\rho(\gamma)\notin Z(H)$.

\begin{lemma}[Free product of groups]
\label{lem:directProduct}
	Let $\F$ be an algebraically closed field and
	let $G_i$, $i=0,\ldots,k$ be groups that admit non-central representations into $\SLF$. Then the variety of characters $X(G_0 \ast \ldots \ast G_k,\F)$ contains a curve with the property each character on the curve restricted to $G_i$ is constant for all $i$. 
	Moreover, the ideal points of the curve detect the splitting. 
\end{lemma}

\begin{proof}
Consider non-central representations $\rho_i \co G_i \to \SLF$ for each $i=0,\ldots,k$. That is, there exists $\gamma_i \in G_i$ such that $\rho_i(\gamma_i)\notin Z(\SLF)=\left\{\begin{pmatrix} \pm1 & 0 \\ 0 & \pm 1\end{pmatrix}\right\}$. Then $\rho_i(\gamma_i)$ has either a unique $1$--dimensional eigenspace or two distinct $1$--dimensional eigenspaces. Hence we can conjugate the representations so that no pair $\rho_i(\gamma_i)$, $\rho_j(\gamma_j)$, $i\neq j$, has a common eigenspace. Moreover,. we may assume that none of the 1--dimensional eigenspaces of each $\rho_i(\gamma_i)$ is spanned by $(1,0)^\intercal$ or $(0,1)^\intercal.$ In particular, writing
\[
\rho_i(\gamma_i) = \begin{pmatrix} a_i & b_i \\ c_i & d_i \end{pmatrix}
\]
we have $b_i  \neq 0 \neq c_i.$

We now construct a parameterised curve $\{\, \rho_t \,\mid\, t \in \F\setminus \{0\} \, \}\subseteq \R(G_0 \ast \ldots \ast G_k,\F)$ with a natural valuation defined by $\nu(t^l) = l$ that detects the splitting.

Let $t \in \F\setminus \{0\}$. Construct a representation $\rho_t \co G_0 \ast \ldots \ast G_k \to \SLF$ by setting
\[
	\rho_t(\alpha_j) = \begin{pmatrix} t^{-j} & 0 \\ 0 & t^{j} \end{pmatrix} \rho_j(\alpha_j) \begin{pmatrix} t^{j} & 0 \\ 0 & t^{-j} \end{pmatrix}
\]
for each $\alpha_j \in G_j$ and extending over $G_0 \ast \ldots \ast G_k$. Note that since $\rho_j(\gamma_j)$ did not have a 1--dimensional eigenspace spanned by $(1,0)^\intercal$ or $(0,1)^\intercal,$ the same is true for $\rho_t(\gamma_j).$

For every $\alpha_j \in G_j$, we have $\tr(\rho_t(\alpha_j))=\tr(\rho_j(\alpha_j))$, which is constant in $t$.

For every pair $\gamma_i \in G_i$,  $\gamma_j \in G_j$, $i\neq j$, 
\[
\begin{split}
	\rho_t(\gamma_i\gamma_j) 	&=  \begin{pmatrix} \phantom{t^{2i}}a_i & t^{-2i}b_i \\ t^{2i} c_i & \phantom{t^{-2i}}d_i \end{pmatrix}
								\begin{pmatrix} \phantom{t^{2i}}a_j & t^{-2j}b_j\\ t^{2j} c_j & \phantom{t^{-2i}}d_j \end{pmatrix}  \\
							&=  \begin{pmatrix} a_ia_j+t^{2(j-i)}b_ic_j & \star \\ \star & t^{2(i-j)}c_ib_j +d_id_j \end{pmatrix}
\end{split}	
\]
with
\[
	\tr\left(\rho_t(\gamma_i\gamma_j) \right) = a_ia_j+t^{2(j-i)}b_ic_j + t^{2(i-j)}c_ib_j +d_id_j
\]
Since $b_ic_j \neq 0 \neq c_ib_j,$ we have $\nu(\tr\left(\rho_t(\gamma_i\gamma_j) \right)) <0$ for each such pair $i,j\in\{0,\ldots, k\}$ with $i \neq j.$ This then gives the desired 1--parameter family of representations. It is indeed a curve, as it can be viewed as the image under an algebraic map of the variety $\{\, (t, s) \in \F^2 \mid ts = 1\,\} \cong \F\setminus \{0\}.$
\end{proof}

We can consider the special case of cyclic groups. 

\begin{corollary}[Free product of cyclic groups]
\label{lem:directProductCyclic}
Let $\F$ be an algebraically closed field of characteristic $p$.

Let $n_0, \ldots, n_k \ge 2$ be integers and $C_{n}$ the cyclic group order $n$.
Then for $p=2$, $X(C_{n_0} \ast \ldots \ast C_{n_k},\F)$ contains a curve that detects the splitting.

If $n_0, \ldots, n_k \ge 3$, then for every prime $p$, $X(C_{n_0} \ast \ldots \ast C_{n_k},\F)$ contains a curve that detects the splitting. 
\end{corollary}

\begin{proof}
	This follows from \Cref{lem:JNF_in_SLk,lem:directProduct}. For all $n \ge 2$, when $p=2$, $C_n$ admits a non-central representation into $\SL_2(\F)$ and the first result for $X(C_{n_0} \ast \ldots \ast C_{n_k},\F)$ when $p=2$ follows. For all $n \ge 3$ and primes $p$, $C_n$ admits a non-central representation into $\SL_2(\F)$ and the second result for $X(C_{n_0} \ast \ldots \ast C_{n_k},\F)$ follows. 
\end{proof}

We can explain further why this difference occurs between characteristic $2$ and characteristic $p\geq3$.

\begin{remark}	
	Consider characteristic $p\ge 3$. If there exists $i$ such that $n_i=2$ the obstacle that arises is that $C_{n_i}=C_2$ admits no non-central representations in $\SL_2(\F)$. The only element of order two in $\SL_2(\F)$ is the central element $-\idMat.$ Hence any homomorphism $C_{n_0} \ast \ldots \ast C_{n_k} \to \SL_2(\F)$ maps the generator of $C_{n_i}$ to $\pm \idMat,$ and the generator of the other factors $C_{n_j}$ to an element of order dividing $n_j$. In this case, there could still be a curve in $X(C_{n_0} \ast \ldots \ast C_{n_k},\F)$, but it would not detect the splitting. \qed
\end{remark}

In the interests of space we will not completely describe the variety of characters $X(C_n \ast C_m,\F)$.

We have the following immediate consequence to \Cref{lem:directProductCyclic}.

\begin{corollary}[Sufficient criterion for curves]
\label{cor:sufficient_criterion_curve}
Let $\F$ be an algebraically closed field of characteristic $p$.
If the finitely generated group $G$ admits an epimorphism to $G_{0} \ast \ldots \ast G_{k}$, where $G_i$ admits a non-central representation into $\SL_2(\F)$, then $X(G,\F)$ contains a curve that detects the splitting for all primes $p$. 
\end{corollary}


The above results combined with \Cref{pro:CSsplittingsurface} allow us to detect the prime decomposition of certain 3--manifolds. 

\begin{corollary}
\label{cor:primeDecomp}
Let $\F$ be an algebraically closed field of characteristic $p$.

Suppose $M$ has prime decomposition $M=M_0\#\ldots\#M_k$ along the union of essential spheres $S=S_1\cup\ldots\cup S_k$. If $\pi_1(M_i)$ admits a non-central representation into $\SL_2(\F)$ for all $i = 0,\ldots k$, then $S$ is detected by an ideal point of a curve in $X(M,\F)$.

In particular, if $\pi_1(M_i)$ admits an epimorphism into $C_{n_i}$ the cyclic group of order ${n_i} \ge 2$ for all $i= 0,\ldots, k$, then $S$ is detected by an ideal point of a curve in $X(M,\F)$ if $p=2$.
If ${n_i} \ge 3$ for all $i = 0,\ldots, k$, then $S$ is detected by an ideal point of a curve in $X(M,\F)$ for all primes $p$.
\end{corollary}

We can get some interesting examples for Lens spaces $L(q,r)$.
\begin{example}[Connected sum of lens spaces.]
Let $\F$ be an algebraically closed field of characteristic $p$.
Consider $M = L(q, r) \# L(2, 1)$ for $q\ge 3$. 
Then $\pi_1(M)\cong C_q\ast C_2$. 

By \Cref{cor:primeDecomp}, the essential sphere in $M$ is detected by $X(M,\F)$ for $p=2$. In characteristic $p\neq 2$, any homomorphism $C_q\ast C_2 \to \SL_2(\F)$ maps the generator of the first factor to an element of order dividing $q$ and the second factor to $\pm\idMat$. Since there are at most finitely many conjugacy classes of elements of any given finite order, this implies that $X(M,\F)$ has dimension zero. 
This shows the essential sphere in $M$ is only detected by $X(M,\F)$ when $p=2$ and we cannot detect the prime decomposition in every characteristic $p$. 
\end{example}

A natural question is then the following. 
\begin{question}
\label{q:primeDecomp}
	Can we always detect the prime decomposition in $\F$ an algebraically closed field of characteristic $2$?
\end{question}

\begin{remark}
Every prime manifold with non-trivial first homology group will admit a non-central representation into $\SL_2(\F)$ for $\F$ an algebraically closed field of characteristic $2$.  Thus, the prime decomposition is always detected in characteristic $2$ if each prime manifold in the decomposition has non-trivial first homology group.
\end{remark}

A positive answer to the next question gives a positive answer to the previous.
\begin{question}
	Does every 3--manifold (other than the 3--sphere) with trivial homology admit a non-central representation into $\SL_2(\F)$ for $\F$ an algebraically closed field of characteristic $2$? 
\end{question}

\begin{example}[Connected sum of Poincar\'e homology spheres.]

Consider the Poincar\'e homology sphere, $PH$. We have
\[
\begin{split}
	\pi_1(PH) 	&=\left\langle a,b\mid a^5=abab=b^3 \right\rangle
\end{split}
\]

This admits a non-central representation into $\SL_2(\F)$ for $\F$ a field of characteristic $p$ conjugate to the form
\[
\begin{split}
	\rho(a) =\begin{pmatrix} x & 1 \\ 0 & x^{-1} \end{pmatrix}, \quad
	\rho(b) =\begin{pmatrix} y & 0 \\ -y(x-x^{-1})-x^{-1} & 1-y \end{pmatrix},\quad x, y \in \F
\end{split}
\]
with
\[
\begin{split}
	1-y+y^2		  	&=0\\
	1-x+x^2-x^3+x^4 	&=0
\end{split}
\]

Take $M = PH \# \ldots \# PH$ the connected sum of some number of Poincar\'e homology spheres. By \Cref{cor:primeDecomp} the union of essential spheres in $M$ is detected by an ideal point of a curve in $X(M,\F)$ for all $p$.
Similarly for the connected sum of $PH$ with other manifolds $M_i$ that admit non-central representations, $M = PH \# M_1 \# \ldots \# M_n$. 
\end{example}


\subsection{Extended example: ${\tt m137}\rm$}
\label{sec:m137}

Take $M$ to be the manifold ${\tt m137}\rm$ in the cusped census of ${\tt SnapPea}\rm$. In this section we compare the variety of characters and $A$--polynomial for $M$ in characteristics 0, 2 and 3. 

Our manifold $M$ is hyperbolic, one-cusped, and is the complement of a knot in $S^2\times S^1$. It can be obtained by 0 Dehn surgery on either component of the link $7^2_1$ in $S^3$. ${\tt SnapPea}\rm$ computes the following fundamental group and peripheral system,
\[
	\pi_1(M)=\left\langle \alpha,\beta \mid \alpha^3\beta^2\alpha^{-1}\beta^{-3}\alpha^{-1}\beta^2\right\rangle,\quad
	\{\M,\L\}=\left\{\alpha^{-1}\beta^{-1},\alpha^{-1}\beta^2\alpha^4\beta^2\right\}
\]

We can adjust the presentation to get 
\[
\begin{split}
	\pie(M)	&=\left\langle \M,\beta \mid \beta^{-1}\M^{-1}\beta^{-1}\M^{-1}\beta^2\M=\M\beta^{-2}\M^{-1}\beta^2\right\rangle\\
	\{\M,\L\}	&=\left\{\M,\beta^2\M^{-1}\beta^{-3}\M^{-1}\beta^2\right\}
\end{split}
\]

Consider the components of the variety of characters containing irreducible representations. Up to conjugation, irreducible representations $\rho\co\pi_1(M)\to\SLF$ have the form
\[
	\rho(\M)=\begin{pmatrix} m & 1 \\ 0 & m^{-1} \end{pmatrix},\quad
	\rho(\beta)=\begin{pmatrix} b & 0 \\ u & b^{-1} \end{pmatrix}
\]
for $u\neq0,-(m-m^{-1})(b-b^{-1})$. 

We will see the only non-trivial coefficients that appear in the characteristic 0 case are multiples of 2 and 3 and thus those characteristics are the only situation where the terms vary. Assessing the group relation we find in all characteristics,
\begin{equation*}
\begin{split}
	u 			&= -\frac{(1-m^3)(1-mb^4)+b^2(1-m+m^2-m^3+m^4)}{m(1+m+m^2)b(1+b^2)}\\
	0 = f(m,b) 		&= (1-2m^3+m^6)(1+b^8)\\
				&\quad+(3-m+m^2-6m^3+m^4-m^5+3m^6)(b^2+b^6)\\
				&\quad+(4-2m+2m^2-9m^3+2m^4-2m^5+4m^6)b^4
\end{split}
\end{equation*}

As there are two generators, the variety of characters is generated by three trace coordinates, $\tr(\rho(\M))$, $\tr(\rho(\beta))$, and $\tr(\rho(\M\beta))$. Writing $s=\tr(\rho(\mathcal{M}))=m+m^{-1}$, $t=\tr(\rho(\beta))=b+b^{-1}$, and $r=\tr(\rho(\mathcal{M}\beta))$ gives in all characteristics
\[
\begin{split}
	g(s,t)		&=(4+4s-s^2-s^3)t^2-(2+3s-s^3)t^4-1 =0\\
	h(s,t,r)	&=r(s+1)t-(s+1)t^2-1 = 0
\end{split}
\] 

The respective irreducible components of the variety of characters are given by
\begin{equation}
\begin{split}
	X_p &= \left\{(s,t,r)\in \F_p^3 \mid g(s,t)=0, h(s,t,r)=0\right\} \\
	&\cong \left\{(s,t)\in \F_p^2 \mid g(s,t)=0 \right\}
\end{split}
\end{equation}
In this example, the equations in characteristic 2 and 3 can be obtained from the equations in characteristic 0 by simply reducing the integer coefficients modulo 2 or modulo 3 respectively. This is because in the calculation of $X_p$, no terms are divided by $2$ or $3$. 

We now turn to the $A$--polynomial, where such a reduction is not possible. Assessing the matrix associated with the longitude $\mathcal{L}$ in each characteristic and using Gr\"obner bases in ${\tt Singular}\rm$ produces 
\[
\begin{split}
	{A}_0(m,l) = \tilde{A}_0(m,l)	&= m^4+2m^5+3m^6+m^7-m^8-3m^9-2m^{10}-m^{11}\\
							&\hspace{-1cm}+l^2\left(-1-3m-2m^2-m^3+2m^4+4m^5+m^6+4m^7\right.\\
							&\left.+m^8+4m^9+2m^{10}-m^{11}-2m^{12}-3m^{13}-m^{14}\right)\\
							&\hspace{-1cm}+l^4\left(-m^3-2m^4-3m^5-m^6+m^7+3m^8+2m^9+m^{10}\right), \\ 
	{A}_2(m,l) = \tilde{A}_2(m,l) 	& =m^4+m^5+m^7+m^9+m^{10}\\
							&\hspace{-1cm}+l^2\left(1+m^3+m^4+m^5+m^8+m^9+m^{10}+m^{13}\right)\\
							&\hspace{-1cm}+l^4\left(m^3+m^4+m^6+m^8+m^9\right)\\
	{A}_3(m,l) =\tilde{A}_3(m,l)	&= m^4-m^5+m^7-m^8+m^{10}-m^{11}\\
							&\hspace{-1cm}+l^2\left(-1+m^2-m^3-m^4+m^5+m^6+m^7\right. \\
							&\left. +m^8+m^9-m^{10}-m^{11}+m^{12}-m^{14}\right)\\
							&\hspace{-1cm}+l^4\left(-m^3+m^4-m^6+m^7-m^9+m^{10}\right)
\end{split}
\]

If we reduce $A_0$ modulo 2, the result gives $A_2$ times a factor $(m-1)$. There is a chance that in characteristic 2, we have a second component containing irreducible representations where we map:
\[
	\rho(\mathcal{M})=\begin{pmatrix} 1 & 1 \\ 0 & 1 \end{pmatrix},\quad
	\rho(\mathcal{\beta})=\begin{pmatrix} b & 0 \\ u & b^{-1} \end{pmatrix}
\]
However, there is no representation conjugate into this form. Hence $A_2$ is obtained from $A_0$ by reducing modulo 2 and deleting a factor of the resulting polynomial. In characteristic 3 the result is more straightforward, and $A_3$ is obtained from $A_0$ by reducing modulo 3. 

Although there is a change in the polynomial in characteristic $2$, we see there is no change in the strongly detected boundary slopes as the slopes of the Newton polygons remain the same across all characteristics.


\subsection{Extended example: ${\tt m019(3,4)}\rm$}
\label{sec:m019}

We focus on an example of a closed manifold where the dimension of the variety of characters decreases in characteristic $2$. Let us take $M$ to be the manifold ${\tt m019(3,4)}\rm$.  Our manifold $M$ is closed and hyperbolic and appears in the census of~\cite{Burton-computing-2018} of closed Haken hyperbolic 3--manifolds.

${\tt Regina}\rm$ computes the following fundamental group,
\[
	\pie(M)=\left\langle \alpha,\beta \mid \alpha \beta^2 \alpha \beta^2 \alpha \beta^{-1} \alpha^3 \beta^{-1}, \alpha^3 \beta^{-3} \alpha^3 \beta^{-1} \alpha^4 \beta^{-1} \right\rangle
\]
Note that $H_1(M) \cong \Z_{40}$ is finite cyclic, and hence there are only finitely many reducible characters.

We will show that in characteristic $p\neq2$, the components of the variety of characters that contain irreducible representations consist of points, and in characteristic $2$, there are no irreducible representations. That is, 
\[
	\dim(X^\text{irr}_p(M))=0,\quad\dim(X^\text{irr}_2(M))=-1
\]
In particular, this means that no essential surfaces in $M$ are detected by the variety of characters in any characteristic. 

Consider the components of the variety of characters containing irreducible representations. Up to conjugation, irreducible representations $\rho:\pi_1(M)\to\SLF$ have the form
\begin{equation}
\label{eqn:m019rep}
	\rho(\alpha)=\begin{pmatrix} a & 1 \\ 0 & a^{-1} \end{pmatrix},\quad
	\rho(\beta)=\begin{pmatrix} b & 0 \\ u & b^{-1} \end{pmatrix}
\end{equation}
for $u\neq0,-(a-a^{-1})(b-b^{-1})$. 

Substitute the irreducible representation \Cref{eqn:m019rep} into the group relations. In characteristic $2$, we find no such representations exist as the expressions either force a contradiction, with $a=0$ or $b=0$, or force that $u=0$ and the representation is reducible. 

In characteristic $p\neq2$,
\begin{equation*}
\begin{split}
	u 		&= \frac{-a^2 + a^4 + b^2 + a^4 b^2 + b^4 - a^2 b^4}{a (1 + a^2) b(1 + b^2)}\\
	0 = f(a,b) 	&= (1 + a^4)(1+b^8) + 2(1 + a^4) (b^2+b^6) + 2(2 + a^2 + 2a^4) b^4 \\
	0 = g(a,b) 	&= (1 + a^2 + a^4)(1+b^2) - a^2 b \\ 
	0 = h(a,b) 	&= a^2(1+b^6) + (1 - a^2 + a^4) (b+b^5) + 2 a^2 (b^2-b^3+b^4)
\end{split}
\end{equation*}
Note that the equations given for $u$, $f(a,b)$, $g(a,b)$, $h(a,b)$ modulo $2$ do not alone satisfy the group relations in characteristic $2$.

As there are two generators, the variety of characters is generated by three trace coordinates, $\tr(\rho(\alpha))$, $\tr(\rho(\beta))$, and $\tr(\rho(\alpha\beta))$. Writing $s=\tr(\rho(\alpha))=a+a^{-1}$, $t=\tr(\rho(\beta))=b+b^{-1}$, and $r=\tr(\rho(\alpha\beta))$ gives in characteristic $p\neq2$,
\[
\begin{split}
	k(s,t,r)	&=str-t^2(s^2-2)-2 =0\\
	l(s,t)		&=(s^2-1)t-1 = 0\\
	m(s,t)	&=t^4(s^2-2)-2t^2(s^2-2)+2(s^2-1)=0\\
	n(s,t)		&=t^3+t^2(s^2-3)-t-2(s^2-2) = 0
\end{split}
\] 
which solves to 
\[
\begin{split}
	r		&=\frac{s(2s^2-3)}{s^2-1}\\
	t		&=\frac{1}{s^2-1}\\
	q(s)		&=2s^8-10s^6+18s^4-12s^2+1 = 0
\end{split}
\] 

Solving the system of equations we find the the respective irreducible components of the variety of characters for $p\neq2$ are given by
\begin{equation}
\begin{split}
	X_p 	&= \left\{(s,t,r)\in \F_p^3 \mid k(s,t,r)=0, l(s,t)=0, m(s,t)=0, n(s,t)=0\right\} \\
		&\cong \left\{s\in \F_p \mid q(s)=0 \right\}.
\end{split}
\end{equation}
Each version of $q(s)=0$ modulo $p\neq 2$ has eight distinct solutions over $\F_p$, so each $X_p$ contains eight points. 

Hence $M$ is an example of a hyperbolic, Haken, 3--manifold with none of its essential surfaces detected by the variety of characters in any characteristic. 


\subsection{Extended example: ${\tt m188(2,3)}\rm$}
\label{sec:m188}

We now focus on an example of a closed manifold where the dimension of the variety of characters increases in characteristic $2$. Take $M$ to be the manifold ${\tt m188(2,3)}\rm$.  Our manifold $M$ is closed, Haken, and hyperbolic. As in the previous example, it is known to be Haken because it appears in the  census of~\cite{Burton-computing-2018}. We show that in characteristic $p\neq 2$, the components of the variety of characters that contain irreducible representations consist of finitely many points, and in characteristic $2$, there is a curve. In particular,
\[
	\dim(X^\text{irr}_0(M))=0,\quad\dim(X^\text{irr}_2(M))=1
\]
In particular, this gives an independent proof of the Haken property.

${\tt Regina}\rm$ computes the following fundamental group,
\[
	\pie(M)=\left\langle
	\alpha^3 \beta^{-1} \alpha \beta^2 \alpha \beta^2 \alpha \beta^{-1},\alpha^{-2} \beta^3 \alpha^{-2} \beta^3 \alpha \beta^{-1} \alpha \beta^3 \right\rangle
\]
Note that $H_1(M) \cong \Z_{2}\oplus \Z_{26}$ is finite cyclic, and hence there are only finitely many reducible characters.

Consider the components of the variety of characters containing irreducible representations. Up to conjugation, irreducible representations $\rho:\pi_1(M)\to\SLF$ have the form
\begin{equation}
\label{eqn:m188rep}
	\rho(\alpha)=\begin{pmatrix} a & 1 \\ 0 & a^{-1} \end{pmatrix},\quad
	\rho(\beta)=\begin{pmatrix} b & 0 \\ u & b^{-1} \end{pmatrix}
\end{equation}
for $u\neq0,-(a-a^{-1})(b-b^{-1})$. 

Substitute the irreducible representation \Cref{eqn:m188rep} into the group relations. 

In characteristic 2 either, 
		\begin{align*}
			u_2 &= a b (1 + b)^8\\
			0 = f_2(b) &= 1 + a^2 \\
			0 = g_2(b) &= 1 + b + b^2 + b^3 + b^4 \\
			\intertext{ or }
			u_2 &= \frac{1 + a^2}{ab}\\
			0 = h_2(b) &= 1 + b^2
		\end{align*}
		
In characteristic $p\neq2$,
		\begin{equation*}
		\begin{split}
			u_p &= \frac{-a^2 + a^4 + b^2 + a^4 b^2 + b^4 - a^2 b^4}{a (1 + a^2) b(1 + b^2)}\\
			0 = f_p(a,b) &= (1 + a^4)(1+b^8) + 2(1 + a^4) (b^2+b^6) + 2(2 + a^2 + 2a^4) b^4 \\
			0 = g_p(b) &= (1 + b^4) (1 + 3 b^2 + 5 b^4 + 3 b^6 + b^8)
		\end{split}
		\end{equation*}

If we took the solution in characteristic $p\neq 2$ and reduced modulo $2$, we would get $u_p=0$, contradicting irreducibility of the representation.
	
As there are two generators, the variety of characters is generated by three trace coordinates, $\tr(\rho(\alpha))$, $\tr(\rho(\beta))$, and $\tr(\rho(\alpha\beta))$. Write $s=\tr(\rho(\alpha))=a+a^{-1}$, $t=\tr(\rho(\beta))=b+b^{-1}$, and $r=\tr(\rho(\alpha\beta))$.

In characteristic 2 either,
		\begin{align*}
			s		& =0\\
			r		& =0\\
			m_2(t)	&=t^2+t+1=0 \\
		\intertext{or }
			t		& =0\\
			r		& =0
		\end{align*}

In characteristic $p\neq2$,
		\[
		\begin{split}
			k_p(s,t,r)	&=str-t^2(s^2-2)-2 = 0\\
			l_p(s,t)	&=-2 (t^2-1)^2 + s^2 ( t^4-2t^2+2) = 0\\
			m_p(t)	&=(t^2 -2) ( t^4-t^2+1)=0
		\end{split}
		\] 

Then the respective irreducible components of each of the variety of characters are given by the following.  

In characteristic 2,
		\begin{equation*}
		\begin{split}
			X_2 	&= \left\{(s,t,r)\in \F_2^3 \mid s=0, r=0, m_2(t)=0\right\}\cup \left\{(s,t,r)\in \F_2^3 \mid t=0, r=0\right\}\\
				&\cong \left\{t\in \F_2 \mid m_2(t)=0 \right\}\cup \left\{s\in \F_2 \right\}
		\end{split}
		\end{equation*}
The first component has two points, the second forms a line.

In characteristic $p\neq2$,
		\begin{equation*}
		\begin{split}
			X_p 	&= \left\{(s,t,r)\in \F_p^3 \mid k_p(s,t,r)=0, l_p(s,t)=0, m_p(t)=0\right\} \\
				&\cong \left\{(s,t)\in \F_p^2 \mid l_p(s,t)=0,m_p(t)=0 \right\}
		\end{split}
		\end{equation*}
This has eight points in characteristic $3$ and twelve points otherwise. The change in number of points in characteristic 3 is a result of the factorisation of $m_3(t)$. In particular, this example exhibits two ramification phenomena: we have a jump in dimension in characteristic 2; considering all Zariski components (or only the 0--dimensional Zariski components), we have a drop in their number in characteristics  2 and 3. 


\bibliography{references}
\bibliographystyle{plain}


\address{Grace S. Garden\\School of Mathematics and Statistics F07, The University of Sydney, NSW 2006 Australia\\{grace.garden@sydney.edu.au\\-----}}

\address{Stephan Tillmann\\School of Mathematics and Statistics F07, The University of Sydney, NSW 2006 Australia\\{stephan.tillmann@sydney.edu.au}}


\Addresses                                       
\end{document}